\documentclass[hidelinks,12pt,reqno]{amsart}

\usepackage[T1]{fontenc}
\usepackage{crimson}
\usepackage[dvipsnames]{xcolor}
\usepackage{xcolor,amssymb,enumerate,tikz-cd,hyperref,amsmath,mathrsfs,mathtools,enumitem,changepage,
mathdots,eucal,xspace,amsthm,parskip,cancel,
amsfonts,fge,xpatch,array,pgfplots,mathabx,stackrel,stmaryrd,makecell,tabularx}
\usepackage[table]{xcolor}
\usetikzlibrary{3d,patterns}
\usepackage{contour}
\contourlength{3pt}

\newcolumntype{C}{>{\centering\arraybackslash}X}

\pgfplotsset{compat=1.18}

\usepackage{scalerel}[2016/12/29]
\usepackage{xparse,subcaption,hyperref,cleveref}

\allowdisplaybreaks

\usepackage[margin=1in]{geometry}

\usepackage[backend=biber,style=alphabetic,doi=false,isbn=false,url=false,maxbibnames=99,maxalphanames=99,minalphanames=1]{biblatex}

\AtBeginBibliography{\footnotesize}

\newtoggle{biblink}

\renewbibmacro*{begentry}{%
  \iffieldundef{doi}
    {\iffieldundef{url}
      {\global\togglefalse{biblink}}
      {\global\toggletrue{biblink}%
       \edef\@biburl{\thefield{url}}%
       \pdfstartlink attr{/C [0 0 1] /Border [0 0 0]} user{/Subtype /Link /A <</S /URI /URI (\@biburl)>>}%
       \bgroup}}
    {\global\toggletrue{biblink}%
     \edef\@biburl{https://doi.org/\thefield{doi}}%
     \pdfstartlink attr{/C [0 0 1] /Border [0 0 0]} user{/Subtype /Link /A <</S /URI /URI (\@biburl)>>}%
     \bgroup}%
}

\renewbibmacro*{finentry}{%
  \iftoggle{biblink}{\egroup\pdfendlink}{}%
  \finentry
}

\newtheorem{thmX}{Theorem}
\newtheorem{corX}[thmX]{Corollary}

\theoremstyle{definition}
\newtheorem{defn}[subsubsection]{Definition}
\newtheorem*{defn*}{Definition}

\newtheorem*{idea*}{Definition}

\newtheorem{ass*}{Assumption}
\newtheorem{rem}[subsubsection]{Remark}
\newtheorem*{rem*}{Remark}

\newtheorem*{example*}{Example}

\newtheorem*{construction*}{Construction}

\newtheorem*{varnt*}{Variant}

\newtheorem*{notat*}{Notation}

\newtheorem*{warning*}{Warning}

\newtheorem*{ques*}{Question}

\newtheorem*{attempt*}{Attempt}

\newtheorem*{lemdefn*}{Lemma-Definition}

\newtheorem{prop}[subsubsection]{Proposition}
\newtheorem*{prop*}{Proposition}
\newtheorem{theorem}[subsubsection]{Theorem}
\newtheorem*{theorem*}{Theorem}
\newtheorem{lem}[subsubsection]{Lemma}
\newtheorem*{lem*}{Lemma}
\newtheorem{cor}[subsubsection]{Corollary}
\newtheorem*{cor*}{Corollary}

\newtheorem*{conj*}{Conjecture}

\usetikzlibrary{decorations.markings}
\newcommand{\myarrowlength}{10pt}
\tikzset{mytip/.tip={Butt Cap[black, length=\myarrowlength, sep=-1.6pt]>[black]},
    std/.style={white, text=black, #1, decoration={transform={xshift=.5*\myarrowlength}, markings, mark=at position .5 with {\arrow{mytip}}}, postaction=decorate},
    myarrow/.default={}}

\usetikzlibrary{decorations.pathmorphing,shapes}
\usetikzlibrary{arrows.meta}
\tikzcdset{arrow style=tikz,
    squigarrow/.style={
        decoration={
        snake, 
        amplitude=.4mm,
        segment length=2mm
        }, 
        rounded corners=.2pt,
        decorate
        }
    }

\tikzset{
  white ultra thick/.style={
    ultra thick,
    preaction={draw,white,line width=#1}
  }
}
\tikzset{
  white thick/.style={
    thick,
    preaction={draw,white,line width=#1}
  }
}

\newcommand{\rightloop}[1][]{%
  \begin{tikzpicture}[baseline={(0,-0.1)}]
    \draw[->, in=30, out=-30, looseness=6] (0,-0.1) to (0,0.1);
  \end{tikzpicture}%
}

\newcommand{\leftloop}[1][]{%
  \begin{tikzpicture}[baseline={(0,-0.1)}]
    \draw[->, in=-30+180, out=30+180, looseness=6] (0,-0.1) to (0,0.1);
  \end{tikzpicture}%
}

\makeatletter   
\xpatchcmd{\@tocline}
{\hfil\hbox to\@pnumwidth{\@tocpagenum{#7}}\par}
{\ifnum#1<0\hfill\else\dotfill\fi\hbox to\@pnumwidth{\@tocpagenum{#7}}\par}
{}{}
\makeatother    

\makeatletter
\def\l@subsection{\@tocline{2}{-9pt}{4pc}{6pc}{}}
\def\l@subsubsection{\@tocline{3}{-9pt}{8pc}{8pc}{}}
 \makeatother

\newcommand{\sbt}{\mathrel{\scaleobj{0.75}{*}}}

\newcommand{\defeq}{\vcentcolon=}

\newcommand{\Gm}{\textup{\textbf{G}}_m\xspace}
\newcommand{\Ga}{\textup{\textbf{G}}_a\xspace}
\newcommand{\BGm}{\textup{B}\textup{\textbf{G}}_m\xspace}
\newcommand{\BT}{\textup{B}T\xspace}

\newcommand{\spn}{\textup{span}\xspace}

\newcommand{\rank}{\textup{rk}\xspace}

\newcommand{\cprim}{P(\mathcal{A}^{*}_{Q})}

\newcommand{\Md}{\text{-}\textup{Mod}\xspace}

\newcommand{\CoMd}{\text{-}\textup{CoMod}\xspace}

\ExplSyntaxOff

\ExplSyntaxOn
\NewDocumentCommand{\massdefine}{m}
 {
  \clist_map_inline:nn { #1 }
   {
    \cs_new_protected:cpn { ##1 } { \textup{##1}\xspace }
   }
 }
\ExplSyntaxOff

\massdefine{Higgs,Bun,Set,Fun,Maps,loc,Hilb,QCoh,Coh,triv,
BM,Obj,Ext,Hom,Mor,Aut,pt,gr,Vect,Frac,im,Rep,Gr,corank,
coker,ad,cl,Res,act,elt,Ind,End,ch,ev,tr,Perf,Pic,Ran,Sh,CoSh,id,Zhu,Mod,coMod,YDMod,Alg,Art,Supp,colim,Stk,Fuk,
Cob,LocSys,Stab,Perv,IC,SI,IH,cSh,fib,cofib,oblv,cone,cocone,Top,Loc,
Ann,Sym,Gal,IndCoh,LGr,HF,PreStk,fin,Whit,KD,Dix,Vir,AbCat,Cat,Div,ComCoAlg,
fact,Lie,CoLie,CoAlg,HN,Jac,CoHA,Hitch,BD,BV,vol,crit,Cone,Hol,Isom,Sch,reg,
PSL,BGL,BP,BG,BSp,BSL,PGL,GL,PSp,PSO,Proj,ShvCat,Groth,dgCat,Conf,Fin,SES,FinSet,Prim,dR,Spf,
LieAlg,LieCoAlg,LieBiAlg,Aff,Tr,Disk,Disks,tors,Wh,CommAlg,CommCoAlg,PoisAlg,
PoisCoAlg,PoisBiAlg,Rib,BPS,PervCoh,corr,Tot,CE,codim,BiAlg,Comm,Pois,CoCommCoAlg,
BU,BSU,BConf,Ar,Diff,Crit,Fact,Harr,PSet,Inst,Tay,Spec,ncSpec,Op,Th,FactCat,FactAlg,FactCoAlg,
CoFactCat,CoMod,FactBiAlg,Cf,surj,Gpd,Gp,Sp,SL,fac,Eu,obj,Sing,Man,FactMod,FactCoMod,Ell,MF,QM,TQFT,taut,fr,BA,Mfld,HC,SSet,ord,
per,KU,HP,lax,oplax,Nak,CoSpan,Corr,Span,CoCorr,Tor,Av,ren,Hecke,
Sph,sm,hol,ext}

\ExplSyntaxOn

\NewDocumentCommand{\massdefinemathcal}{m}
 {
  \clist_map_inline:nn { #1 }
   {
    \cs_new_protected:cpx { ##1 } { \exp_not:N \mathcal { \tl_range:nnn { ##1 } { 1 } { -2 } } }
   }
 }

\NewDocumentCommand{\massdefinemathfrak}{m}
 {
  \clist_map_inline:nn { #1 }
   {
    \cs_new_protected:cpx { ##1 } { \exp_not:N \mathfrak { \tl_range:nnn { ##1 } { 1 } { -2 } } }
   }
 }
 
 \NewDocumentCommand{\massdefinetext}{m}
 {
  \clist_map_inline:nn { #1 }
   {
    \cs_new_protected:cpx { ##1 } { \exp_not:N \textup { \tl_range:nnn { ##1 } { 1 } { -2 } } }
   }
 }

\NewDocumentCommand{\textbftextup}{m}{%
  \textbf{\textup{#1}}%
}

 \NewDocumentCommand{\massdefinetextbf}{m}
 {
  \clist_map_inline:nn { #1 }
   {
    \cs_new_protected:cpx { ##1 } { \exp_not:N \textbftextup { \tl_range:nnn { ##1 } { 1 } { -2 } } }
   }
 }
 
\ExplSyntaxOff

\massdefinemathcal{Al,Bl,Cl,Dl,El,Fl,Gl,Hl,Il,Jl,Kl,Ll,Ml,Nl,Ol,Pl,Ql,Rl,Sl,Tl,Ul,Vl,Wl,Xl,Yl,Zl}
\massdefinemathfrak{Ak,Bk,Ck,Dk,Ek,Fk,Gk,Hk,Ik,Jk,Kk,Lk,Mk,Nk,Ok,Pk,Qk,Rk,Sk,Tk,Uk,Vk,Wk,Xk,Yk,Zk,
ak,bk,ck,dk,ek,fk,gk,hk,ik,jk,kk,lk,mk,nk,ok,pk,qk,rk,sk,tk,uk,vk,wk,xk,yk,zk,
slk,glk,sok,spk,ospk}
\massdefinetext{At,Bt,Ct,Dt,Et,Ft,Gt,Ht,It,Jt,Kt,Lt,Mt,Nt,Ot,Pt,Qt,Rt,St,Tt,Ut,Vt,Wt,Xt,Yt,Zt,
at,bt,ct,dt,et,ft,gt,jt,kt,lt,mt,nt,ot,qt,rt,st,ut,vt,wt,wtt,xt,yt,zt,sst,SSt,SUt,SLt,GLt,HHt,SOt,Spt,CEt,SpOt,OSpt,HCt,CSt,MCt,opt}
\massdefinetextbf{Ab,Bb,Cb,Db,Eb,Fb,Gb,Hb,Ib,Jb,Kb,Lb,Mb,Nb,Ob,Pb,Qb,Rb,Sb,Tb,Ub,Vb,Wb,Xb,Yb,Zb,
ab,bb,cb,db,eb,fb,gb,hb,ib,jb,kb,lb,mb,nb,ob,pb,qb,rb,tb,ub,vb,wb,xb,yb,zb,Grb,RPb,CPb}

\newcommand{\UU}{\mathrm{U}\xspace}

\newcommand{\HH}{\textup{H}\xspace}

\newcommand{\B}{\Bt}

\newcommand{\N}{\Nb}

\newcommand{\AQW}[3][]{\mathcal{A}^{#1}_{{#2},{#3}}}

\newcommand{\MSd}[3][]{\mathcal{M}^{#1}_{#2,#3}}

\begin{document}

\address{Ecole Polytechnique Fédérale de Lausanne (EPFL), CH-1015 Lausanne, Switzerland}
\email{shivang.jindal@epfl.ch}
\author{Shivang Jindal}
\address{Kavli Institute for the Physics and Mathematics of the Universe (WPI), The University
of Tokyo Institutes for Advanced Study, The University of Tokyo, Kashiwa, Chiba 277-8583,
Japan}
\email{sarunas.kaubrys@ipmu.jp}
\author{ \v{S}arūnas Kaubrys}
\address{Beijing Institute of Mathematical Sciences and Applications (BIMSA), Huairou District Beijing, 101408, China}
\email{alyoshalatyntsev@gmail.com}
\author{Alexei Latyntsev}

\date{\today}
\title{Critical CoHAs, vertex coalgebras and Deformed Drinfeld coproducts}
\maketitle
\begin{adjustwidth}{20pt}{20pt}
\small{\textsc{Abstract}: We construct a vertex coproduct on the Kontsevich--Soibelman cohomological Hall algebra (CoHA) of a quiver
with potential, following Joyce \cite{Jo}. We show it forms a vertex bialgebra. By applying a vertex algebraic analogue of Majid--Radford bosonisation, we form an extension of the CoHA of quivers with potential which incorporates a Cartan part. In the case of ADE quivers our vertex coproduct recovers Drinfeld's deformed coproduct on the Yangian. We compare the vertex coproduct with a localised
coproduct defined by Davison and with the construction of Dotsenko--Mozgovoy when the potential is trivial. Our construction gives a new proof of the cohomological integrality theorem for symmetric quivers with trivial potential.}
\end{adjustwidth}
\setcounter{tocdepth}{1}
\tableofcontents

\setcounter{section}{-1}

\section{Introduction}
\noindent
Using natural geometric structures on moduli spaces of objects $\Ml_{\Cl}$ in abelian or derived categories $\Cl$, Joyce \cite{joyce2021enumerativeinvariantswallcrossingformulae,Jo} has constructed a structure of a vertex algebra on the homology $\Ht_{\sbt}(\Ml_{\Cl})$ or equivalently a vertex coalgebra on the cohomology. This structure has already found several applications in enumerative geometry such as \cite{lim_bojko_moreira,gross_joyce_tanaka,bu_counting_sheaves} and \cite{liu2025equivariantktheoreticenumerativeinvariants} in the K-theory setting. However, most enumerative geometry is interested instead in the \textit{vanishing cycle cohomology} or Borel--Moore homology of such stacks. In this paper we fill this gap, defining a vertex coalgebra on critical Cohomological Hall algebras of quivers with potential. This gives a new geometric representation theory interpretation of these vertex coproducts, matching them with known (Drinfeld) coproducts on quantum groups.

Let $Q$ be a quiver with potential $W$. We study two structures on the following vector space
\begin{equation} \label{coha_vector_space_intro}
    \Al_{Q,W}\ =\ \bigoplus_{d \in \Nb^{Q_{0}}} \Ht^{\sbt}(\Ml_{Q,d}, \varphi_{\textup{Tr}W_{d}}\Qb_{\Ml_{Q,d}}[\dim \Ml_{Q,d}]),
\end{equation}
where $\Ml_{Q}$ is the moduli stack of quiver representations and $\varphi_{\textup{Tr} W,d}\Qb_{\Ml_{Q,d}}[\dim \Ml_{Q,d}]$ is the vanishing cycles perverse sheaf. 
The first algebraic structure is the Kontsevich-Soibelman Cohomological Hall algebra (CoHA) \cite{KS}.
The second is the new vertex coproduct, an algebraic structure generalising work of Joyce \cite{Jo} in the case of $W=0$, and analogous to Liu's vertex coproduct on critical K-theory \cite{liu2022multiplicative} for tripled quivers. 

We show these structures are compatible: they form a (vertex) \textit{bialgebra} in an ambient (meromorphic) braided category of modules over a tautological subring of cohomology. These types of results go back to Ringel and Green \cite{Gr,Ri} who also showed that Hall algebras are quantum groups. Analogously, we show that for triples of Dynkin quivers associated to simple Lie algebra $\gk$, we have an isomorphism of (vertex) bialgebras exchanging our vertex coproduct with the Drinfeld coproduct on Yangians:\footnote{The superscripts relate to working equivariantly with respect to a one-dimensional torus with $T_{\hbar}$ with $\Ht^{\sbt}(\BT_\hbar)=\Cb[\hbar]$, adding a sign correction $\chi$ to the CoHA product, and bosonisation, respectively.}
\begin{equation} 
  \label{eqn:CoHAYangian} 
  \Al_{\widetilde{Q},\widetilde{W}}^{T_{\hbar}, \chi, \ext} \ \simeq \ Y_{\hbar}(\gk)^{\ge 0}.
\end{equation}
The isomorphism as algebras is due to Yang--Zhao \cite{YZ}, in which the authors also construct a coproduct, which they show is equivalent to the Drinfeld coproduct.

Being a vertex bialgebra allows us to apply \textit{vertex Majid--Radford bosonisation} to extend CoHAs uniformly, i.e. add a Cartan piece, and in the Dynkin case this formula for its coproduct induces \eqref{eqn:CoHAYangian}. If $W=0$, we show that the Joyce vertex algebra is a universal chiral envelope and show it is isomorphic to the construction of \cite{dotsmozgva}. 
\subsection{Cohomological Hall product and Joyce--Liu vertex coproduct}  \label{ssec:CoHA_intro}
We now give a more detailed introduction to the algebraic structures we study in this paper.
\subsubsection{} All algebraic structures on $\Al_{Q,W}$ will be mirrored by and constructed from geometric structures on the moduli stack $\Ml_Q$ of representations of the quiver $Q$.
\begin{equation} \label{corresp}
\begin{tikzcd}[row sep = {30pt,between origins}, column sep = {45pt,between origins}]
 &\SES_Q \ar[rd,"p"]\ar[ld,,"q"'] & &[20pt] & \SES_Q  \ar[ld,"p"'] \ar[rd,"q"] & \\ 
\Ml_Q \times \Ml_Q & &\Ml_Q & \Ml_Q&& \Ml_Q \times \Ml_Q \ar[ll,"\oplus"'] 
\end{tikzcd}
\end{equation}
On the left we have the stack parametrising short exact sequences of representations of $Q$, with maps $p$ and $q$ forgetting all but the middle and outer terms. On the right we have the map $\oplus$ taking direct sum of representations, and only commutes if we replace $q$ with its section $s$ taking a pair of representations to the trivial direct sum extension.

In fact, the above diagrams are equivariant for a certain class of actions by a torus $T$ leaving the potential of the quiver invariant, and we usually consider a $T$-equivariant version $\mathcal{A}^{T}_{Q,W}$ of the vanishing cycles cohomology in \eqref{coha_vector_space_intro}, and $\Ml_Q^T$ the stack quotient by $T$.

The Kontsevich--Soibelman CoHA product \cite{KS} on critical cohomology is defined by pulling back by $q$ and then pushing forward by $p$. The Joyce-Liu vertex coproduct is instead defined by using the map $\oplus$:
\begin{enumerate}
    \item We can define a direct sum pullback $\oplus^{*} : \mathcal{A}^{T}_{Q,W} \to \mathcal{A}^{T}_{Q,W} \otimes \mathcal{A}^{T}_{Q,W}$ using compatibility of the vanishing cycles functor with various sheaf operations.
    \item We can define a translation, via pullback under the action $\act \colon \B \Gb_m \times \Ml^{T}_{Q} \to \Ml^{T}_{Q}$, which scales the automorphisms,
    \begin{equation}
        \act^{*} \colon \mathcal{A}^{T}_{Q,W} \ \to \ \mathcal{A}^{T}_{Q,W} \otimes \Ht^{\sbt}(\B \Gb_{m}) \ \simeq\ \mathcal{A}^{T}_{Q,W}[z]
    \end{equation}
    \item We multiply resulting structure by the so called \emph{Joyce-Borcherds} bicharacter
\begin{equation}\label{eqn:PsiIntro}
  \Psi(\Ext,-z)\ =\ \sum_{k \geq 0} (-z)^{\textup{rk} \Ext - k}c_{k}(\Ext)
\end{equation}
of the two term complex of vector bundles $\Ext  =  \textup{RHom}(-,-)  \in \textup{Perf}(\Ml^{T}_{Q} \times_{\BT} \Ml^{T}_{Q}).$
\end{enumerate}

The result below understands this structure as being a (braided colocal) \textit{vertex coproduct}, a definition due to Borcherds and Hubbard \cite{Bo,Hu}; see section \ref{sec:VA}. Loosely speaking, it is meant to capture the notion of a vector ``sitting at'' a point in $\Cb$ splitting into two vectors sitting at $z_1,z_2\in \Cb$, and this assignment depending meromorphically on $z=z_1-z_2$; see \cite{beilinson_drinfeld}.
\begin{thmX}[Theorem \ref{quiver_potential_jl}]\label{quiver_potential_jl_intro}
Let $Q$ be any quiver with potential $W$ and a torus action that leaves the potential invariant, and assume it satisfies the $T$-equivariant K\"unneth property \eqref{eqn:KunnethAssumption}.\footnote{This is known to hold if $T =1$ or in canonical tripled cases, see section \ref{sssec:kunneth_assumpt} for more discussion.}. Then
\begin{samepage}
\begin{align} \label{joyce_co_va_intro}
    \Delta(z)\ :\ \mathcal{A}^{T}_{Q,W} \ & \to \  \mathcal{A}^{T}_{Q,W} \otimes_{\Ht^{\sbt}(\BT)} \mathcal{A}^{T}_{Q,W}  (( z^{-1})) \nonumber \\[5pt]
    \alpha \ & \mapsto \  \ \Psi(\Ext,-z) \cdot \act^{*}_{1} \oplus^{*}(\alpha)
\end{align}
\end{samepage}
 defines a coassociative (alias noncolocal or weakly coassociative) vertex coalgebra. Furthermore, $\Delta(z)$  is cocommutative (alias colocal) in the sense that we have
 \begin{equation}
     \Delta(z) =   \sigma\cdot S(z)\Delta(-z) \act^{*} 
 \end{equation}
 with $  S(z)  = (-1)^{\rank \Ext}\frac{\Psi(\sigma^* \Ext^\vee, z)}{\Psi(\Ext, z)}$ and $\sigma$ the involution swapping the factors with Koszul sign rule. 
\end{thmX}

If the torus is trivial and the quiver symmetric, this says that the vertex coproduct is colocal up to a sign $S(z) = (-1)^{\rank \Ext}$.

\subsection{Localised versus vertex coproducts}

\subsubsection{} We also compare with another way of dealing with the singularities in \eqref{eqn:PsiIntro} using \textit{localised coproducts}. Both this and the vertex coproduct are attempting to define the pull-push ``$q_*p^*$'' along the right diagram \eqref{corresp}, but with $q$ not being proper this is not a priori well-defined. 

The method of Davison \cite{Da2} is to formally define $q_*$ with the inverse of $q^*$ composed with multiplication by its Euler class of its shifted tangent complex, which is a quotient $e(Q_{0})/e(Q_1)$ of cohomology classes. The resulting coproduct
$$\Delta_{\loc} \ : \ \Al^T_{Q,W} \ \to \ \Al^T_{Q,W} \otimes_{\Ht^{\sbt}(\BT)}\Al^T_{Q,W}[S^{-1}]$$
is therefore only well-defined if we localise a set $S$ of cohomology classes.

The connection between the localised coproduct and the Joyce-Liu coproduct is given by an observation that the class $e(Q_{0})/e(Q_1)$ is exactly the Euler class of the complex $e(\Ext)$. Furthermore, we have that 
\begin{equation}
    \act^{*}_{1} e(\Ext)\ = \ \Psi(\Ext, -z)
\end{equation}
where on the left we pull back by the $\BGm$ action on the first factor then expand as a Laurent series in $z^{-1}$. This allows us to bypass the need to invert, at the cost of working with series in $z$, and gives the following comparison result.

\begin{thmX}[Theorem \ref{thm:DavisonIsJoyce}] \label{loc_intro}
There is a map
\begin{equation}
      \act^{*}_{1} : \AQW[T]{Q}{W} \otimes  \AQW[T]{Q}{W}[S^{-1}]\ \to \  \AQW[T]{Q}{W} \otimes \AQW[T]{Q}{W} ((z^{-1}))
\end{equation}
intertwining the Joyce--Liu vertex and Davison localised coproducts: $\act^{*}_{1} \cdot \Delta_{\loc}= \Delta(z)$.
\end{thmX}

We expect this should be viewed as a comparison result between algebraic structures on configuration spaces of the cohomology rings of $\Ml^{T}_{Q}$.

\subsection{Compatibility and bosonisation}

\subsubsection{} The first main result of this paper is to show that our two structures are compatible. We state the theorem for the sign twisted (c.f. section \ref{ssec:SignTwist}) versions of the multiplication although we get analogous theorems for the usual CoHA.
\begin{thmX}[Theorem \ref{thm:CoHABialgebra} symmetric case, Theorem \ref{thm:CoHABialgebra_non_symm} general case] \label{thm:BialgIntro}
 The ($\psi$ sign twisted) CoHA product $\star$ and Joyce--Liu vertex coproduct $\Delta(z)$ on together form a vertex bialgebra on $\mathcal{A}^{T,\psi}_{Q,W}$; in particular,
    $$\Delta(b\star b', z) \ = \ \Delta(b,z)\star_{R(z)}\Delta(b',z)$$
    where on the right we have braided by the spectral $R$ matrix 
   \begin{equation}
    R(z)\ =\  \frac{\Psi(\sigma^* \Ext^{\vee}, z)}{\Psi(\Ext, z)}
\end{equation}
   before multiplying the first/third and second/fourth factors. \footnote{Note that this $R$-matrix agrees with the operator $S(z)$ in Theorem \ref{quiver_potential_jl_intro} up to a sign.} 
\end{thmX}

The form of Theorem \ref{thm:BialgIntro} suggests that the CoHA $\mathcal{A}^{T,\psi}_{Q,W}$ is a bialgebra object internal to some ``meromorphic'' braided monoidal category, with braiding given by the spectral $R$-matrix $R(z)$. Indeed, to state the bialgebra axioms one is forced to work within a category with at least a braided monoidal structure. 

We will consider the category of modules over the \textit{tautological subring} $\Ht^{\sbt}(\Ml^{T}_Q)_{\taut} \subseteq  \Ht^{\sbt}(\Ml^{T}_Q)$ of the cohomology of the moduli stack, which is generated by chern classes of tautological bundles. Joyce's $R$-matrix splits as
$$R_{d_{1},d_{2}}(z) \ = \ z^{\chi(d_{1},d_{2})-\chi(d_{2},d_{1})}R_{\taut,d_{1},d_{2}}(z)$$
consisting of a rational function which is trivial in the symmetric case, and the component $R_{\taut}(z)$ is tautological-valued:
$$R_{\taut}(z) \ \in \ \Ht^{\sbt}(\Ml^{T}_Q)_{\taut} \otimes_{\Ht^{\sbt}(\BT)}\Ht^{\sbt}(\Ml^{T}_Q)_{\taut}[[z^{-1}]].$$
Thus $B=\Al^{T,\psi}_{Q,W}$ and $H= \Ht^{\sbt}(\Ml^{T}_Q)_{\taut}$ satisfy the conditions of Theorem \ref{thm:IntroBosonisation} below. \par
In the non-symmetric case we do not give such an interpretation since in that case the category of tautological ring modules is \textit{not} braided. See subsection \ref{ssec:non_symm} for some more discussion.
\subsubsection{Bosonisation} 
\begin{thmX}[Theorem \ref{thm:VertexBosonisation}] \label{thm:IntroBosonisation} Let $B$ be a vertex bialgebra inside $H\Md$, for $H$ a holomorphic\footnote{See section \ref{sssec:Hol} for a definition, and Lemma \ref{lem:holomorphic_covertex}. } vertex bialgebra with commutative multiplication and spectral quasitriangular element $R(z)$. Then the vector space 
$$B\rtimes H\ = \ B\otimes H$$
has a canonical vertex bialgebra structure, with explicit formulas \eqref{eqn:VertexBosonisationProduct} and \eqref{eqn:VertexBosonisationCoproduct}.
\end{thmX}
This should be viewed as a Tannakian reconstruction of the algebra whose category of modules is $B\Md(H\Md)$. The vertex coproduct comed from structure on this category. In the case of ordinary bialgebras the procedure of constructing $B \rtimes H$ from a  bialgebra $B$ in $H\Md$  is known as bosonisation, discovered by Majid and Radford \cite{MaBos,Ra}.

\begin{corX}[Theorem \ref{thm:CoHABosonisation}] \label{cor:CoHAExtended}
    Let $Q$ be a symmetric quiver with potential. Then the extended CoHA 
    \begin{equation}\label{eqn:IntroAExt}
        \mathcal{A}^{T, \psi, \textup{ext}}_{Q,W}\ =\ \mathcal{A}^{T,\psi}_{Q,W} \rtimes \Ht^{\sbt}(\Ml_{Q})_{\textup{taut}},
    \end{equation}
 whose underlying vector space is $\mathcal{A}^{T,\psi}_{Q,W} \otimes_{\Ht^{\sbt}_{T}(\pt)} \Ht^{\sbt}(\Ml_{Q})_{\textup{taut}}$, has a canonically defined vertex bialgebra structure, whose ``extended CoHA'' product \eqref{eqn:CoHABosonisedProduct} and ``extended Joyce--Liu'' vertex coproduct \eqref{eqn:CoHABosonisedCoproduct}, e.g.
 $$ \Delta^{\ext}(b\otimes h,z) \ = \ \left(\Delta(b,z)_{(1)}\otimes R^{(2)}(z)\Delta(h,z)_{(1)} \right)\otimes \left(R^{(1)}(z)\cup\Delta(b,z)_{(2)}\otimes \Delta(h,z)_{(2)} \right).$$
\end{corX}

\subsection{Obtaining Drinfeld's coproduct on ADE Yangians}\label{ssec:DrinfeldYangians}

\subsubsection{} For a semisimple finite-dimensional Lie algebra $\mathfrak{g}$, Drinfeld \cite{Dr2} defined a deformation of the universal enveloping loop algebra 
$$\gr_\hbar Y_{\hbar}(\gk)\ \simeq\ \UU(\gk[u])$$
called the \textit{Yangian} $Y_\hbar(\gk)$. This carries three algebraic structures: an associative product, an ordinary \textit{standard} coproduct, and the \textit{deformed Drinfeld  coproduct}
$$\Delta_{\textup{Dr}}(z) \ : \ Y_\hbar(\gk) \ \to \ Y_\hbar(\gk)\otimes_\hbar Y_\hbar(\gk)((z^{-1}))$$
whose definition was finally worked out in full in \cite{gautam2017meromorphic,GTW}. We recall it in section \ref{ssec:DrinfeldCoproduct}.

As a vector space the Yangian admits a triangular decomposition
$$Y_\hbar(\gk) \ = \ Y^{<0}_\hbar(\gk) \otimes Y_\hbar^0(\gk) \otimes Y^{>0}_\hbar(\gk)$$
where each tensor factor is a subalgebra and $\Delta_{\textup{Dr}}(z)$ preserves the positive and negative Borel parts $Y^{\ge 0}_\hbar(\gk)$, $Y^{\le 0}_\hbar(\gk)$: they are sub vertex bialgebras.

\subsubsection{} By combining the work of Yang--Zhao \cite{yang2018cohomological}, Schiffmann-Vasserot  \cite{schiffmann2022cohomologicalhallalgebrasquivers} and an algebraic dimensional reduction theorem \cite{dimensionalreductionalgebras,yangzhaodimension}, there is an isomorphism of $\mathbf{C}[\hbar]$-algebras 
\begin{equation}\label{eqn:YangZhaoIso}
    Y^{>0}_{\hbar}(\gk)\ \stackrel{\sim}{\to} \ \mathcal{A}^{T,\chi}_{\widetilde{Q},\widetilde{W}}
\end{equation}
where $\mathcal{A}^{T,\chi}_{\tilde{Q},\tilde{W}}$ is the cohomological Hall algebra of the tripled Dynkin quiver $\widetilde{Q}$ with canonical cubic potential, an an appropriately chosen torus.

In this setting, in Proposition \ref{prop:JoyceTripledComputation} we compute the Joyce--Liu extended coproduct on spherical elements of tripled quivers. Their value is very simple: they are all vertex-primitive, and so under the isomorphism \eqref{eqn:YangZhaoIso}  corresponds to the ``unbosonised Drinfeld'' vertex coproduct 
$$\Delta_{\textup{uDr}}(x_i^+(u),z)\ = \ x^+_i(u-z)\otimes 1 + 1 \otimes x_i^+(u).$$ 
Then identifying $Y^0(\gk)$ with the tautological cohomology ring, we obtain the deformed Drinfeld coproduct as a bosonisation:

\begin{thmX}[Theorem \ref{thm:DrinfeldJoyce}] \label{drinfeld_to_joyce_comparison_intro} 
For $Q$ an ADE quiver, there is an isomorphism  
\begin{equation}
       f \ : \ Y^{\ge 0}_\hbar(\gk)  \ \stackrel{\sim}{\to} \ \mathcal{A}^{T,\chi, \textup{ext}}_{\widetilde{Q},\widetilde{W}} 
\end{equation}
identifying the Drinfeld coproduct $\Delta_{\textup{Dr}}(z)$ and the Joyce--Liu coproduct $\Delta(z)$. 
\end{thmX}
An analogous result was first obtained in \cite[Thm.B]{YZ}. 
\subsubsection{Outlook: Maulik-Okounkov Yangians} Due to work of Cao, Maulik, Okounkov, Zhou and Zhou \cite{MO,COZZ} one may define a Yangian $Y_\hbar(\gk_{Q,W})$ for any symmetric quiver with potential, which one expects to come with a standard coproduct as well as a Drinfeld coproduct
$$ \Delta_{\textup{Dr}}(z) \ : \ Y_\hbar(\gk_{Q,W})  \ \to \ Y_\hbar(\gk_{Q,W}) \otimes_\hbar Y_\hbar(\gk_{Q,W}) ((z^{-1})),$$
and with $R$-matrices intertwining the two coproducts and their opposites with each other.\footnote{Here $\hbar$ denotes the equivariant parameters of the torus acting.} The analogue of \eqref{eqn:YangZhaoIso} has been shown for tripled quivers with potential by Botta--Davison and Schiffmann--Vasserot \cite{BD,schiffmann2024s} and is expected for general quivers with potential. The diagonal part of the $R$-matrix for $Y_\hbar(\gk_{Q,W})$ is expected to match up with this paper's $R(z)$.

We expect our method of explicitly computing bosonised vertex coproducts on spherical elements can be extended to symmetric quivers with potential, so long as we have $T$-equivariant K\"unneth assumption. Therefore, given any equivalence as in \eqref{eqn:YangZhaoIso} one should be able to identify the Drinfeld coproduct with the Joyce vertex coalgebra when restricted to the spherically generated subalgebra. For instance, by \cite{davison2022affine}, \cite{sch_vas_cheredenik}, \cite{jindal2026cohacyclicquiversintegral} and \cite{diaconescu2026cohomologicalhallalgebrasonedimensional}, this gives Drinfeld type coproducts for the affine Yangians.


\subsection{Why should such structure exist?} 

\subsubsection{} We briefly explain two guiding points of view which motivated this work.

\subsubsection{Quantum groups} \label{ssec:MotivationQuantumGroups} Let $\gk$ be a finite-dimensional simple Lie algebra with triangular decomposition $\gk= \nk^-\oplus \tk \oplus \nk^{+}$. We consider the quantum group 
$$U_q(\nk^{+})\ \in \ \BiAlg(\Rep_qT)$$
 It is a bialgebra satisfying $\Delta(e_i)=e_i\otimes 1 + 1\otimes e_i$ for every generator $e_i\in \nk^{+}$; see \cite{Ga}. This has the additional structure of a Nichols algebra \cite{andruskiewitsch2002pointed}. Here, $\Rep_qT$ is the category vector spaces graded over the weight lattice $\Lambda_G$, viewed as a braided monoidal category with braiding 
$$\sigma q^{\kappa(\lambda,\mu)} \ : \ \Cb_\lambda \otimes\Cb_\mu \ \stackrel{\sim}{\to} \  \Cb_\mu \otimes \Cb_\lambda.$$
We view an object as a module over the group algebra $\Cb[\Lambda_G]$, where a generator $k_\nu$ acts as multiplication by $q^{\kappa(\nu,-)}$.

In particular, $U_q(\nk^{+})$ is \textit{not} a subbialgebra of $U_q(\bk)\subseteq U_q(\gk)$; instead we construct $U_q(\bk)$ by \textit{bosonisation}: it is the algebra whose category of modules is $U_q(\nk^+)\Md(\Cb[\Lambda_G]\Md)$, which for generic $q$ recovers the usual definition
$$U_q(\bk) \ \simeq\  U_{q}(\nk^{+}) \rtimes \Cb[\Lambda_G],$$
by \cite[33.1.5]{Lu} or \cite[Thm.4.2]{andruskiewitsch2002pointed}. The coproduct is then $\Delta^{\ext}(e_i)= e_i\otimes 1 + k_i\otimes e_i$, which now \textit{does} agree with the coproduct in $U_q(\gk)$. One can then take Drinfeld double of to obtain all of $U_q(\gk)$.

\subsubsection{Hall algebras} By a result of Ringel and Green \cite[Thm.3.16]{Sch}, the Hall algebra of $\Fb_q$-representations of an ADE quiver $Q$, is when extended by the Grothendieck group isomorphic to the Borel quantum group
$$\Hb_Q \rtimes \Kt_0(\Rep_{\Fb_q}Q) \ \simeq \ U_q(\bk).$$
The coproduct on $\Hb_Q$ is due to Green \cite[$\S$1.4]{Sch}.

Comparing our paper to the above two stories, the analogy goes as follows:
\begin{center}
\begin{tabular}{c@{\hspace{20pt}}ccccccc}
{\small\textit{finite quantum groups}}  & $U_q(\nk^{+})$&$\Delta$ & $\Rep_qT$  & $q^\kappa$ & $\kappa$ &$\Lambda_G$ & $\cdots$\\[5pt]
{\small\textit{Hall algebras}} & $\Hb_{Q}$ & $\Delta_{\text{Green}}$ & $\Kt_0(\Rep_{\Fb_q} Q)\Md$ &$q^\chi$ & $\chi = \# \Ext_{\Fb_q}$ & $\pi_0(\Ml_Q)$ & $\cdots$\\ [5pt]
{\small\textit{CoHAs - this paper}} & $\Al^T_{Q,W}$ & $\Delta(z)$ & $\Ht^{\sbt}(\Ml^{T}_Q)_{\taut}\Md_{\pi_0(\Ml_{Q})\times \Zb}$ &$R(z)$ & $\chi=\rank \Ext$ & $\pi_0(\Ml_Q)$ & $\cdots$\\ 
 \end{tabular} 
\end{center}

\subsubsection{Physics} \label{ssec:MotivationPhysics}
 In physics, the category of modules over the CoHA $B=\Al^T_{Q,W}$ is expected to be identified with the category of \textit{line operators} of a certain $3d$ holomorphic-topological quantum field theory $\partial \Tl$ on $\Cb\times \Rb$, see \cite{CWY,DN,DNP}, and colliding these lines starting above points $z_i,z_j\in \Cb$ is meant to correspond to the meromorphic tensor structure induced by the Joyce--Liu vertex coproduct $\Delta(z_i-z_j)$.
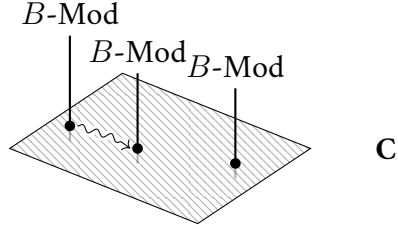
\begin{center}
\begin{tikzpicture}
\filldraw[pattern=north west lines, pattern color=black!30] 
  (-2,0) -- (0.5,-1) -- (2,0) -- (-0.5,1) -- cycle;

  \draw[thick] (-0.3,0) -- (-0.3,1);
  \draw[thick,opacity=0.3] (-0.3,-0.2) -- (-0.3,0);
  \fill (-0.3,0) circle (2pt);
  \node[above] at (-0.3,1) {$B\Md$};

    \draw[thick] (-1.2,0.3) -- (-1.2,1.5);
    \draw[thick,opacity=0.3] (-1.2,0.1) -- (-1.2,0.3);
    \fill (-1.2,0.3) circle (2pt);
    \node[above] at (-1.2,1.5) {$B\Md$};

    \draw[thick] (1,-0.2) -- (1,0.8);
    \draw[thick,opacity=0.3] (1,-0.4) -- (1,-0.2);
    \fill (1,-0.2) circle (2pt);
    \node[above] at (1,0.8) {$B\Md$};

    \draw[-,decorate,decoration={snake,amplitude=.4mm,segment length=2mm,post length=1mm},line width = 3pt,white] (-1.2+0.1,0.3) to[out=-20,in=140] (-0.3-0.1,0);
     \draw[->,decorate,decoration={snake,amplitude=.4mm,segment length=2mm,post length=1mm}] (-1.2+0.1,0.3) to[out=-20,in=140] (-0.3-0.1,0);

  \node at (3,0) {$\Cb$};
\end{tikzpicture}
\end{center}
More precisely, it induces a $B$-module structure on $M\otimes N \left(\left( (z_i-z_j)^{-1}\right)\right)$ for any pair of $B$-modules $M,N$. In particular, this gives a physics explanation for why we should expect a vertex coproduct on the CoHA (and nothing more).

This is expected to the boundary (or interface) of a $4d$ holomorphic-topological quantum field theory on $\Cb\times \Rb^2$, obtained by twisting \cite{elliot_safronov} the reduction of string theory along a Calabi--Yau threefold $X$, whose category of coherent sheaves are locally modelled on $(Q,W)$. This is the physics heuristic for the relation between the Calabi--Yau-three category and the (double) CoHA. 

For instance, as this paper was being finished, Tudor Dimofte and Lo\"ic Bramley informed us about ongoing work related to CoHAs and vertex coalgebras. In particular, they use a Tannakian reconstruction approach on the category of line operators of the IR limit of a $4d$ holomorphic-topologically twisted QFT to produce Yangian-like algebras equipped with standard and vertex coproducts. In this case, the category of line operators is expected to be related to BPS states and the algebra produced is conjecturally the double of the CoHA. The corresponding vertex coproduct is expected to match the extended Joyce-Liu coproduct on the nonnegative half.

\subsection{Relation to the work of Dotsenko--Mozgovoy}
\subsubsection{}
When we take a symmetric quiver $Q$ and set the potential $W= 0$ then the CoHA $\AQW[]{Q}{0}$ becomes a supercommutative algebra. We can then dualise the algebra structure to obtain a cocommutative coalgebra and dualise the Joyce vertex coalgebra to obtain that $(\AQW[]{Q}{0})^{\vee}$ is a cocommutative coalgebra with a compatible vertex algebra structure. In this setting we can then apply a Milnor-Moore type theorem \cite{han2021cocommva} to conclude that 
\begin{thmX}[Theorem \ref{chiral_envelope}] \label{chiral_intro}
 The dual of the CoHA is a universal chiral envelope 
 \begin{equation}
     (\AQW[]{Q}{0})^{\vee} \simeq \textup{U}^{\ch}(P_{Q})
 \end{equation}
  where $P_{Q}$ is a vertex Lie algebra defined as the primitives of the dual of the CoHA multiplication.
\end{thmX}
In \cite{dotsmozgva}, the authors give an explicit generators and relations definition of a vertex algebra on the vector space  $(\AQW[]{Q}{0})^{\vee}$ and show it is compatible.
\begin{thmX}[Theorem \ref{DM_comp}]
    The Joyce vertex algebra structure on $\mathcal{A}^{*}_{Q,0}$ and the vertex algebra in \cite{dotsmozgva} are isomorphic.
\end{thmX}
Therefore, for a symmetric quiver with no potential we obtain a completely explicit generators and relations computation of the Joyce vertex algebra. Finally, as a corollary of Theorem \ref{chiral_intro} we get another proof of  cohomological integrality originally proved in \cite{efimov2012cohomological}.

\subsection{Acknowledgements} 

\subsubsection{} We would like to thank Lo\"ic Bramley, Ben Davison, Tudor Dimofte,  Woonam Lim,  Henry Liu,  Olivier Schiffmann, Yaping Yang and Gufang Zhao for useful discussions. We would also like to thank the organisers Ben Davison and Lucien Hennecart of the 2023 Workshop on "Vertex algebras and Hall algebras in enumerative geometry" where this project was initially started. The work of the second author was supported by World Premier International Research Center Initiative (WPI), MEXT, Japan, and the third author for most of this project by Danish National Research Foundation (DNRF157) and Villum Fonden (37814).

\newpage
\section{Preliminaries} \label{sec:Quiver}
In this section we introduce the moduli stacks of quiver representations.

\subsection{Quivers with potential } \label{ssec:Quivers}

\subsubsection{Quivers} 
\begin{center}
\begin{tikzpicture}[scale=1.5]
\begin{scope} 
 [rotate = -90] 
 \filldraw[] (0,0) circle (1.5pt);
 \filldraw[] (1,1) circle (1.5pt);
 \filldraw[] (-0.3,1.2) circle (1.5pt);
 \filldraw[] (-1,1) circle (1.5pt);
 \filldraw[] (0,-1.3) circle (1.5pt);

 \draw[<-] (0.1-0.07,0.1+0.07) to[out=45+20,in=-135-20] (0.9-0.05,0.9+0.05);
 \draw[->] (0.1+0.07,0.1-0.07) to[out=45-20,in=-135+20] (0.9+0.05,0.9-0.05);

 \draw[->] (-0.1,0.1) -- (-0.9,0.9);
 \draw[->] (0,-0.14) -- (0,-1.16);
 \draw[->,looseness=10] (1,1+0.14) to[out=80,in=10] (1+0.14,1);
 \draw[->,looseness=10]  (-0.1,-1.3-0.1) to[out=-135+10,in=-45-10]  (0.1,-1.3-0.1);

  \filldraw[] (0.6,1.6) circle (1.5pt);
  \filldraw[] (-0.6,0.3) circle (1.5pt);
  \filldraw[] (0.8,2.3) circle (1.5pt);
  \filldraw[] (-0.8,-0.2) circle (1.5pt);
  \draw[->, shorten <=4pt, shorten >=4pt] (0.6,1.6) -- (0.8,2.3);
  \draw[<-, shorten <=4pt, shorten >=4pt] (-0.6,0.3) -- (-0.8,-0.2);

  \node[below,yshift=-3pt] at (0,-1.3) {$i$};
  \node[above,yshift=0pt] at (0,-0.65) {$\wtt(e)$};
  \node[below,yshift=-2pt] at (0,-0.65) {$e$};
 \end{scope}
 \end{tikzpicture}
 \end{center}
A \textbf{quiver} is a directed graph $Q = (Q_0,Q_1)$ whose sets $Q_0,Q_1$ of vertices $i$ and edges $e$ are finite. It is \textbf{graded} with respect to a lattice $N \simeq \Zb^r$ if we have a weight function $\wtt \colon Q_{1} \to N$ assigning an element of the lattice $\wtt(e) \in N $ to each edge. Write $s(e),t(e)$ for the source and target of the edge.

The opposite quiver $Q^{\textup{op}}$ is formed by reversing the direction of all arrows, and negating all gradings: $\wtt^{\textup{op}}(e^*) = - \wtt(e)$ where $e^*$ is the reversed version of $e$. We call a quiver \textbf{symmetric} if $Q = Q^{\textup{op}}$ and graded quiver \textbf{graded symmetric} if $(Q,\wtt) \simeq (Q^\textup{op},\wtt^{\textup{op}})$.

\subsubsection{Representations} A \textbf{representation} of $Q$ is an assignment of a finite dimensional vector space $V_i$ to each vertex and a linear map $\rho_e$ to each edge

\begin{center}
\begin{tikzpicture}[scale=1.5]
\begin{scope} 
 [rotate = -90] 
\clip (0,-0.9) circle (1.2cm);
 \filldraw[] (1,1) circle (1.5pt);

 \draw[<-] (0.1-0.07,0.1+0.07) to[out=45+20,in=-135-20] (0.9-0.05,0.9+0.05);
 \draw[->] (0.1+0.07,0.1-0.07) to[out=45-20,in=-135+20] (0.9+0.05,0.9-0.05);

 \draw[->] (-0.1,0.1) -- (-0.9,0.9);
 \draw[->] (0,-0.16) -- (0,-1.16);
 \draw[->,looseness=10] (1,1+0.14) to[out=80,in=10] (1+0.14,1);
 \draw[->,looseness=10]  (-0.1,-1.3-0.2) to[out=-135+10,in=-45-10]  (0.1,-1.3-0.2);

  \filldraw[] (0.6,1.6) circle (1.5pt);
  \filldraw[] (0.8,2.3) circle (1.5pt);
  \draw[->, shorten <=4pt, shorten >=4pt] (0.6,1.6) -- (0.8,2.3);

  \node[above,yshift=0pt] at (0,-0.65) {$\rho_e$};

  \node[] at (0,-1.31) {$V_i$};
  \node[] at (0,-0.01) {$V_j$};
 \end{scope}

  \node at (0.9,0) {$\cdots$};
 \end{tikzpicture}
 \end{center}
This is equivalent to a finite dimensional left module over the \textit{path algebra} $\Cb Q$ of $Q$, whose basis is the set of paths in $Q$ and multiplication is given by concatenation of paths. The \textit{dimension} vector of a representation is the element $d = (\dim V_i)_{i\in Q_0} \in \Nb^{Q_0}$, writing $d_{i}$ for the dimension at the $i$th vertex. We write $\delta_i$ for the dimension vector which is $1$ at vertex $i$ and zero elsewhere.  

The category $\Rep Q$ of representations of $Q$ is an abelian category, and the rank of the derived Hom spaces 
$$\chi(\rho, \rho')  \ = \  \dim(\Hom(\rho, \rho')) - \dim(\Ext^{1}(\rho,\rho')),$$
is called the \textit{Euler form}, which only depends on the dimension vectors $d,d'$ of $\rho,\rho'$, and is 
$$\chi(d,d') \ = \ \sum_{i } d_i d'_i  -  \sum_{e : i \to j} d_i d'_j.$$

\subsubsection{Potential} A \textbf{potential} of $Q$ is an element $W\in \Cb Q/ [\Cb Q,\Cb Q]$. 

A potential is given by a linear combination of cyclic words in $Q$, where two cyclic words are considered to be the same if one can be cyclically permuted to be the other. If $W$ is a single cyclic word and $e \in Q_1$, then we define \[ \frac{\partial W}{\partial e}\ =\ \sum _{W = cec'} c'c\]
    and we extend this definition linearly to general $W$. The \emph{Jacobi algebra} associated to $(Q,W)$ is defined to be \[ \Jac(Q,W)\ \defeq\ \Cb Q/\left\langle \frac{\partial W}{\partial e}  \, : \, e \in Q_1 \right\rangle. \]

The direct sum and tensor product of representations are defined vertex-wise.

\subsection{Doubled and tripled quivers} \label{canonicalcubicpotential}
 
\subsubsection{Preprojective and Jacobi algebras}The \textit{double} $\overline{Q}$ and \textit{triple} $Q^{(3)}$ of a quiver $Q$ are quivers formed by adding a copy $e^*:j\to i$ of each edge $e:i\to j$ in the opposite direction, and then adding a loop $\omega_i$ to each vertex, respectively.
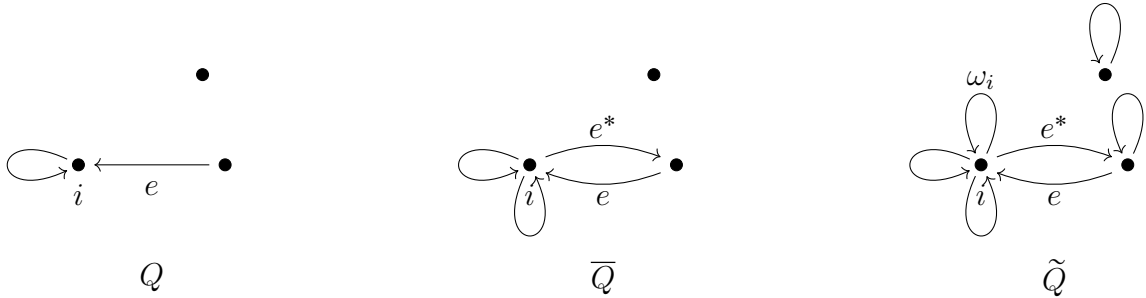
\begin{center}
\begin{tikzpicture}[scale=1.5]
\begin{scope} 
 [rotate = -90] 
 \filldraw[] (0,0) circle (1.5pt);
  \filldraw[] (-0.8,-0.2) circle (1.5pt);
 \filldraw[] (0,-1.3) circle (1.5pt);

 \draw[->] (0,-0.14) -- (0,-1.16);
 \draw[->,looseness=20]  (-0.05,-1.3-0.1) to[out=-135+20,in=-45-20]  (0.05,-1.3-0.1);

  \node[below,yshift=-3pt] at (0,-1.3) {$i$};
  \node[below,yshift=-2pt] at (0,-0.65) {$e$};

  \node at (1,-0.65) {$Q$};

\begin{scope} 
 [yshift=4cm] 
 \filldraw[] (0,0) circle (1.5pt);
  \filldraw[] (-0.8,-0.2) circle (1.5pt);
 \filldraw[] (0,-1.3) circle (1.5pt);

 \draw[->] (0.07,-0.14) to[out=-90+20,in=90-20] (0.07,-1.16);
 \draw[<-] (-0.07,-0.14) to[out=-90-20,in=90+20] (-0.07,-1.16);

 \draw[->,looseness=20]  (-0.05,-1.3-0.1) to[out=-135+20,in=-45-20]  (0.05,-1.3-0.1);
 \draw[->,looseness=20]  (0.1,-1.3-0.05) to[out=-135+20+90,in=-45-20+90]  (0.1,-1.3+0.05);

  \node[below,yshift=-3pt] at (0,-1.3) {$i$};
  \node[below,yshift=-6pt] at (0,-0.65) {$e$};
  \node[above,yshift=+6pt] at (0,-0.65) {$e^*$};

  \node at (1,-0.65) {$\overline{Q}$};
 \end{scope}

\begin{scope} 
 [yshift=8cm] 
 \filldraw[] (0,0) circle (1.5pt);
  \filldraw[] (-0.8,-0.2) circle (1.5pt);
 \filldraw[] (0,-1.3) circle (1.5pt);

 \draw[->] (0.07,-0.14) to[out=-90+20,in=90-20] (0.07,-1.16);
 \draw[<-] (-0.07,-0.14) to[out=-90-20,in=90+20] (-0.07,-1.16);
 
 \draw[->,looseness=20]  (-0.05,-1.3-0.1) to[out=-135+20,in=-45-20]  (0.05,-1.3-0.1);
 \draw[->,looseness=20]  (0.1,-1.3-0.05) to[out=-135+20+90,in=-45-20+90]  (0.1,-1.3+0.05);
 \draw[->,looseness=20]  (-0.1,-1.3+0.05) to[out=-135+20-90,in=-45-20-90]  (-0.1,-1.3-0.05);

 \draw[->,looseness=20]  (-0.1,0.05) to[out=-135+20-90,in=-45-20-90]  (-0.1,-0.05);
 \draw[->,looseness=20]  (-0.8-0.1,-0.2+0.05) to[out=-135+20-90,in=-45-20-90]  (-0.8-0.1,-0.2-0.05);

  \node[below,yshift=-3pt] at (0,-1.3) {$i$};
  \node[below,yshift=-6pt] at (0,-0.65) {$e$};
  \node[above,yshift=+6pt] at (0,-0.65) {$e^*$};
  \node[above,yshift=3pt] at (-0.5,-1.3) {$\omega_i$};

  \node at (1,-0.65) {$\widetilde{Q}$};
 \end{scope}

 \end{scope}
 \end{tikzpicture}
 \end{center}
 The \textit{preprojective algebra} is the quotient by the two-sided ideal 
 $$\Pi_Q \ = \ \Cb \overline{Q} / \left\langle \textstyle{\sum}_e [e,e^*] \right\rangle.$$
The \textit{canonical tripled potential} on $\widetilde{Q}$ is
$$ \widetilde{W} \ = \  \left(\sum_{i} \omega_i \right) \left( \sum_{e} [e,e^{*}] \right). $$
We have an isomorphism 
\[ \Pi_Q[\omega]\ \simeq\ \Jac(\widetilde{Q},\widetilde{W})\]
for the polynomial ring $\Pi_Q[\omega]$, sending $\omega \mapsto \sum \omega_i$.

\subsection{Moduli stacks of quiver representations} \label{ssec:quiver_stacks}

 \subsubsection{} The \textbf{moduli stack of representations} of $Q$ of dimension $d$ is the quotient stack
\[\MSd[]{Q}{d}\ \defeq\ \Rep_{d}(Q)/\GL_{d} \] 
of the following vector space by the following group:
\begin{align*}
\Rep_{d}(Q) &\ \defeq\ \bigsqcap_{e:i \to j} \Hom(\Cb^{d_i},\Cb^{d_j}) \\ 
\GL_{d} &\ \defeq\ \bigsqcap_{i \in Q_0} \GL_{d_i}
\end{align*} 
where the group $\GL_{d}$ acts on $\Rep_{d}(Q)$ by conjugation. The union over all dimension vectors is the moduli stack of objects of the abelian category $\Rep Q$ in the sense of \cite{TV}. For instance, its $\Cb$-points are precisely the groupoid of finite dimensional representations of $Q$.

\subsubsection{Torus equivariance} Let $Q$ be a quiver graded by the lattice $N$ of a torus $T=\Hom(N,\Gm)$, with its grading denoted $\wtt:Q_1\to N$ as before. Thus we have an action of the torus on the representation space 
$$T \times \Rep_d(Q) \ \to \ \Rep_d(Q), \hspace{15mm} t\cdot \rho_e \ = \ t^{\wtt(e)} \rho_e$$
where $t^{\wtt(e)}\in \Gm$ is the image of an element of the torus under the character $\lambda_e:T\to \Gm$ associated to $\wtt(e)\in N$, which acts by multiplication on the element of the vector space $\rho_e\in \Hom(\Cb^{d_i},\Cb^{d_j})$.

The homomorphism $\lambda_e$ gives a morphism of stacks $\lambda_{e}  : \BT \to \BGm$. This defines a cohomology class
$$\mathbf{t}(e) \ \in\ \HH^{2}(\BT,\Qb)\ \simeq \ N$$
defined by the pullback along $\lambda_e$ of the first chern class of the tautological bundle on $\BGm$. Note that each weight $\wtt(e)$ also defines a line bundle $\mathcal{L}_{e}$ on $\BT$. We will sometimes abuse notation and write
\begin{equation}
    \wtt(e)\ =\ c_{1}(\mathcal{L}_{e})\ =\ \mathbf{t}(e)\ \in \ \HH^{2}(\BT,\Qb).
\end{equation}

\subsubsection{Tautological bundles and $T$-equivariant moduli space} We then  define
 \[ \MSd[T]{Q}{d}  \ = \  \Rep_{d}(Q)/(\GL_{d}\times T). \] The \textbf{tautological bundle} $\El_{i,d}$ for vertex $i$ is the pullback of the universal vector bundle along 
$$\Ml^{T}_{Q,d} \ \to \ \B \GL_{d} \times \BT\ \to\  \B \GL_{d_{i}} \times \BT\ \to\ \B \GL_{d_{i}}.$$
Its fibre over a point corresponding to a representation $\rho$ is the vector space $V_i$. The \textit{tautological map} of vector bundles $\rho_e: \El_i \to \El_j$ attached to any edge $e$ is induced by the map 
$$\Ml^{T}_{Q,d} \ \to \ \Hom(\Cb^{d_i},\Cb^{d_j})/ \GL_{d_i}\times \GL_{d_j}.$$
Its fibre over a point is the linear map $\rho_e: V_i \to V_j$.

\subsubsection{Ext complex} \label{theta_quiver} The \textbf{Ext complex} is the two-term complex of vector bundles on $\Ml^{T}_{Q,d}\times_{\BT} \Ml^{T}_{Q,d'}$, given by 
\begin{equation}
  \label{eqn:ExtComplex}
   \Ext_{d,d'} \ = \ \left(\bigoplus_i \El_{i,d}^\vee \boxtimes \El_{i,d^{'}} \ \to \ \bigoplus_{e \ : \  i \to j} \El_{i,d}^\vee \boxtimes \El_{j,d^{'}} \ \otimes\ \Ll_{e}\right).
\end{equation}
We denote the zeroeth and first factors by $\Ext_0$ and $\Ext_1$. The differential is given by viewing the left summands as the the vector bundle of maps from $\El_i \boxtimes 1$ to $1 \boxtimes \El_i$, and composing with the tautological maps $\rho_e$ for each edge.

We will often have to dualise and also have to pull back the $\Ext$-complex by the swap map $\sigma : \Ml^{T}_{Q,d} \times_{\BT} \Ml^{T}_{Q,d^{'}} \to \Ml^{T}_{Q,d^{'}} \times_{\BT} \Ml^{T}_{Q,d}$; for this we note that $\Ll_e^\vee$ corresponds to the character $-\wtt(e)$, and $\sigma^*\Ll_e=\Ll_e$.

\subsubsection{Relation to the potential} The weight function induces a $N$-grading on $\Cb Q$, sending a path to the sum of the weights of its edges. We say the potential is \textit{invariant} for the $T$ action if $W$ has weight zero under this grading: $\wtt(W)=0$. 

\subsection{Cohomology of representation stacks}

\subsubsection{} The cohomology of the stack of quiver representations is a product over connected components
\begin{equation}
    \Ht^{\sbt}(\Ml_{Q}^T)\ =\ \bigsqcap_{d \in \Nb^{Q_{0}}} \Ht^{\sbt}(\Ml_{Q,d}^T),
\end{equation}
where each factor is generated by the chern classes of its tautological bundles: 
\begin{align*}
   \Ht^{\sbt}(\Ml_{Q,d}^T) & \ \simeq \ \Ht^{\sbt}_T(\BGL_d)\\
   & \ \simeq \ \Ht^{\sbt}(\BT)[c_r(\El_i)\ : \ i\in Q_0,\, 1\le r\le d_i]\\
   & \ \simeq \ \Sym_{\Ht^{\sbt}(\BT)}\left\{x_{i,\alpha}\ : \ i\in Q_0, \, 1\le \alpha \le d_i\right\}
\end{align*}
where $x_{i,\alpha}$ is a chern root of the tautological bundle $\El_i$.

\subsubsection{Euler classes of tautological bundles} The Euler class of the tautological bundle $\El_i$  is $e(\El_{i,d})  =  \bigsqcap_{n=1}^{d_i}x_{i,n}$. Moreover, we have formulas that agree with the euler class formulas in \cite[Section 4.1]{BD}
\begin{prop} \label{prop:euler_ext_to_euler_ben}
The Euler classes $e(\Ext_0)_{d,d'}$ and $e(\Ext_1)_{d,d'}$ are equal to\footnote{Note that here we have $+\mathbf{t}(e) = + \wtt(e)$ in the definition of $e(Q_1)$, as opposed to $- \mathbf{t}(e)$ from \cite{BD}. }
\[ e_{d,d^{'}}(Q_1) \ \defeq \ \bigsqcap_{e \, : \, i\to j} \bigsqcap_{(n,m)=(1,1)}^{(d_i,d_j')} (1\otimes x_{j,m} - x_{i,n}\otimes 1 + \wtt(e))\ \in\ \ \Ht^{\sbt}(\Ml^{T}_{Q,d}) \otimes_{\Ht^{\sbt}(\BT)}\Ht^{\sbt}(\Ml^{T}_{Q,d^{'}})\] and 
\[ e_{d,d^{'}}(Q_0) \ \defeq \ \bigsqcap_{i \in Q_0} \bigsqcap_{(n,m)=(1,1)}^{(d_i,d_i')} (1\otimes x_{i,m} - x_{i,n}\otimes 1 )\ \in\ \Ht^{\sbt}(\Ml^{T}_{Q,d}) \otimes_{\Ht^{\sbt}(\BT)}\Ht^{\sbt}(\Ml^{T}_{Q,d^{'}})  \] 
respectively.
\end{prop}
\begin{proof}
Easily follows from additivity of chern roots under tensor products. For instance,
\begin{align*}
e \left(\bigoplus_{i \in Q_{0}} \El_{i,d}^\vee \boxtimes \El_{i,d^{'}} \right)  &\ = \ \bigsqcap_{i \in Q_0} e(\El_{i,d}^\vee \boxtimes \El_{i,d^{'}} ) \ =\ \bigsqcap_{i \in Q_0} \bigsqcap_{(n,m)=(1,1)}^{(d_i,d_i')}  (1\otimes x_{i,m} - x_{i,n}\otimes 1 )
\end{align*} 
where the last equality follows since the chern roots are negated upon dualising. Similarly 
\begin{align*}
e \left( \bigoplus_{e \in i \to j \in Q_{1}}  \El_{i,d}^\vee \boxtimes \El_{j,d^{'}} \ \otimes\ \Ll_{e} \right) &\ =\ \bigsqcap_{e \, : \, i \to j } e( \El_{i,d}^\vee \boxtimes \El_{j,d^{'}} \ \otimes\ \Ll_{e}) \\ 
 &\ =\ \bigsqcap_{e \, : \, i\to j} \bigsqcap_{(n,m)=(1,1)}^{(d_i,d_j')}  (1\otimes x_{j,m} - x_{i,n}\otimes 1 + \wtt(e)
\end{align*} 
\end{proof}

\subsubsection{Underlying monoidal structure} \label{ssec:twisted_monoidal} We write $\Vect_{\Lambda}^T$ for the category of vector spaces with a ``cohomological'' $\Zb$-grading, a grading by the cone
$$\Lambda \ = \  \pi_0(\Ml_Q) \ = \ \Nb^{Q_0},$$
 and with an action of $\Ht^{\sbt}(\BT)$, with trivial $\Lambda$-grading and cohomologically graded in the usual way.

For any pair of such objects 
$$V \ = \ \bigoplus_{(d_1,n_1) \in \Lambda\times \Zb}V_{(d_1,n_1)}, \hspace{15mm} W \ = \ \bigoplus_{(d_2,n_2)\in \Lambda \times \Zb}W_{(d_2,n_2)}$$
we define their \textit{twisted tensor product} by
\begin{equation}
  \label{eqn:OtimesTDefinition} 
    V\otimes_T W \ =\ \bigoplus_{(d,n) \in \Lambda\times \Zb} \bigoplus_{(d,n)=(d_1+d_2,n_1+n_2)}  V_{(d_{1},n_1)} \otimes_{\Ht^{\sbt}(\BT)} W_{(d_{2},n_2)} [\chi_{Q}(d_{1},d_{2})-\chi_{Q}(d_{2},d_{1})]
\end{equation}
where $\chi_{Q}$ is the Euler form of the quiver and the shift is by the cohomological $\Zb$ degree. It is neither symmetric nor braided monoidal in general, see section \ref{ssec:BraidingsBialgebras}. If the quiver $Q$ is symmetric, then the cohomological shift is trivial and the monoidal structure is symmetric.

\subsection{Cohomological Hall algebra structure} \label{ssec:coha_prelims}

\subsubsection{}

Consider the correspondence of $T$-equivariant stacks
\begin{equation}\label{eqn:CoHACorrespondence}
\begin{tikzcd}[row sep = {30pt,between origins}, column sep = {55pt, between origins}]
 &\SES^{T}_{Q,d_1,d_2}\ar[rd,"p"]\ar[ld,,"q"'] & \\ 
\Ml^{T}_{Q,d_1} \times_{\BT} \Ml^{T}_{Q,d_2} & &\Ml^{T}_{Q,d_1+d_2} 
\end{tikzcd}
\end{equation}
Then the cohomological Hall algebra structure on
\begin{equation}
  \Al^{T}_{Q,W}\  =\ \bigoplus_{d \in \Nb^{Q_{0}}} \Ht^{\sbt}(\Ml^{T}_{Q,d}, \varphi_{W,d} \Qb_{\Ml_{Q,d}}[\dim \Ml_{Q,d}]) 
\end{equation}
 is defined as in \cite{BD,KS,RSYZ} by induction along the above diagram, using the formula
\begin{align*}
    m = p_{*}q^{*}
\end{align*}
induced by the functoriality of vanishing cycles in the Appendix \ref{sec:app_cohomology}.
Here we suppress the Thom--Sebastiani and isomorphism from the notation. This induces a map $m: \mathcal{A}^{T}_{Q,W} \otimes_{T} \mathcal{A}^T_{Q,W} \to \mathcal{A}^T_{Q,W}$. 
\subsubsection{Remark} The  CoHA $\Al^T_{Q,W}$ will be $\Lambda$ graded by connected component, but the cohomology $\Ht^{\sbt}(\Ml^T_Q)$ or tautological ring $\Ht^{\sbt}(\Ml_Q^T)_{\textup{taut}}$, defined in section \ref{ssec:tautring}, will be defined in trivial graded degree $0\in \Lambda$.

\subsection{Vertex coalgebras}\label{sec:VA}

\subsubsection{} \label{vertex_def} We define now define a version of coassociative coalgebra whose coproduct depends on a point $z$ on the complex plane ``meromorphically", due to Borcherds and Hubbard \cite[Def. 2.1]{Hu}.

A \textbf{nonlocal (alias coassociative) vertex coalgebra} is a vector space $V$ together with a \textit{translation} endomorphism, \textit{covacuum} covector and a \textit{cofield} map  
\begin{equation}\label{eqn:VertexCoAlg}
    T \ : \ V \ \to \ V , \hspace{15mm} \epsilon \ : \ V \ \to \ k, \hspace{15mm}
    \Delta(z) \ : \ V \ \to \ V\otimes V ((z^{-1}))
\end{equation}
satisfying:
\begin{itemize}
 \item Compatibility with the translation,
    \begin{equation}\label{eqn:Translate}
         \Delta(z) \cdot T  \ = \ \frac{d}{dz} \Delta(z) + (1 \otimes T) \cdot \Delta(z).
    \end{equation}
    \item Vertex coassociativity (also known as weak coassociativity),
\begin{equation} \label{eqn:VertexCoass}
    (\Delta(z) \otimes \id ) \Delta(w) \ = \  (\textup{id} \otimes \Delta(w))\Delta(z+w)
\end{equation}
\item Compatibility with the covacuum: we have $(\epsilon \otimes \id)\cdot\Delta(z)=\id$ and $(\id\otimes \epsilon)\cdot \Delta(z)=\id + \Ol(z)$ is $V[z]$-valued.
\end{itemize}

Let $\Lambda$ be a commutative monoid equipped with a bilinear form $\chi$ and $R$ be a $\Zb$-graded commutative ring with augmentation. A \textbf{$\Lambda$-graded nonlocal (alias coassociative) vertex coalgebra} is a $\Lambda\times \Zb$-graded $R$-module 
$$V \ = \ \bigoplus_{(\lambda,n)\in \Lambda\times \Zb} V_{(\lambda, n)}$$
with the structure of a nonlocal vertex algebra as above with $\otimes$ replaced with 
$$\otimes_R[\chi(\lambda, \mu) - \chi(\mu,\lambda)]$$
as in \eqref{eqn:OtimesTDefinition}. We equip $z$ with $\Lambda\times \Zb$-degree $|z|=(0,2)$ so that the maps \eqref{eqn:VertexCoAlg} are $R$-linear with degree $(0,-2),(0,0),(0,0)$ respectively. We call the $\Zb$-degree \textit{cohomological}.
We write 
\begin{equation}
    \Delta_{\lambda_{1},\lambda_{2}}(z) \ : \ V_{\lambda_{1}+\lambda_{2}} \ \to \ V_{\lambda_{1}} \otimes_{R} V_{\lambda_{2}} ((z^{-1}))
\end{equation}
for the associated map of $R$-modules on graded pieces, which has cohomological degree $\chi(\lambda_2,\lambda_1)-\chi(\lambda_1,\lambda_2)$, and \eqref{eqn:VertexCoass} is equivalent to
\begin{equation} \label{coassoc_eq}
    (\Delta_{\lambda_{1},\lambda_{2}}(z) \otimes \id ) \Delta_{\lambda_{1}+\lambda_{2},\lambda_{3}}(w) \ = \  (\textup{id} \otimes \Delta_{\lambda_{2},\lambda_{3}}(w))\Delta_{\lambda_{1},\lambda_{2}+\lambda_{3}}(z+w).
\end{equation}

 In our setting, we will work with $\Lambda=(\Nb^{Q_0},+)$ the dimension lattice of a quiver $Q$, where $+$ is the sum operation on dimension vectors. Our commutative ring will be $R = \Ht^{\sbt}(\BT)$ with the twisted monoidal structure as in subsection \ref{ssec:twisted_monoidal}.

\subsubsection{Remark} The analogue of cocommutativity for vertex coalgebras is colocality (also called the Jacobi identity), see for instance part ($6$) of Theorem \ref{thm:CoHABialgebra}.

\subsubsection{Holomorphicity} \label{sssec:Hol} A coassociative vertex coalgebra is called \textit{holomorphic} if the cofield map factors as
$$\Delta(z) \ : \ V \ \to \ V \otimes_{R} V [z] \ \subseteq \ V \otimes_{R} V ((z^{-1})),$$
 i.e. if it ``has no poles''. Just as in \cite[$\S$1.4]{FBZ}, we have

\begin{lem} \label{lem:holomorphic_covertex}
If $V$ is a coalgebra with locally nilpotent coderivation $T$, i.e. 
\begin{equation}\label{eqn:Coderivation}
    (T \otimes \id) \cdot \Delta + (\id\otimes T) \cdot  \Delta\ =\ \Delta \cdot T
\end{equation}
and $T^nv =0$ for $n \gg 0$, then $\Delta(z)\ =\ (e^{zT}\otimes \id)\Delta$ defines a holomorphic vertex coalgebra, and every holomorphic vertex coalgebra takes this form.
\end{lem}
\begin{proof}
We show that the assignment
\begin{align*}
   \Delta & \ \mapsto \  \Delta(z)\ =\ (e^{zT}\otimes \id)\Delta\\
   \Delta(0)&\ \mapsfrom \ \Delta(z)
\end{align*}
sets up a bijection between such data and  holomorphic vertex coalgebras. To do this, it is enough to show that for any holomorphic vertex coalgebra we have that 
$$\Delta(z) \ = \ (e^{zT}\otimes \id)\cdot \Delta(0)$$
which follows from \eqref{eqn:Translate}, as their $z^n$ coefficients agree, since
$$\left( \frac{d}{dz}\right)^n\Delta(z)\vert_{z=0} \ = \ \ad_T^n\Delta(0)\ = \ \left( \frac{d}{dz}\right)^n(e^{zT}\otimes \id)\cdot\Delta(0)\vert_{z=0}$$
where $\ad_T(-)=(-)\cdot T - (\id\otimes T)\cdot (-)$, and for the second equality we used the coderivation property \eqref{eqn:Coderivation}.
\end{proof}

\newpage
\section{Joyce--Liu vertex coalgebra for quivers with potential} \label{sec:JoyceCoproduct}
\noindent
Let $Q$ be an arbitrary quiver with an arbitrary potential $W$, which we allow to be graded by the character lattice $N$ of a torus $T$. In this section, we will construct a nonlocal vertex coalgebra structure on the cohomological Hall algebra
\begin{equation}
    \Al_{Q,W}^T = \bigoplus_{d \in \Nb^{Q_{0}}} \Ht^{\sbt}(\Ml^{T}_{Q,d}, \varphi_{W,d} \Qb_{\Ml_{Q,d}}[\dim \Ml_{Q,d}]),
\end{equation}
under a relative K{\"u}nneth assumption.

\subsection{Geometric structures on quiver moduli stacks} \label{ssec:GeometricStructures}

\subsubsection{} 
 The vertex coproduct structure will be induced from the following geometric structures on the moduli stack and a function $\tr W : \Ml_Q^T \to \Ab^1$ coming from the potential:
  \begin{enumerate}
  \item The direct sum of quiver representations gives a map
  \begin{equation*}
      \Ml_{Q}^T \times_{\BT} \Ml_{Q}^T \ \stackrel{\oplus}{\to} \ \Ml_{Q}^T
  \end{equation*}
  which is commutative, associative, and the $0$ representation $\pt \xrightarrow{0} \Ml_{Q}^T$ defines a unit.
   \item The direct sum and potential are compatible in the sense that
   \begin{equation} \label{OplusWCompatibility}
       \oplus\cdot\tr W \ = \ \tr W\boxplus \tr W, \hspace{15mm} 0\cdot \tr W \ = \ 0.
   \end{equation}
   \item There is an action 
   \begin{equation}
       \act \ : \ \BGm\times\Ml_{Q}^T\ \to\ \Ml_{Q}^T
   \end{equation}
   of the group stack $\BGm$ compatible with the above structures: the above maps are $\BGm$-equivariant, and $\act\cdot \tr W  = 0 \boxplus \tr W$.
   \item We have a perfect complex
   \begin{equation}
       \Ext \ \in \  \Perf(\Ml^T_Q \times_{\BT} \Ml^T_Q)
   \end{equation} 
    satisfying
    \begin{align*}
        (\oplus \times \id)^* \Ext & \simeq \Ext_{13} \oplus \Ext_{23} &
        ( \id \times \oplus)^* \Ext & \simeq \Ext_{12} \oplus \Ext_{13}\\
        \act^{*}_{1} \Ext & \simeq \gamma^{-1} \boxtimes \Ext &
        \act^{*}_{2} \Ext & \simeq \gamma \boxtimes \Ext 
    \end{align*}
    where $\Ext_{ij} = \pi^{*}_{ij} \Ext$. 
\end{enumerate}

For trivial torus $T$ in the quiver case, the direct sum map (1) can be viewed as arising from contravariance of the moduli of objects functor of \cite{TV}. Indeed, $\Ml_{Q}$ is an open substack of the moduli of objects $\Ml_{\Dt(\Rep Q)}$ of the dg category $\Dt(\Rep Q)$ of quiver representations. The direct sum map on $\Ml_{\Dt(\Rep Q)}$ is induced by the diagonal map $\Dt(\Rep Q) \to \Dt(\Rep Q) \times \Dt(\Rep Q) \simeq \Dt(\Rep Q) \sqcup \Dt(\Rep Q)$, and the direct sum on $\Ml_{Q}$ is given by restriction. 

For (2), the function $\tr W$ is then given by restriction of a function on $ \Ml_{\Dt(\Rep Q)}$ induced by a Hochschild homology class of the category $\Dt(\Rep Q)$. Then the compatibility with direct sum pullback in $(2)$ follows again by restriction from the discussion in \cite[Prop 8.43]{kinjo2024coha}.

Loosely speaking, (3) and (4) follow  since every point of $\Ml_{Q}$ contains a $\Gb_{m}$ in its stabiliser and by setting $\Ext$ to be given by $\textup{RHom}(-,-)$. More precisely, it is the pullback along $q$ of the tangent complex of the map $p$ in \eqref{eqn:CoHACorrespondence}, with explicit description given in equation \eqref{eqn:ExtComplex}.

\subsubsection{K\"unneth isomorphism assumption}  \label{sssec:kunneth_assumpt} We need to assume
\begin{equation}
      \label{eqn:KunnethAssumption} 
      \Ht^{\sbt}(\Ml^{T}_{Q,d_{1]}}\times_{\BT}\Ml^T_{Q,d_{2}},\varphi_{\tr W_{d_{1}}} \boxtimes \varphi_{\tr W_{d_{2}}}) \ \simeq \ \Ht^{\sbt}(\Ml^{T}_{Q,d_{1}},\varphi_{\tr W_{d_{1}}}) \otimes_{\Ht^{\sbt}(\BT)} \Ht^{\sbt}(\Ml^T_{Q,d_{2}},\varphi_{\tr W_{d_{2}}}).
    \end{equation}
This is trivially true when $T=1$ by the global K{\"u}nneth isomorphism, but as discussed in appendix \ref{ssec:Kunneth} the relative K{\"u}nneth formula is in general not true. 

Nevertheless, \eqref{eqn:KunnethAssumption} is known to hold for tripled quivers with canonical tripled potential and appropriate torus action the assumption is also known by \cite[Thm. 9.6]{Da1}. It is also known in the case of $W=0$ and more generally by the argument in \cite[Thm. 3.4]{davison2022affine}, which holds whenever $\mathcal{A}_{Q,W}$ has a pure mixed Hodge structure.

\subsection{Structures on critical cohomology}\label{ssec:StructuresCriticalCohomology}
We use the geometric structures on $\Ml^{T}_{Q}$ to induce structures on vanishing cycles cohomology. The structures will be induced using the vanishing cycles toolkit in the Appendix \ref{sec:app_cohomology}.
\subsubsection{Direct sum} The direct sum map on the moduli stack induces a map 
$$\oplus^* \ : \ \Ht^{\sbt}(\Ml_Q)\ \to \ \Ht^{\sbt}(\Ml_Q \times \Ml_Q).$$
The K\"unneth isomorphism does not apply here since $\Ml_Q$ has infinitely many connected components, but restricting to the component labelled by a pair of dimension vectors $d_1,d_2$ it does, and we have a map
$$\oplus^*_{d_1,d_2} \ : \ \Ht^{\sbt}(\Ml_{Q,d_1 + d_2}) \ \to \ \Ht^{\sbt}(\Ml_{Q,d_1}) \otimes \Ht^{\sbt}(\Ml_{Q,d_2}).$$
We now apply a similar argument to equivariant critical cohomology.

\begin{lem}\label{lem:HolomorphicJoyceCoproduct}
  Assume we have a graded quiver with potential that satisfies Assumption \ref{sssec:kunneth_assumpt}. There is a cocommutative coproduct with coderivation 
  $$\oplus^* \ : \ \Al^T_{Q,W} \ \to \ \Al^T_{Q,W}\otimes_{\Ht^{\sbt}(\BT)} \Al^T_{Q,W}[\Delta], \hspace{15mm} T \ : \ \Al^T_{Q,W}\ \to \ \Al^T_{Q,W}[2],$$
  where we shift by cohomological degree $\Delta = \chi(d_1)+\chi(d_2)-\chi(d_1 + d_2)$.
\end{lem}
\begin{proof} 
To simplify notation, we write $\Qb_{d}$ for $\Qb_{\Ml_{Q,d}^T}[\chi(d)]$. We have the following chain of maps
   \begin{align*}
       \varphi_{W_{d_{1}+d_{2}}} \Qb_{d_1+d_2} &\ \to\ \oplus_{*} \oplus^{*}\varphi_{W_{d_{1}+d_{2}}} \Qb_{d_1+d_2}  \\ 
       &\ \to \ \oplus_{*} \varphi_{\oplus^{*} W_{d_{1}+d_{2}}} \Qb_{d_1+d_2}   \\
      & \ = \ \oplus_{*} \varphi_{W_{d_{1}} \boxplus W_{d_2}} \Qb_{d_1+d_2}  \\
      & \ \simeq \ \oplus^{*}(\varphi_{W_{d_{1}}}\Qb_{d_1} \boxtimes \varphi_{W_{d_{2}}}\Qb_{d_2}).
   \end{align*}
   of sheaves on $\Ml^{T}_{Q,d_{1}+d_{2}}$. The first is the unit for the $(\oplus^*,\oplus_*)$ adjunction, the second is the functoriality of vanishing cycles under pullback, the third is compatibility between $W$ and the direct sum map, and the last is the Thom-Sebastiani isomorphism for vanishing cycles. Taking derived global sections gives 
    \begin{equation*} \label{eqn:DirectSumCriticalCohomology}
    \oplus^*_{d_{1}, d_{2}} : \Ht^{\sbt}(\Ml^{T}_{Q,d_{1}+d_{2}},\varphi_{W_{d_{1}+d_{2}}}) \to   \Ht^{\sbt}(\Ml^{T}_{Q,d_{1}},\varphi_{W_{d_{1}}}) \otimes_T \Ht^{\sbt}(\Ml^{T}_{Q,d_{2}},\varphi_{W_{d_{2}}})[\chi(d_1)+\chi(d_2)-\chi(d_1 + d_2)]
    \end{equation*}
    by the K\"unneth isomorphism assumption. Coassociativity and cocommutativity are inherited from associativity and commutativity of $\oplus$. For instance, it is coassociative because the diagram of stacks
   \begin{center}
   \begin{tikzcd}[row sep = {40pt,between origins}, column sep = {30pt}]
   {\Ml^{T}_{Q,d_{1}} \times_{\BT} \Ml^{T}_{Q,d_{2}} \times_{\BT} \Ml^{T}_{Q,d_{3}}} \ar[r,"\oplus\times \id"] \ar[d,"\id\times \oplus"]  & {\Ml^{T}_{Q,d_{1}+ d_{2}} \times_{\BT} \Ml^{T}_{Q,d_{3}}} \ar[d,"\oplus"] \\ 
   {\Ml^{T}_{Q,d_{1}} \times_{\BT} \Ml^{T}_{Q,d_{2}+ d_{3}}}\ar[r,"\oplus"]  & {\Ml^{T}_{Q,d_{1}+ d_{2}+ d_{3}}}
   \end{tikzcd}
   \end{center}
   commutes. Using the map $0:\pt \to  \Ml^{T}_{Q}$ we construct the counit similarly.

   We apply an identical argument to the action by $\BGm$ to get the following map of sheaves on $\BGm \times \Ml^T_{Q,d}$, where the only difference is we use the compatibility between $W$ and the action map $\act$.
   \begin{align*}
     \varphi_{W_d}& \ \to \ \act_* \act^* \varphi_{W_d} \ \to \ \act_* \varphi_{\act^{*} W_d} \ = \ \act_* \varphi_{0\, \boxplus W_d}   \ \simeq \ \act_*(\Qb_{\Ml^T_{Q}} \boxtimes \varphi_{W_d}).
   \end{align*}
   This induces a map 
    \begin{equation*}
        \act^{*} \ : \  \Ht^{\sbt}(\Ml_{Q,d}, \varphi_{W_d}) \ \to \ \Ht^{\sbt}(\BGm) \otimes \Ht^{\sbt}(\Ml_{Q,d}, \varphi_{W,d})  \ = \ \Cb[z]\otimes \Ht^{\sbt}(\Ml_{Q,d}, \varphi_{W,d})
    \end{equation*}
    where $z$ is the first chern class of the tautological line bundle on $\BGm$. Since $\act$ is a group stack action, this defines an action for the coalgebra $\Cb[z]$. Moreover, $\oplus^*$ is linear for this coaction, as follows from the commutative diagram 
    \begin{center}
    \begin{tikzcd}[row sep = {40pt,between origins}, column sep = {30pt}]
    {\Ml^{T}_{d_{1}} \times_{\BT} \Ml^{T}_{d_{2}} \times \B \Gb_{m}} \ar[r,"\id\times \Delta"] \ar[d,"\oplus \times \id"]  &[-10pt]{\Ml^{T}_{Q,d_{1}} \times_{\BT} \Ml^{T}_{Q,d_{2}} \times \B \Gb_{m} \times \B \Gb_{m}}\ar[r,"\act\times \act"]  &[10pt] {\Ml^{T}_{Q,d_{1}} \times_{\BT} \Ml^{T}_{Q,d_{2}}} \ar[d,"\oplus"] \\ 
    {\Ml^{T}_{d_{1}+d_{2}} \times \B \Gb_{m}} \ar[rr,"\act"] && {\Ml^{T}_{d_{1}+d_{2}}}
    \end{tikzcd}
    \end{center}
    Thus $\Al_{Q,W}^T$ defines a coalgebra inside $\Ht^{\sbt}(\BGm)\CoMd$. We thus define the endomorphism $T$ to be the coefficient of $z$ in the action map $\act^*=\exp(Tz)$, which is a coderivation as $z$ is primitive.
\end{proof}
\subsubsection{$W= 0$ case} \label{ssec:coprod_on_0pot}
In particular, by taking $W=0$ we get a $\Cb[z]$-coaction and coproduct on ordinary cohomology 
$$\Al^T_Q \ = \ \bigoplus_{d \in \Lambda} \Ht^{\sbt}(\Ml^T_{Q,d})[\dim \Ml_{Q,d}]$$
which we also denote by $\act^*$ and $\oplus^*$. Note that $\Al^T_Q$ acts on $\Al^T_{Q,W}$ by cup product, and this algebra and module structure is compatible with the coproduct and $\Cb[z]$-coaction as follows.
 \begin{lem} \label{lem:linearity_two_acts}
   Denote the coproducts and translations on $\mathcal{A}^{T}_{Q,0}$ and $\mathcal{A}^{T}_{Q,W}$  using the subscript $0$ and $W$ respectively. They are compatible as follows
   \begin{align*}
       \act^*_{W}(h \cdot b) \ & = \ \act^*(h)_{0} \cdot \act^{*}_{W}( b)  \\
        \oplus^*_{W}(h \cdot b) \ & = \ \oplus^*_0(h) \cdot \oplus^{*}_{W}( b)
   \end{align*}
   for each cohomology class $h \in \Al^{T}_{Q,0}$ and critical cohomology class $b \in \Al^T_{Q,W}$. 
 \end{lem}
\begin{proof}
    The compatibility of direct sum with cup product follows from the commutative diagram
    \begin{center}
      \begin{center}
      \begin{tikzcd}[row sep = {40pt,between origins}, column sep = {20pt}]
      \Ml_Q^T\times_{\BT}\Ml_Q^T \ar[r,"\Delta\times\Delta"] \ar[d,"\oplus"] &  \Ml_Q^T\times_{\BT}\Ml_Q^T\times_{\BT}\Ml_Q^T\times_{\BT}\Ml_Q^T \ar[d,"\oplus\times \oplus"] \\ 
       \Ml_Q^T \ar[r,"\Delta"]& \Ml_Q^T\times_{\BT}\Ml_Q^T
      \end{tikzcd}
      \end{center} 
    \end{center}
    and compatibility of act maps from

    \begin{center}
      \begin{center}
      \begin{tikzcd}[row sep = {40pt,between origins}, column sep = {20pt}]
      \BGm\times \Ml_Q^T\times_{\BT}\Ml_Q^T \ar[r,"\Delta\times\Delta"] \ar[d,"\act_1"] & (\BGm\times \Ml_Q^T\times_{\BT}\Ml_Q^T)^{\times_{\BT}2} \ar[d,"\act_1\times \act_1"] \\ 
       \Ml_Q^T\times_{\BT}\Ml_Q^T \ar[r,"\Delta"]&  ( \Ml_Q^T\times_{\BT}\Ml_Q^T)^{\times_{\BT}2} 
      \end{tikzcd}
      \end{center} 
    \end{center}
\end{proof}
We usually omit the subscripts $0$ and $W$ on coproducts and translations.

\subsubsection{Joyce--Borcherds twist}
We next introduce the \textit{Joyce--Borcherds twist}
\begin{equation}
   \Psi(\Ext_{d_{1},d_{2}},z)\ =\ \sum_{k\ge 0}z^{\rank \theta-k}c_k(\Ext_{d_{1},d_{2}}) \in \Ht^{\sbt}(\Ml^{T}_{Q,d_{1}}) \otimes_{\Ht^{\sbt}(\BT)} (\Ht^{\sbt}(\Ml^{T}_{Q,d_{2}})((z^{-1})), 
\end{equation}
and sometimes write $\Psi(\Ext,z)$ to emphasise the role of the perfect complex $\Ext$.
The following algebraic properties will follow from compatibilities of the complex $\Ext$ with respect to the different operations on our moduli stack $\Ml^{T}_{Q}$. We will also use properties of chern series of perfect complexes as in Appendix \ref{sec:appendix}. 
    \begin{prop} \label{jb_prop}
  The Joyce--Borcherds twist satisfies 
$$(\oplus_{d_1,d_2}^*\times \id)\Psi(\Ext_{d_1+d_2,d_3},z) \ = \ \Psi(\Ext_{d_1,d_3},z)\Psi(\Ext_{d_2,d_3},z)$$
$$(\id \times \oplus_{d_2,d_3}^*)\Psi(\Ext_{d_1,d_2+d_3},z) \ = \ \Psi(\Ext_{d_1,d_2},z)\Psi(\Ext_{d_1,d_3},z)$$
and 
$$\act_{1,w}^*\Psi(\Ext_{d_1,d_2},z) \ = \ \Psi(\Ext_{d_1,d_2},z-w), \hspace{15mm} \act_{2,w}^*\Psi(\Ext_{d_1,d_2},z) \ = \ \Psi(\Ext_{d_1,d_2},z+w).$$
\end{prop}
\begin{proof}
These properties all follow from the identities
$$(\oplus_{d_1,d_2}^*\times \id)\Ext_{d_1+d_2,d_3} \ = \ \Ext_{d_1,d_3}\oplus \Ext_{d_2,d_3}$$
$$(\id \times \oplus_{d_2,d_3}^*)\Ext_{d_1,d_2+d_3} \ = \ \Ext_{d_1,d_2}\oplus\Ext_{d_1,d_3}$$
and 
$$\act_{1}^*\Ext_{d_1,d_2} \ = \ \gamma^{-1}\boxtimes \Ext_{d_1,d_2}, \hspace{15mm} \act_{2}^*\Ext_{d_1,d_2} \ = \ \gamma \boxtimes \Ext_{d_1,d_2},$$
where we apply Lemma \ref{lem:psi_properties} on how $\Psi(-,z)$ interacts with direct sums and tensoring with line bundles.
\end{proof}

\subsection{Main construction}

\begin{theorem}\label{quiver_potential_jl}\label{thm:JoyceVCA}
Let $Q$ be an $N$-graded quiver with invariant potential $W$, which satisfies the K\"unneth assumption \eqref{eqn:KunnethAssumption}. Then
\begin{align*}
     \Delta(z) \ : \  \AQW[T]{Q}{W} \ & \to \  \AQW[T]{Q}{W} \otimes_T \AQW[T]{Q}{W}  (( z^{-1}))  \\
     \alpha \ &\mapsto \  \Psi(\Ext,-z) \cdot \act^{*}_{1} \oplus^{*}(\alpha)
\end{align*}
defines a coassociative vertex coproduct linear over $\Ht^{\sbt}(\BT)$.
\end{theorem}

In fact $\Delta(z)$ is also braided colocal--the vertex analogue of braided cocommutativity--see the last part of Theorem \ref{thm:CoHABialgebra}. Note that we twist by the $\Ext$ complex of the ambient smooth stack $\Ml_Q$, not of the Calabi--Yau-three moduli stack $\Ml_{Q,W}=\Crit(\textup{Tr}W)$ inside it, as one might first have guessed. 
\begin{proof}
The counit will be induced by vanishing cycle pullback along the inclusion $\pt \xrightarrow{0} \Ml^{T}_{Q}$. The compatibility with translation will follow from \ref{jb_prop}, \ref{ssec:coprod_on_0pot} and that $\act^{*}_{z} = e^{zT}$, so $\frac{d}{dz} \act^{*}_{z,1} = (T \otimes 1) \act^{*}_{z,1}$. \par
   In this proof we will freely use the compatibilities in Lemma \ref{lem:linearity_two_acts}. We now give a direct proof of vertex coassociativity
  \begin{equation}\label{eqn:JoyceVertexAssociativity}
    (\Delta_{ d_{1}, d_{2}}(z)\otimes \id)\Delta_{d_{1} + d_{2}, d_{3}}(w) = (\id \otimes \Delta_{ d_{2}, d_{3}}(w))\Delta_{d_{1}, d_{2} + d_{3}}(z+w)
 \end{equation}
  following Liu \cite[3.2.13]{liu2022multiplicative}, and using the shorthand $\Psi(z)=\Psi(\Ext,z)$. Let us expand both sides starting with the left 
\begin{align*}
  & (\Psi(-z)_{d_1,d_2}\otimes \id) \act_{z,1}^* (\oplus_{d_1,d_2}^* \otimes \id)\ \left( \Psi(-w)_{d_1+d_2,d_3}\act_{w,1}^*\oplus^*_{d_1+d_2,d_3}\right) \\
  & \hspace{15mm} \ = \ (\Psi(-z)_{d_1,d_2}\otimes \id) \act_{z,1}^*\Psi(-w)_{d_1,d_3}\Psi(-w)_{d_2,d_3} (\oplus_{d_1,d_2}^* \otimes \id) \left( \act_{w,1}^*\oplus^*_{d_1+d_2,d_3}\right)\\ 
  & \hspace{15mm} \ = \ \Psi(-z)_{d_1,d_2}\Psi(-z-w)_{d_1,d_3}\Psi(-w)_{d_2,d_3}\  \act_{z,1}^* (\oplus_{d_1,d_2}^* \otimes \id)  \act_{w,1}^*\oplus^*_{d_1+d_2,d_3}
\end{align*}
where we have used the first and second parts of Proposition \ref{jb_prop} in the first and second equality. Similarly, we can expand the right hand side to
\begin{align*}
    & (\id \otimes \Psi(-w)_{d_2,d_3}) \act_{w,1}^* (\id\otimes \oplus_{d_2,d_3}^*)\ \left( \Psi(-z-w)_{d_1,d_2+d_3}\act_{z+w,1}^*\oplus^*_{d_1,d_2+d_3}\right) \\
  & \hspace{15mm} 
  \ = \  (\id \otimes \Psi(-w)_{d_2,d_3}) \act_{w,1}^*\Psi(-z-w)_{d_1,d_3}\Psi(-z-w)_{d_1,d_2} (\id\otimes \oplus_{d_2,d_3}^*) \left( \act_{z+w,1}^*\oplus^*_{d_1,d_2+d_3}\right) \\
  & \hspace{15mm} \ = \  \Psi(-w)_{d_2,d_3} \Psi(-z-w)_{d_1,d_3}\Psi(-z)_{d_1,d_2}\  \act_{w,1}^*(\id\otimes \oplus_{d_2,d_3}^*)  \act_{z+w,1}^*\oplus^*_{d_1,d_2+d_3}.
\end{align*} 
The $\Psi$ factors on left and right are the same, so it remains to show that 
$$  \act^{*}_{z,1}(\oplus^{*}_{d_{1},d_{2}} \otimes \id)\act^{*}_{w,1 } \ = \  \act^{*}_{z+w,1} \act^{*}_{w,2} (\oplus^{*}_{d_{1},d_{2}} \otimes \id).$$
But this is an easy consequence of the commutativity of
\begin{center}
\begin{tikzcd}[row sep = {40pt,between origins}, column sep = {30pt}]
\BGm \times \Ml \times \Ml \ar[r,"\Delta\times \id"] \ar[d,"\id\times \oplus"]  &[-10pt] \BGm\times \Ml \times \BGm \times \Ml \ar[r,"\act\times \act"]  &[10pt] \Ml \times \Ml \ar[d,"\oplus"] \\ 
\BGm \times \Ml \ar[rr,"\act"]  && \Ml  
\end{tikzcd}
\end{center}
which implies the first equality in
$$ \act^{*}_{z,1}(\oplus^{*}_{d_{1},d_{2}} \otimes \id)\act^{*}_{w,1 }  \ = \ \act^{*}_{z,1}\act^{*}_{w,1 } \act^{*}_{w,2 } (\oplus^{*}_{d_{1},d_{2}} \otimes \id) \ = \ \act^{*}_{z+w,1} \act^{*}_{w,2} (\oplus^{*}_{d_{1},d_{2}} \otimes \id),$$
with the second following because $\act_z^*\act^*_w = \act^*_{z+w}$. This finishes the proof that $\Delta(z)$ is a coassociative vertex coalgebra.
\end{proof}

\newpage

\section{Localised coproducts} \label{sec:localised_coprod}
\noindent In section \ref{sec:JoyceCoproduct} we defined a graded vertex coproduct on $\Al_{Q,W}^T$, whose main piece of data is a map of vector spaces 
$$\Delta(z) \ : \ \Al_{Q,W}^T \ \to \ \Al_{Q,W}^T \otimes_{T}\Al_{Q,W}^T((z^{-1})).$$ 
In this section we show that this data may be repacked in a more or less equivalent form, as a \textit{localised} coproduct 
$$\Delta_{\loc} \ : \ \Al_{Q,W}^T \ \to \ \Al_{Q,W}^T \otimes_{T}\Al_{Q,W}^T \left[S^{-1}\right]$$
first defined by Davison \cite{Da2} but frequently appearing in the CoHA literature \cite{YZ}. To begin with give the definition of localised coproducts and relate them to vertex coproducts. In Theorem \ref{thm:DavisonIsJoyce} we show that applying this construction to Davisons' $\Delta_{\loc}$ recovers the Joyce--Liu vertex coproduct $\Delta(z)$.

We now give a definition of localised coproducts adapted from \cite{Da2}. If the reader does not want to understand the abstract definition, they can jump to section \ref{ssec:loc_coprod_ben}, where we recall the construction in \cite{Da2}, which is the only localised coproduct we will use.

\subsection{Localised coproducts to vertex coproducts} \label{ssec:LocalisedCoalgebra}

\subsubsection{}We work over a base $\mathbf{Z}$-graded commutative ring $R$. All tensor products are implicitly over $R$. Let $\Ol = (\Ol_\alpha)_{\alpha \in \Lambda}$ be a collection of graded commutative rings with a cocommutative graded coproduct 
$$\delta_{\alpha,\beta} \ : \  \Ol_{\alpha+\beta} \ \to \ \Ol_\alpha \otimes \Ol_\beta$$
and $S_{\alpha,\beta} \subseteq \Ol_\alpha \otimes \Ol_\beta$ a collection of multiplicative subsets satisfying 
\begin{equation}
  \label{eqn:MultiplicativeSubsetHexagon} 
  (\delta_{\alpha,\beta} \otimes \id)S_{\alpha+\beta,\gamma} \ \sim \ S_{\alpha,\gamma,13}\cdot S_{\beta,\gamma,23} , \hspace{15mm} (\id \otimes \delta_{\beta,\gamma})S_{\alpha,\beta+\gamma} \ \sim \ S_{\alpha,\beta,12}\cdot S_{\alpha,\gamma,13}
\end{equation}
where equivalence means that the associated localisations of $\Ol_\alpha \otimes \Ol_\beta \otimes \Ol_\gamma$ are equal.

A \textit{localised coalgebra} is a direct sum $A = \bigoplus_{\alpha\in \Lambda} A_\alpha$ of $\Ol_\alpha$-modules with 
$$\Delta_{\alpha,\beta} \ : \ A_{\alpha+\beta} \ \to \ A_\alpha \otimes A_\beta [S_{\alpha,\beta}^{-1}]$$
maps of $\Ol_\alpha$-modules, i.e. $\Delta_{\alpha,\beta}(ba)=\delta_{\alpha,\beta}(b)\Delta_{\alpha,\beta}(a)$, which is coassociative, 
\begin{equation}
    \label{fig:LocalisedCoassociativity}
\begin{tikzcd}[row sep = {15pt,between origins}, column sep = {20pt}]
 &A_\alpha \otimes A_{\beta+\gamma} [S_{\alpha,\beta+\gamma}^{-1}] \ar[rd]  & \\ 
 & &A_\alpha \otimes (A_\beta \otimes A_\gamma) [S_{\beta,\gamma,23}^{-1}, (\id\otimes \delta_{\beta,\gamma})S_{\alpha,\beta+\gamma}^{-1}] \ar[dd,equals,"\wr"] \\
A_{\alpha+\beta+\gamma} \ar[ruu] \ar[rdd]  & & \\
& & (A_{\alpha}\otimes A_\beta) \otimes A_\gamma [S_{\alpha,\beta,12}^{-1}, (\delta_{\alpha,\beta} \otimes \id)S_{\alpha+\beta,\gamma}^{-1}]\\
 &A_{\alpha+\beta} \otimes A_\gamma [S_{\alpha+\beta,\gamma}^{-1}] \ar[ru] & 
\end{tikzcd}
\end{equation}
and which satisfies the counit axioms with respect to a map $\epsilon:A_0 \to k$. The vertical identification in \eqref{fig:LocalisedCoassociativity} follows from \eqref{eqn:MultiplicativeSubsetHexagon}. 

\subsubsection{Remark} We will abuse notation and refer to  
\begin{equation}
  \label{eqn:LocalisedCoproduct}
  \Delta \ : \ A \ \to \ A \otimes A [S^{-1}] \ \defeq \ \bigoplus_{\alpha,\beta} A_\alpha \otimes A_\beta [S_{\alpha,\beta}^{-1}]
\end{equation}
as the localised coproduct.

\subsubsection{} Now assume that all the above structures are equivariant for an action of $\Ga$. That is, we have coactions of $\Cb[z]$ 
$$\act_\Ol \ = \ \exp(zT_\Ol)  \ : \ \Ol_\alpha \ \to \ \Ol_\alpha[z], \hspace{15mm} \act_A \ = \ \exp(zT_A) \ : \ A_\alpha \ \to \ A_\alpha[z]$$
given by algebra and module maps:
\begin{equation}\label{eqn:ActAxiom}
  \act_\Ol(bb') \ = \ \act_\Ol(b)\act_\Ol(b'), \hspace{15mm} \act_A(ba) \ = \ \act_\Ol(b)\act_A(a), 
\end{equation}
such that $S_{\alpha,\beta}$ is translation invariant:\footnote{To be precise, this means that the localisation of $\Ol_\alpha \otimes \Ol_\beta[z]$ by the multiplicative subsets $(\act_\Ol \otimes \act_\Ol)S_{\alpha, \beta}$ and $S_{\alpha, \beta} \subseteq \Ol_\alpha \otimes \Ol_\beta\subseteq \Ol_\alpha \otimes \Ol_\beta[z]$ are isomorphic.}
\begin{equation}\label{eqn:STranslationInvariant}
  (\act_\Ol \otimes \act_\Ol)S_{\alpha, \beta} \ \sim\ S_{\alpha,\beta}, 
\end{equation}
and satisfying  
\begin{equation}\label{eqn:CoprodAxiom}
   \delta_{\alpha,\beta}\cdot \act_\Ol \ = \ (\act_\Ol \otimes \act_\Ol)\cdot \delta_{\alpha,\beta}, \hspace{15mm} \Delta_{\alpha,\beta}\cdot \act_A \ = \ (\act_A\otimes \act_A)\cdot \Delta_{\alpha,\beta}
\end{equation}
where the last equation is valued in the same localisation by translation invariance of $S$.

We call the above structure a \textbf{translation-equivariant localised coalgebra}.

\subsubsection{Producing vertex coproducts} \label{ssec:producing_vertex} Consider a translation equivariant localised coalgebra such that each 
$$(\act_\Ol \otimes \id)S_{\alpha,\beta} \ \in \ \Ol_\alpha \otimes \Ol_\beta ((z^{-1}))$$
is an \textit{invertible} element.

\begin{defn} \label{defn:VertexFromLocalised}
   The maps
$$\delta_\Ol(z) \ : \ \Ol_{\alpha+\beta} \ \to \ \Ol_{\alpha}\otimes \Ol_\beta((z^{-1})), \hspace{15mm} \Delta_A(z) \ : \ A_{\alpha+\beta} \ \to \ A_\alpha \otimes A_\beta((z^{-1}))$$
are defined by $\delta_\Ol(z)=\delta_\Ol\cdot \act$ and $\Delta_A(z) = \Delta_A \cdot \act$ using the following map:
\begin{align*}
  \act \ : \  \Ol_{\alpha}\otimes \Ol_{\beta} [S_{\alpha,\beta}^{-1}] &\ \stackrel{\act_\Ol \otimes \id}{\to} \ \Ol_\alpha \otimes \Ol_\beta[z] \left[ (\act_\Ol \otimes \id)S_{\alpha,\beta}^{-1} \right]\\
  & \hspace{5mm} \ \hookrightarrow \ \Ol_\alpha \otimes \Ol_\beta((z^{-1})) \left[ (\act_\Ol \otimes \id)S_{\alpha,\beta}^{-1} \right]\\
  & \hspace{5mm}  \ \simeq \ \Ol_\alpha \otimes \Ol_\beta((z^{-1}))
\end{align*}
and likewise 
$$\act \ : \ A_\alpha \otimes A_\beta [S_{\alpha,\beta}^{-1}] \ \to \ A_\alpha \otimes A_\beta ((z^{-1}))$$
sending $\frac{a\otimes a'}{s}\ \mapsto \ \frac{\act_A(a) \otimes a'}{\act(s)}$.
\end{defn}

\subsection{Davisons' localised coproduct on the CoHA}
\label{ssec:loc_coprod_ben}

\subsubsection{} We now consider $R = \Ht^{\sbt}(\BT)$ and $\Ol_d= \Ht^{\sbt}(\Ml^T_{Q,d})$ with its cup product, coproduct $\delta_{d,e}=\oplus^*_{d,e}$ and multiplicative subset 
$$S_{d_1,d_2} \ = \ \langle e(\Ext_0), e(\Ext_1)\rangle \ = \ \left \langle e(Q_0), e(Q_1) \right \rangle$$
generated by the Euler class of the components of the $\Ext$ complex \eqref{eqn:ExtComplex}. We define
$$\act_{\Ol} \ : \ \Ht^{\sbt}(\Ml^T_{Q}) \ \stackrel{\act^*}{\to} \ \Ht^{\sbt}(\BGm \times\Ml^T_{Q})\ \simeq \ \Ht^{\sbt}(\Ml^T_{Q})[z]$$
as in Lemma \ref{lem:TautBialgebra}, acting on chern roots as $x_{i,n}\mapsto x_{i,n}+z$. This satisfies the above axioms for $\Ol$, working inside the linear monoidal category $\Vect^T_\Lambda$ with its product $\otimes_T$ defined in \eqref{eqn:OtimesTDefinition}.
\subsubsection{Remark} We could have equivalently used 
$$S \ \sim \ \langle e(\El_i^\vee\boxtimes \El_i), \, e(\El_i^\vee\boxtimes \El_j\, \otimes \, \Ll_e) \ : \ e : i \to j \rangle.$$
\subsubsection{} We now define a localised coproduct on $A=\Al^T_{Q,W}$, assuming the K\"unneth assumption \eqref{eqn:KunnethAssumption}. The \textbf{Davison localised coproduct} from \cite{Da2} and \cite[Section 5.2]{BD} is  
\begin{align*}
   \Delta_{\loc} \ : \ \Al^T_{Q,W} &\ \to \ \Al^T_{Q,W} \otimes_{T}\Al^T_{Q,W} \left[ S^{-1}\right] \\
   \alpha & \ \mapsto \ \frac{e(Q_0)}{e(Q_1)} \cdot (q^*)^{-1}p^*\alpha 
\end{align*}
defined in terms of the diagram 
\begin{center}
\begin{tikzcd}[row sep = {30pt,between origins}, column sep = {45pt, between origins}]
 & \SES^{T}_{Q,d_{1},d_{2}} \ar[rd,"p"] \ar[ld,"q"']   & \\ 
\Ml^T_{Q,d_{1}} \times_{\BT} \Ml^{T}_{Q,d_{2}} \ar[ru,"s"',bend right = 10]  & &\Ml^{T}_{Q,d_{1}+d_{2}} 
\end{tikzcd}
\end{center}
and the equivariant Euler classes $e(Q_0), e(Q_1)$ as defined in Proposition \ref{prop:euler_ext_to_euler_ben}.

That is, we take use the pullbacks $p^*,q^*$ on critical cohomology as in  \eqref{eqn:CritPullback}, the latter being invertible because $q$ is an affine fibration\footnote{Note it is important to consider the map $q \colon \SES^{T}_{Q,d_{1},d_{2}} \to \Ml^{T}_{Q,d_{1}} \times_{\BT} \Ml^{T}_{Q,d_{2}}$ and not $\SES^{T}_{Q} \to \Ml^{T}_{Q,d_{1}} \times \Ml^{T}_{Q,d_{2}}$, since the latter is \textit{not} an affine fibration.} , and the Thom--Sebastiani and K{\"u}nneth isomorphisms:
\begin{align*}
    \Ht^{\sbt}(\Ml^T_{Q,d_{1}},\varphi_W) & \ \stackrel{p^*}{\to}\ \Ht^{\sbt}(\SES^T_{Q,d_1,d_2},\varphi_{W_{\SES}}) \\
 &\ \stackrel[\raisebox{2pt}{$\scriptstyle \sim$}]{q^*}{\leftarrow} \  \Ht^{\sbt}(\Ml^T_{Q,d_1}\times_{BT} \Ml^T_{Q,d_2},\varphi_{W\boxplus W})\\
 & \ \simeq \  \Ht^{\sbt}(\Ml^T_{Q,d_1},\varphi_{W})\otimes_{\Ht^{\sbt}(\BT)}\Ht^{\sbt}(\Ml^T_{Q,d_2},\varphi_{W}).
\end{align*}
In the middle we denoted the common function $p^*W =q^*(W \boxplus W)$ by $W_{\SES}$.

\begin{lem}\label{lem:DavisonDirectSum}
  We have that 
  $$\Delta_{\loc}(\alpha) \ = \ e(\Ext)\cdot \oplus^*(\alpha)$$ 
  where $e(\Ext)$ is the localised Euler class (as defined in \ref{ssec:two_term_Euler}) of the two-term complex $\Ext$. 
\end{lem}
\begin{proof} We first note that $(q^*)^{-1}=s^*$ because $s$ is a right inverse of $q$, so by functoriality of the pullback map we have 
    \begin{equation}
      \label{eqn:DavisonDirectSumFirstPart} 
      \oplus^* \ = \ (ps)^* \ = \ s^*p^* \ = \ (q^*)^{-1}p^*.
    \end{equation}
  We are then finished if we have 
    \begin{equation}
      \label{eqn:DavisonDirectSumSecondPart} 
      e(\Ext) \ = \ \frac{e(Q_0)}{e(Q_1)},
    \end{equation}
  which follows from Proposition \ref{prop:euler_ext_to_euler_ben}.
\end{proof}

Finally we let $\act_A$ be the composition 
$$\mathcal{A}^{T}_{Q,W} \ \stackrel{\act^*}{\to} \ \Ht^{\sbt}(\BGm) \otimes \mathcal{A}^{T}_{Q,W}\ \simeq \ \mathcal{A}^{T}_{Q,W}[z], $$
using the construction in Lemma \ref{lem:HolomorphicJoyceCoproduct}, then

\begin{prop}
  $\Al^T_{Q,W}$ is a translation equivariant localised coproduct satisfying the properties in subsection \ref{ssec:producing_vertex}.
\end{prop}
\begin{proof}
The localised coproduct and $\Cb[z]$-coaction are linear over $\Ol_d=\Ht^{\sbt}(\Ml^T_{Q,d})$ by Lemma \ref{lem:linearity_two_acts}, proving \eqref{eqn:ActAxiom} and \eqref{eqn:CoprodAxiom}. It is coassociative by \cite{Da2}. 

It remains to check that $S$ satisfies the relevant axioms. It satisfies the first hexagon relation \eqref{eqn:MultiplicativeSubsetHexagon} for multiplicative subsets because 
\begin{align*}
   &\Ht^{\sbt}(\Ml_{Q,d_1}^T)\otimes_T\Ht^{\sbt}(\Ml_{Q,d_1}^T)\otimes_T\Ht^{\sbt}(\Ml_{Q,d_1}^T)\left[ ( \oplus\times \id)^* e(\Ext_i)^{-1} \ : \ i=0,1\right]\\[5pt]
   & \hspace{15mm} \ \simeq \ \Ht^{\sbt}(\Ml_{Q,d_1}^T)\otimes_T\Ht^{\sbt}(\Ml_{Q,d_1}^T)\otimes_T\Ht^{\sbt}(\Ml_{Q,d_1}^T)\left[ e(\Ext_i)_{13}^{-1}, e(\Ext_i)_{23}^{-1} \ : \ i=0,1\right]
\end{align*}
since $(\oplus\times \id)^*e(\Ext_i) = e(\Ext_i)_{13}e(\Ext_i)_{23}$, and likewise we have the second hexagon equation \eqref{eqn:MultiplicativeSubsetHexagon}. Every element of $S$ satisfies translation invariance \eqref{eqn:STranslationInvariant} because 
$$(\act_{z}^*\otimes \act_z^*)e(\Ext_i) \ = \ e \left( (\act\times \act)^*\Ext_i\right) \ = \ e(\gamma\otimes \gamma^{-1}\boxtimes \Ext_i) \ = \ e(\Ext_i)$$
as the $\BGm$ weights $(-1,1)$ of the Ext complex $\Ext=(\Ext_0\to \Ext_1)$ sum to zero. Likewise, since the leading term of
$$(\act_{z}^*\otimes \id)e(\Ext_i) \ = \ \Psi(\Ext_i,-z)\ =\ \pm z^{\rank \Ext_i} \mp c_1(\Ext_i)z^{\rank \Ext_i-1} \pm \cdots$$
has nonzero leading term, it follows that it is invertible as a Laurent series in $z^{-1}$. 
\end{proof}

Note the relation between the invertibility of elements in $(\act_z^*\otimes \id)S$ and \cite[Prop. 4.1]{Da2}.

\subsection{Recovering the Joyce--Liu vertex coproduct}

\subsubsection{} The main result in this section is that Davisons' localised coproduct agrees with the Joyce--Liu vertex coproduct, up to the map $\act$ in Definition \ref{defn:VertexFromLocalised} interpolating between localised and vertex coproducts. 

For the sake of explicitness, we note that in our case $\act$ takes a localised class, whose denominator is a product of  $x_{i,n}\otimes 1 - 1\otimes x_{j,m} +\wtt(e)$ over edges $e : i \to j$, replaces any chern root in the first tensor factor of denominator and numerator with itself plus $z$, then Laurent expands in $z^{-1}$.

\begin{theorem}\label{thm:DavisonIsJoyce}
  The diagram 
  \begin{center}
  \begin{tikzcd}[row sep = {30pt,between origins}, column sep = {20pt}]
   & \Al^T_{Q,W} \otimes_{T} \Al^T_{Q,W}\left[ S^{-1}\right]  \ar[dd,"\act"] \\
  \Al^T_{Q,W} \ar[rd,"\Delta(z)"'] \ar[ru,"\Delta_{\loc}"]  &   \\ 
   & \Al^T_{Q,W} \otimes_{T} \Al^T_{Q,W}((z^{-1})) 
  \end{tikzcd}
  \end{center}commutes.
\end{theorem}

\begin{proof}
The above diagrams factors as 
\begin{center}
\begin{tikzcd}[row sep = {15pt,between origins}, column sep = {145pt, between origins}]
    &[-20pt] &[20pt] \Al^T_{Q,W}\otimes_{T} \Al^T_{Q,W}[S^{-1}]  \ar[dddd,"\act"] \\
 &\Al^T_{Q,W}\otimes_{T} \Al^T_{Q,W} \ar[dd,equals] \ar[ru,"e(Q_0) \textup{/} e(Q_1)"]  & \\ 
\Al^T_{Q,W} \ar[rd," \oplus^*"'] \ar[ru," (q^{*})^{-1}p^*"]  & & \\
&\Al^T_{Q,W}\otimes_{T} \Al^T_{Q,W} \ar[rd,"\Psi(\Ext \textup{,} -z)\cdot \act_{z \textup{,} 1}"']   & \\
    & & \Al^T_{Q,W}\otimes_{T} \Al^T_{Q,W} ((z^{-1}))
\end{tikzcd}
\end{center}
whose left cell commutes by \eqref{eqn:DavisonDirectSumFirstPart}. It thus suffices to show that
$$\act \left( \frac{e(Q_1)}{e(Q_0)}\cdot \alpha_1\otimes \alpha_2\right) \ = \ \Psi(\Ext,-z)\cdot \act^{*}_{z,1}(\alpha_1 \otimes \alpha_2).$$
Using equation \eqref{eqn:DavisonDirectSumSecondPart} we have that 
$$\act \left( \frac{e(Q_1)}{e(Q_0)}\cdot \alpha_1\otimes \alpha_2\right)  \ = \ \act_{z,1}^*\left(\frac{e(Q_1)}{e(Q_0)}\right)\cdot \alpha_1\otimes \alpha_2 \ = \ \act_{z,1}^*(e(\Ext))\cdot \alpha_1 \otimes\alpha_2$$
so we are finished as soon as we know that
\begin{equation}
  \label{eqn:PsiAct} 
  \act_{z,1}^*(e(\Ext))\ = e(\act_{1}^*\Ext)\ =\ e(\gamma^{-1}\boxtimes \Ext)  \ = \ \Psi(\Ext,-z),
\end{equation}
but the middle equality follows as $\Ext$ has $\BGm$ weight $-1$ in the first factor, and the final equality follows from Proposition \ref{prop:PsiEuler}. This proves the Theorem.
\end{proof}

\begin{cor}
Consider $\alpha \in \mathcal{A}^{T}_{Q,W,d}$
  We have the following explicit formula for the Joyce--Liu vertex coproduct:
  $$\Delta(\alpha,z) \ = \ \sum_{d^\prime,d^{\prime \prime} } \frac{\bigsqcap_i \bigsqcap_{(n,m)=(1,1)}^{(d^{\prime}_{i},d^{\prime \prime }_{i})} (-z-x_{i,n}\otimes 1 + x_{j,m}\otimes 1)}{\bigsqcap_{e \, : \, i \to j} \bigsqcap_{(n,m)=(1,1)}^{(d^{\prime}_{i},d^{\prime \prime}_{j})} (-z-x_{i,n}\otimes 1 + x_{j,m}\otimes 1 + \wtt(e))} \ \cdot \ \act_{z,1}^*\oplus_{d^{\prime},d^{\prime \prime}}^*(\alpha)$$
  where we sum over all dimension vectors $d,d^{'}$ with $d^{\prime} +d^{\prime \prime} = d$.
\end{cor}

\begin{proof}
   This follows from the fact that 
   $$\act \ : \ \Ht^{\sbt}(\Ml^T_{Q,d^{\prime}})\otimes_T \Ht^{\sbt}(\Ml^T_{Q,d^{\prime \prime}})[S^{-1}] \ \to \ \Ht^{\sbt}(\Ml^T_{Q,d^{\prime}}) \otimes_T \Ht^{\sbt}(\Ml^T_{Q,d^{\prime \prime}})((z^{-1}))$$
   acts on chern roots as $x_{i,n}\mapsto x_{i,n}+z$.
\end{proof}

\newpage 
 \section{The critical CoHA is a bialgebra}
 
\noindent In this section we will prove that the CoHA algebra structure is compatible with the Joyce-Liu vertex coalgebra. This culminates in Theorem \ref{thm:CoHABialgebra}, where we also give a comprehensive list of algebraic structures and compatibilities on the CoHA. To formulate these compatibilities we now start by defining a particular subalgebra of the singular cohomology of $\Ml^{T}_{Q}$.

\subsection{The tautological ring} \label{ssec:tautring}

\subsubsection{}In the following two sections we define a \textit{tautological} subring of the cohomology 
$$\Ht^{\sbt}(\Ml^{T}_{Q})_{\taut} \ \subseteq \ \Ht^{\sbt}(\Ml^{T}_{Q}) \ = \ \bigsqcap_{d\in \mathbf{N}^{Q_{0}}} \Ht^{\sbt}(\Ml^{T}_{Q,d})$$
as well as a \textit{tautological part of the $R$-matrix}
$$R_{\taut}(z) \ \in \ \Ht^{\sbt}(\Ml^{T}_{Q})_{\taut} \otimes_{T} \Ht^{\sbt}(\Ml^{T}_{Q})_{\taut}[[z^{-1}]],$$
and show that it is a commutative, cocommutative holomorphic vertex bialgebra.

\subsubsection{} The \textbf{tautological ring} is the $\Ht^{\sbt}(\BT)$-subalgebra of the cohomology of $\Ml^{T}_{Q}$ generated by the chern characters of the tautological bundles: 
$$\Ht^{\sbt}(\Ml_Q^T)_{\taut} \ = \ \Ht^{\sbt}(\BT)\langle \ch_r(\El_i) \ : \ r \ge 0, i\in Q_0 \rangle \ \subseteq \ \bigsqcap_{d \in \Nb^{Q_{0}}} \Ht^{\sbt}(\Ml^{T}_{Q,d}) \ = \ \Ht^{\sbt}(\Ml_Q^T).$$
Equivalently, we can take the ring generated by the chern classes $c_k(\El_i)$.

\begin{prop}\label{prop:TautFree}
  The tautological ring is a free polynomial $\Ht^{\sbt}(\BT)$-algebra over the chern generators $\ch_r(\El_i)$:
  $$\Ht^{\sbt}(\Ml_Q^T)_{\taut} \ \simeq \ \Ht^{\sbt}(\BT)[\gamma_{i,r} \ : \ r \ge 0, i \in Q_0] \hspace{15mm} \ch_r(\El_i) \ \mapsto \ \gamma_{i,r}.$$
\end{prop}
\begin{proof}
   It is enough to show that the classes $\ch_r(\El_i)$ are algebraically independent over $\Ht^{\sbt}(\BT)$. Since we have $\Ht^{\sbt}(\Ml_Q^T)\simeq \Ht^{\sbt}(\BT)\otimes \Ht^{\sbt}(\Ml_Q)$, it is enough to show that the classes $\ch_r(\El_i)$ are algebraically independent in $\Ht^{\sbt}(\Ml_Q)$.   Note that for any $d\in \Zb^r$, one can show that
   $$\Cb[x_k \ : \ 0 < k \le d]^{\Sk_d} \ \simeq \ \Cb[p_k \ : \ 0 < k \le d]$$
   is a polynomial algebra in the symmetric power sum $p_k = \sum_{i} x_i^k = \sum x_{i_1}^{k_1}\cdots x_{i_r}^{k_r}$. Thus any nontrivial polynomial in $\ch_r(\El_i)$ which vanishes in the tautological ring would, by taking its image in $\Ht^{\sbt}_T(\Ml_{Q,d})\simeq \Cb[x_k \ : \ 0 < k \le d]^{\Sk_d}$ for any $d$ greater than the $r$ appearing in this polynomial, give a nontrivial relation betwen the $p_k$, which is impossible since symmetric power sums are algebraically independent.
\end{proof}

\subsubsection{Warning} \label{war:Euler} Euler classes of the tautological bundles $\El_i$ however usually do not lie in the tautological ring. Indeed, we have that
$$e(\El_i) \ = \ \left( c_{\rank \El_{i,d}}(\El_{i,d}) \right) \ \in \ \bigsqcap_{d \in \Lambda} \Ht^{\sbt}(\Ml_{Q,d}^T),$$
whose componentwise cohomological degree $\chi(d,\delta_i)=2\rank \El_{i,d}$ is generally not bounded above. This is in contrast to every element of the tautological ring, which is a finite sum of elements with finite cohomological degree.

For instance, if $Q=\bullet$ and $\El$ is the tautological vector bundle over $\Ml_Q=\BGL$ then 
$$e(\El) \ = \ \left( 1,\, x_1,\, x_1x_2,\, x_1x_2x_3,\, \ldots \ \right)$$
is certainly not a polynomial in the symmetric power sums $p_k$.

\subsection{Tautological part of the $R$-matrix} 

\subsubsection{} The \textbf{tautological part of the $R$-matrix} is 
$$R_{\taut}(z) \ = \ \frac{c(\sigma^*\Ext^\vee,z^{-1})}{c(\Ext,z^{-1})},$$
i.e. $R_{\taut}(z) = \left( R_{\taut}(z)_{d_1,d_2}\right)_{d_1,d_2\in \Lambda}$ where 
$$R_{\taut}(z)_{d_1,d_2} \ = \ \frac{c(\sigma^*\Ext_{d_2,d_1}^\vee,z^{-1})}{c(\Ext_{d_1,d_2},z^{-1})} \ \in \ \Ht^{\sbt}(\Ml^{T}_{Q,d_{1}})\otimes_{T} \Ht^{\sbt}(\Ml^T_{Q,d_2})[[z^{-1}]],$$
expanded as a geometric series in $z^{-1}$.
\begin{prop}
  This defines an element of 
  $$R_{\taut}(z) \ \in \ \Ht^{\sbt}(\Ml_Q^T)_{\taut} \otimes_{T} \Ht^{\sbt}(\Ml_Q^T)_{\taut}[[z^{-1}]].$$ 
\end{prop}
\begin{proof}
   We expand the denominator as a geometric series, which gives $R_{\taut}(z)$ as a cohomology-valued power series in $z^{-1}$ with coefficients polynomials in the nonzero chern classes of $\Ext$ and $\sigma^*\Ext^\vee$, hence the nonzero chern characters. By the explicit form \eqref{eqn:ExtComplex} for the Ext complex and additivity of chern characters under direct sum, these are polynomials in the $z^{>0}$ coefficients of
   $$\ch(\El_i^\vee \boxtimes \El_j \, \otimes\, \Ll,z) \ = \ \ch(\El_i^\vee,z) \otimes \ch(\El_j,z)\, \cdot\, \ch(\Ll,z)$$
   which are themselves $\Ht^{\sbt}(\BT)$-polynomials in $\ch_r(\El_k)$ for $r \ge 0$, hence lie in the tautological ring. 
\end{proof}

In the ADE case this agrees with the ordinary tautological ring by Proposition \ref{prop:phiareincartan}, and for general quivers with potential is probably the better definition to give.

\subsubsection{} \label{sssec:full_joyce} We will also consider the  \textbf{full Joyce $R$-matrix} 
\begin{equation}
  \label{eqn:FullRMatrix} 
  R(z)  \ = \ \frac{\Psi(\sigma^*\Ext^\vee,z)}{\Psi(\Ext,z)} \ \in \ \Ht^{\sbt}(\Ml^{T}_Q \times_{\BT} \Ml^{T}_Q) ((z^{-1}))
\end{equation}
which is non-tautological; see Proposition \ref{prop:explicit_fullrmatrix} for an explicit formula in terms of chern roots. We can see that
\begin{equation*}
    R(z)_{d_1,d_2} \  = \  \frac{z^{\chi(d_2,d_1)}}{z^{\chi(d_1,d_2)}}\frac{c(\sigma^*\Ext^{\vee}_{d_2,d_1},z^{-1})}{c(\Ext_{d_1,d_2},z^{-1})} \  = \  \frac{z^{\chi(d_2,d_1)}}{z^{\chi(d_1,d_2)}}R_{\taut}(z)
\end{equation*}
and so if the quiver is symmetric then $R(z)$ equals $R_{\taut}(z)$.

\subsubsection{Classical limits} 
If the Ext complex is symmetric $[\sigma^*\Ext]\simeq [\Ext]$ as an elements in the Grothendieck group of $\Ml_{Q}\times \Ml_{Q}$, not necessarily $T$-equivariant, then the leading term of the $R$-matrix is the identity. This is the case where the quiver is symmetric, for instance in the tripled quiver case:

\begin{prop}\label{prop:RMatrixLeadingTerm}
If the quiver $Q$ is symmetric (e.g. a tripled quiver) then we have that 
$$R_{\taut}(z) \ = \ 1 \otimes 1 + \Ol(\hbar_1, \ldots , \hbar_r)$$
where $\Ol(\hbar_1, \ldots , \hbar_r)$ is a power series whose coefficients as $\Ht^{\sbt}(\BT)= \Cb[\hbar_1, \ldots , \hbar_r]$-polynomials have zero constant term.
\end{prop}

\subsection{Structures on the tautological ring} \label{ssec:StructuresTautologicalRing}

\subsubsection{} We show that the geometric structures on the moduli stacks (section \ref{ssec:GeometricStructures}) induce algebraic structures on the tautological ring. This more or less follows from section \ref{ssec:StructuresCriticalCohomology} applied to the case of zero potential function, the only remaining thing to check being that these operations preserve the tautological ring.

\begin{lem}\label{lem:TautBialgebra}
    In addition to its cup product, there is a cocommutative coproduct and an endomorphism 
  $$\oplus^* \ : \ \Ht^{\sbt}(\Ml_Q^T)_{\taut} \ \to \ \Ht^{\sbt}(\Ml_Q^T)_{\taut} \otimes_{T} \Ht^{\sbt}(\Ml_Q^T)_{\taut}, \hspace{15mm} T \ : \ \Ht^{\sbt}(\Ml_Q^T)_{\taut} \ \to \ \Ht^{\sbt}(\Ml_Q^T)_{\taut}$$
  making $\Ht^{\sbt}(\Ml_Q^T)_{\taut}$ into a $\Ht^{\sbt}(\BT)$-linear commutative, cocommutative bialgebra with biderivation.
\end{lem}
\begin{proof}
  Since $\oplus^*\El_i = \El_i \boxtimes 1  + 1 \boxtimes\El_i$ and $\act^*\El_i = \gamma \boxtimes \El_i$, the maps on cohomology defined in Lemma \ref{lem:HolomorphicJoyceCoproduct} send
  \begin{align*}
     \ch_r(\El_i) &\ \stackrel{\oplus^*}{\mapsto} \ \ch_{r}(\El)\otimes 1 + 1 \otimes \ch_r(\El_i),\\
\ch_r(\El_i) &\ \stackrel{\act_z^*}{\to} \  \ch_r(\gamma \boxtimes \El_i) \ = \ \sum  \frac{1}{k!}z^k \ch_{r-k}(\El_i)
  \end{align*}
by section \ref{sec:ChernClasses}, and so $\act_z^*=\exp(Tz)$ where $T$ sends $\ch_r(\El_i)\mapsto \ch_{r-1}(\El_i)$ for positive $r$. Both maps therefore preserve the subspaces of tautological classes.
\end{proof}
\begin{cor}
  $\Ht^{\sbt}(\Ml_Q^T)_{\taut}$ is a holomorphic $\Ht^{\sbt}(\BT)$-linear vertex bialgebra with vertex coproduct $\act^{*}_{z,1} \circ \oplus^{*}$
\end{cor}
\begin{proof}
  This follows from the standard equivalence between holomorphic vertex bialgebras and bialgebras with biderivation extending Lemma \ref{lem:holomorphic_covertex}. This equivalence sends the coproduct $\Delta$ and coderivation $T$ to $\Delta(z)  =  (e^{zT}\otimes \id)\Delta$, which in this case is $\oplus^*(z)=\act_{z,1}^*\oplus^*$.
\end{proof}

\subsubsection{Description as a colimit}

 \label{eqn:MapsDeltai} Let $\Cb^e$ denote the \textit{trivial} representation of $Q$ with dimension $e\in \Lambda$, all of whose edge maps are zero. Taking the direct sum with this representation induces a map of stacks
 \begin{equation}
   \label{eqn:Iota} 
   \iota_{d,e} \ : \ \Ml_{Q,d}^T \ \to \ \Ml_{Q,d+e}^T,
 \end{equation}
induced by the associated map on the representation vector spaces
$$\Rep(Q,d) \ \hookrightarrow \ \Rep(Q,d+e), \hspace{15mm} V \ \mapsto \ V \oplus \Cb^e$$
which intertwines the actions of $\GL_d\hookrightarrow \GL_{d+e}$. We thus have a diagram of stacks indexed by a poset
$$(\Lambda, \le) \ \to \ \Stk, \hspace{15mm} d \le d' \ \ \mapsto \ \ \Ml^T_{Q,d} \ \stackrel{\iota_{d,d'-d}}{\to} \ \Ml^T_{Q,d'},$$
and can identify the tautological ring as the associated limit on cohomology. Writing $\Ht^{\sbt}(\Ml^{T}_Q)_{\taut}^+$ for the subring generated by positive rank chern classes,  We have 
  \begin{equation}
  \Ht^{\sbt}(\Ml^T_{Q})_{\taut}^+ \ \stackrel{\sim}{\to} \ \lim_{d\in \Lambda} \Ht^{\sbt}(\Ml^T_{Q,d})
  \end{equation}
  is the limit of the above diagram as a graded algebra.  In particular, the tautological ring should loosely speaking be viewed as the $T$-equivariant cohomology of the limit of the above functor
$$\Ml_{Q,\infty} \ \defeq \ \colim_{d\in \Lambda} \Ml_{Q,d} \ \simeq \ \Rep(Q,\infty)/\GL_\infty$$
where $\Rep(Q,\infty)=\bigcup_d \Rep(Q,d)$ and  $\GL_\infty=\bigcup_d \GL_d$ are the unions of the representation spaces and groups under the above-defined maps, and the second isomorphism holds since quotient stacks are defined as colimits and any two colimits commute.

\subsection{Braidings and bialgebras}
\label{ssec:BraidingsBialgebras}
We start this section with making precise exactly in what sense the products and coproducts we consider will end up being compatible. The set up here is used in the main Theorem \ref{thm:CoHABialgebra} of this section.
\subsubsection{} 
A bialgebra is a vector space $A$ equipped with an algebra structure and a coalgebra structure such that 
\begin{equation}\label{eqn:Bialgebra}
   \Delta(a\cdot b) \ = \ \Delta(a)\cdot_\beta \Delta(b),
\end{equation}
where the right hand side means we swap the middle two factors using the braiding before multiplying:
\begin{center}
\begin{tikzcd}[row sep = {30pt,between origins}, column sep = {20pt}]
A \otimes A \ar[rr,"m"] \ar[d,"\Delta \otimes \Delta"]  &[-60pt] &A  \ar[dd,"\Delta"] \\ 
 (A \otimes A) \otimes (A\otimes A) \ar[rd,"\beta_{23}","\sim"' sloped] && \\
 &(A \otimes A) \otimes (A\otimes A) \ar[r,"m\otimes m"]  & A\otimes A 
\end{tikzcd}
\end{center}
 To be able to define an algebra or coalgebra, we have to work in a monoidal category $(\Cl,\otimes)$. To be able to define a bialgebra, we have to work in a braided monoidal category $(\Cl,\otimes,\beta)$. Since our objects of interest are vertex and localised bialgebras, we need to work with meromorphic and localised braided monoidal categories.
 \subsubsection{}  \label{sssec:MeromorphicBraided} There exist several notions of meromorphic braided monoidal category in the literature, see for example \cite{La,Soibelman}. However, in this paper we do not try to put our results into this categorical context, instead we define everything in terms of explicit maps
$$\beta_{M,N}(z) \ : \ M\otimes N ((z^{-1})) \ \stackrel{\sim}{\to} \ N\otimes M ((z^{-1})), \hspace{15mm} \beta_{M,N,\loc} \ : \ M\otimes N [S^{-1}] \ \stackrel{\sim}{\to} \ N\otimes M [S^{-1}]$$
for any pair of objects $M,N$, satisfying the hexagon relations with spectral parameter: 
\begin{equation}
  \label{eqn:SpectralHexagon}
\beta_{L\otimes M ((w^{-1})),N}(z) \ = \ \beta_{L,N}(z-w)\beta_{M,N}(z), \hspace{15mm} \beta_{L,M\otimes N((w^{-1}))}(z) \ = \ \beta_{L,M}(z+w)\beta_{L,N}(z),
\end{equation}
and hence the meromorphic braid relation: 
\begin{equation}
  \label{eqn:MeromorphicBraid} 
  \beta_{L,M}(z)\beta_{L,N}(z+w)\beta_{M,N}(w) \ = \ \beta_{M,N}(w)\beta_{L,N}(z+w)\beta_{L,M}(z).
\end{equation}
and likewise for $\beta_{\loc}$, defined in subsection \ref{ssec:LocalisedProduct}.

\subsubsection{Tautological ring modules} 
We will consider meromorphic braidings on the category of tautological modules.
We first consider the properties of the tautological part of the $R$-matrix for a general quiver.

Firstly a warning: the tautological part of the $R$-matrix does \textit{not} satisfy the spectral hexagon relations. The failure is measured by the antisymmetrisation  $\widetilde{\chi}(d_{1},d_{2})= \chi(d_{1},d_{2})-\chi(d_{2},d_{1})$ of the Euler form:
\begin{lem}\label{lem:SpectralHexagon}
$R_{\textup{taut}}(z)$ satisfies 
\begin{align} \label{eqn:RTautHex1}
    (\oplus^*(w)\,\otimes\, \id)R_{\taut}(z)_{d_1+d_2,d_3}& \ = \ \frac{(z-w)^{\widetilde{\chi}(d_3,d_1)}}{z^{\widetilde{\chi}(d_3,d_1)}} R_{\taut,13}(z-w)_{d_1,d_3}R_{\taut,23}(z)_{d_2,d_3},\\
    \label{eqn:RTautHex2}
    (\id\, \otimes\, \oplus^*(w))R_{\taut}(z)_{d_1,d_2+d_3}& \ = \ \frac{(z+w)^{\widetilde{\chi}(d_2,d_1)}}{z^{\widetilde{\chi}(d_2,d_1)}} R_{\taut,12}(z+w)_{d_1,d_2}R_{\taut,13}(z)_{d_1,d_3}.
\end{align} 
\end{lem}
\begin{proof}
We prove this using a combination of
\begin{align} \label{eqn:directsumcoprodtaut}
    (\oplus^*\otimes \id)R_{\taut,d_{1}+d_{2},d_{3}}(z) \ & = \ R_{\taut,d_{1},d_{3}}(z)R_{\taut,23,d_{2},d_{3}}(z) \\
    (\id \otimes \oplus^*)R_{\taut,d_{1},d_{2}+d_{3}}(z) \ &= \ R_{\taut,d_{1},d_{2}}(z)R_{\taut,13,d_{1},d_{3}}(z) \nonumber 
\end{align}
\begin{align} \label{actcoprodtaut}
    \act_{w,1}^*R_{\taut,d_{1},d_{2}}(z) \ & = \ \left(\frac{z-w}{z}\right)^{\widetilde{\chi} (d_{2},d_{1})}R_{\taut}(z-w) \nonumber \\
    \act_{w,2}^*R_{\taut,d_{1},d_{2}}(z) \ & = \ \left(\frac{z+w}{z}\right)^{\widetilde{\chi}(d_{2},d_{1})}R_{\taut}(z+w)
\end{align}
by the definition of $\oplus^*(w)$. Note that $R_{\taut}(z) = \bigsqcap c(\El,z)^\pm$ is an alternating product of $c(\El,z)$'s over a set of vector bundles $\El$ on $\Ml^T_Q\times_{\BT}\Ml^T_Q$ satisfying 
  $$(\oplus^*\,\otimes\, \id)\El \ = \ \El_{13}\oplus\El_{23}, \hspace{15mm}  \act_{1}^*\El \ = \ \gamma^{-1}\boxtimes \El$$
   $$(\id\, \otimes\, \oplus^*)\El \ = \ \El_{12}\oplus\El_{13}, \hspace{15mm} \act_{2}^*\El \ = \ \gamma\boxtimes \El.$$
  The result thus follows for $c(\El,z)$ by Proposition \ref{chern_series_prop}, and hence for $R_{\taut}(z)$.
\end{proof}
Using the relation between the full and tautological parts \ref{sssec:full_joyce} we immediately get
\begin{cor}
    The  full Joyce $R$-matrix satisfies the spectral hexagon relations
$$(\oplus^*(w)\otimes \id)R(z) \ = \ R(z-w)_{13}R(z)_{23} , \hspace{10mm}
    (\id \otimes \oplus^*(w))R(z) \ = \ R(z)_{13}R(z+w)_{13}$$
and $\sigma(R(z))$ satisfies the same, with $w$ negated on the right hand side of both equations.
\end{cor}
\begin{cor} \label{cor:TautQuasitriangular}
  Let $Q$ be symmetric, then $\Ht^{\sbt}(\Ml^T_Q)_{\taut}$ is a quasitriangular vertex bialgebra: a vertex bialgebra with an element $R_{\taut}(z)$   satisfying Lemma \ref{lem:SpectralHexagon} and
  $$\sigma(z) \cdot \oplus^*(\act^*_{z} h,-z) \ = \  R_{\taut}(z)^{-1}\oplus^*(h,z) R_{\taut}(z).$$ 
\end{cor}
\begin{proof}
 The $R_{\taut}(z)$-conjugation on the right side is trivial since the cup product on even degree classes is commutative. We finish by noting that 
  $$\sigma \act_{-z,1}^* \oplus^*\act_{z}^* \ = \ \sigma\cdot \act_{z,2}^* \oplus^* \ = \ \act_{z,1}^*\oplus^*$$ 
  which implies that $\sigma \oplus^*(\act^*_{z} h,-z) = \oplus^*(h,z)$ as required.
\end{proof}

Let $Q$ be symmetric, so that $R(z) = R_{\textup{taut}}(z)$ is \textit{both} tautological-valued and satisfies the spectral hexagon relations. The category on which we will consider a meromorphic braiding is
$$\Cl \ =\ H\Md(\Vect^T_\Lambda),$$
 the category of modules over the tautological ring $H=\Ht^{\sbt}(\Ml^{T}_Q)_{\taut}$. Here
 $$\Vect_\Lambda^T \ =\ \Ht^{\sbt}(\BT)\Md_{\Lambda\times \Zb}$$
 is the category of $\Lambda$-graded vector spaces with an additional ``cohomological'' $\Zb$-grading denoted by $|\ \ |$ and action of $\Ht^{\sbt}(\BT)$. Both categories are equipped with monoidal structure $\otimes_T$ as defined in \eqref{eqn:OtimesTDefinition}.
The \textbf{meromorphic braiding} on $(\Cl,\otimes_T)$ is given by
\begin{align}\label{eqn:MeromorphicBraiding} 
    \beta(z) \ : \  V_\lambda\otimes_T W_\mu ((z^{-1})) &\ \stackrel{\sim}{\to} \ W_\mu\otimes_T V_\lambda ((z^{-1}))\\
     v \otimes w  & \ \mapsto \ \sigma\cdot \left(R(z)\cdot v\otimes w\right)\nonumber
\end{align}
where $\sigma$ is the Koszul sign braiding and we have braided vector spaces concentrated in degree $\lambda,\mu\in \Lambda$. Here we view \eqref{eqn:MeromorphicBraiding} as a map of modules over the tautological ring, where a tautological class $h$ acts on both via the holomorphic vertex coproduct $\oplus^*(h,z)=\act_{z,1}^*\oplus^*(h)$. This defines a meromorphic braiding as a consequence of the above corollaries.

\subsubsection{Vertex bialgebras} 
Given such a braiding as above, we define

\begin{defn} \label{defn:VertexBialgebra}
   A \textit{vertex bialgebra} is an object $B$ with an associative product $m$ and coassociative vertex coproduct $\Delta(z)$, such that
\begin{equation}\label{fig:CoHACoProd}
\begin{tikzcd}[row sep = {50pt,between origins}, column sep = {20pt}]
B \otimes_T B \ar[rr,"m"] \ar[d,"\Delta(z) \otimes_T \Delta(z)"]  &[-150pt] &B  \ar[dd,"\Delta(z)"] \\ 
 (B \otimes_T B) \otimes_T (B\otimes_T B)((z^{-1})) \ar[rd,"\beta_{23}(z)","\sim"' sloped] && \\[-10pt]
 &(B \otimes_T B) \otimes_T (B\otimes_T B) ((z^{-1}))\ar[r,"m\otimes_T m"]  & B\otimes_T B((z^{-1}))
\end{tikzcd}
\end{equation}
commutes. 
\end{defn}

Recalling the definition of localised coproduct from section \ref{ssec:LocalisedCoalgebra},

\begin{defn}
   A \textit{localised bialgebra} is a graded object $B=(B_\alpha)$ with translation equivariant localised coproduct and associative graded product 
$$m \ : \ B_{\alpha}\otimes B_{\beta} \ \to \ B_{\alpha+\beta}$$ 
linear over $\Ol$, such that 
\begin{equation}
\begin{tikzcd}[row sep = {50pt,between origins}, column sep = {20pt}]
B \otimes_T B \ar[rr,"m"] \ar[d,"u\cdot \Delta_{\loc} \otimes_T \Delta_{\loc}"]  &[-150pt] &B  \ar[dd,"\Delta_{\loc}"] \\ 
 (B \otimes_T B) \otimes_T (B\otimes_T B)_{(12,14,32,34)} \ar[rd,"\beta_{\loc,23}","\sim"' sloped] && \\[-10pt]
 &(B \otimes_T B) \otimes_T (B\otimes_T B)_{(13,14,23,24)}\ar[r,"m\otimes_T m"]  & (B\otimes_T B)_{(12)}
\end{tikzcd}
\end{equation}
commutes, where $u$ is the inclusion into the localisation by the $(14,32)$-factors.
\end{defn}

In the above, we have extended $m\otimes_T m$ to a map on the localisation
$$ B^{\otimes_T4} _{(13,23)} \ \simeq \ B^{\otimes_T4}[(\delta\times \delta)^*S^{-1}] \ \stackrel{m\otimes_Tm}{\to} \ B^{\otimes_T2}[S^{-1}]\ = \ B^{\otimes_T 2}_{(12)}$$
where the first isomorphism is given by applying \eqref{eqn:MultiplicativeSubsetHexagon} to give that the localisations by $S_{13}S_{14}S_{23}S_{24}$ and $(\delta\times \delta)^*S$ are isomorphic, and the middle map is induced by $m\otimes_T m$ by $\Ol$-linearity of $m\otimes_T m$.

\subsubsection{Localised structures} \label{ssec:LocalisedProduct} We now specialise to the case of cohomology of quiver moduli stacks as in subsection \ref{ssec:loc_coprod_ben}.  Given a collection of $\Zb$-graded modules $V_{\alpha_i}$ over the cohomology rings $\Ht^{\sbt}(\Ml^T_{Q,d_i})$, we define their \textit{localised monoidal product} by
$$\left(V_{d_1}\otimes_T V_{d_2}\right)_{(12)} \ = \ V_{d_1}\otimes_T V_{d_2}[S_{d_1,d_2}^{-1}]$$
and likewise for iterated tensor products:
\begin{equation}
  \label{eqn:LocalisedProduct} 
  \left(V_{d_1}\otimes_T \cdots \otimes_T V_{d_n}\right)_{(i_1j_1,i_2j_2, \ldots\ )} \ = \ \left(V_{d_1}\otimes_T \cdots \otimes_T V_{d_n}\right)\left[S_{i_1j_1}^{-1},S_{i_2j_2}^{-1}, \cdots \ \right].
\end{equation} 
The \textbf{localised braiding} on $\otimes_T$ is
\begin{align}
    \beta_{\loc} \ : \  (V_{\alpha}\otimes_T W_{\beta})_{(12)} &\ \stackrel{\sim}{\to} \ (W_{\beta}\otimes_T V_{\alpha})_{(21)}\\
     v \otimes w  & \ \mapsto \ \sigma\cdot \left(R_{\loc}\cdot v\otimes w\right) \nonumber\\
     & \ = \  (-1)^{|v|\cdot |w|}\sigma(R_{\loc})\cdot w\otimes v\nonumber
\end{align} 
where
$$R_{\loc} \ = \ \frac{e(\sigma^*\Ext^\vee)}{e(\Ext)} \ \in \  \Ht^{\sbt}(\Ml_{Q,d_1}^T)\otimes_{\Ht^{\sbt}(\BT)} \Ht^{\sbt}(\Ml_{Q,d_2}^T)[S_{d_1,d_2}^{-1}]$$
is the \textbf{localised Joyce $R$-matrix}. We emphasise that we \textit{cannot} assemble this into an element of a localisation of the tautological ring, see Warning \ref{war:Euler}. 

\begin{lem} \label{lem:loc_r_to_zr}
  We have $R_{\loc} = \frac{e(Q_1)}{e(Q_1^{\opt})}$.
\end{lem}
\begin{proof}
   We compute
   $$R_{\loc} \ = \ \frac{e(\sigma^*\Ext_0^\vee)e(\Ext_1)}{e(\sigma^*\Ext_1^\vee)e(\Ext_0)} \ = \ \frac{e(\Ext_1)}{e(\sigma^*\Ext_1^\vee)} \ = \ \frac{e(Q_1)}{e(Q_1^{\opt})}$$
   where in the second equality we used that $\sigma^*\Ext_0^\vee \simeq \Ext_0$, and in the third Proposition \ref{prop:euler_ext_to_euler_ben}.
\end{proof}
\begin{lem}
  \label{lem:RLocR} The full Joyce and localised Joyce $R$-matrices are compatible: $(\id \otimes \act_z^*)(R_{\loc}) =  R(z)$.
\end{lem}
\begin{proof}
Follows from Proposition \ref{prop:PsiEuler}.
\end{proof}
\begin{cor} \label{cor:BetaMeromorphicBetaLocalised}
   $\act_2\cdot \beta_{\loc}= \beta(z)\cdot \act_1$.
\end{cor}

\subsubsection{Explicit formula for full Joyce $R$-matrix}
We can now explicitly compute the full Joyce $R$-matrix in terms of chern roots of tautological bundles
\begin{prop}\label{prop:explicit_fullrmatrix}
   We have that 
\begin{equation}
  \label{eqn:RfullComputation} 
   R(z)_{d,d^{'}} \ =  \bigsqcap_{e \, : \, i\to j}\frac{ \bigsqcap_{(n,m)=(1,1)}^{(d_i,d_j')} (1\otimes x_{j,m} - x_{i,n}\otimes 1 + z + \wtt(e))}{ \bigsqcap_{(n,m)=(1,1)}^{(d_j,d_i')} (1\otimes x_{i,m} - x_{j,n}\otimes 1 + z - \wtt(e))}
\end{equation}
 viewed as an element of 
$$\Sym_{\Ht^{\sbt}(\BT)}\left( x_{i,\alpha} \ : \ i\in Q_0,\, 0\le \alpha\le d_{1,i}\right)\otimes_{T} \Sym_{\Ht^{\sbt}(\BT)}\left( x_{i,\beta} \ : \ i\in Q_0,\, 0\le \beta\le d_{2,i}\right) ((z^{-1})).$$
\end{prop}
\begin{proof}
    This is a combination of the Lemmas \ref{lem:loc_r_to_zr}, \ref{lem:RLocR}, Proposition \ref{prop:euler_ext_to_euler_ben} and the fact that $\act^{*}$ sends a chern root of a tautological bundle $x$ to $x+z$.
\end{proof}

\subsubsection{Sign twists} \label{ssec:SignTwist} Consider the \textit{sign twist} map 
\begin{align*}
\tau \ : \ \Nb^{Q_0} \times \Nb^{Q_0}  &\ \to \  \Zb/2  \\ 
(d_1,d_2)& \ \mapsto \ \chi(d_1,d_1)\chi(d_2,d_2) + \chi(d_1,d_2) .
\end{align*} 
We define the twisted meromorphic and localised braidings by
\begin{equation}
  \beta^\tau(z) \ = \ (-1)^\tau \beta(z), \hspace{15mm} \beta_{\loc}^\tau \  = \ (-1)^{\tau}\beta_{\loc}.
\end{equation}
Thus the sign twisted localised braiding $\beta_{\loc}^\tau$ agrees with $\widetilde{\textup{sw}}^{\tau} = \sigma \left( (-1)^{\tau}\frac{e(Q_1^{\opt})}{e(Q_1)}\cdot (-)\right)$ in the notation of \cite[Section 4.1]{BD}. Note in \cite{BD} $\widetilde{\textup{sw}}$ is written already containing $\tau$ but we consider both with and without the $\tau$ sign twist.
This will become useful because the ordinary CoHA product only forms a bialgebra for the $\tau$-twisted braidings, although to match up with Yangians we precompose the CoHA product with another sign twist $\psi$, 
 giving a bialgebra with respect to the untwisted braidings.
 
 Consider any 
$$\psi \ : \ (\Zb/2)^{Q_0}\times (\Zb/2)^{Q_0} \ \to \ \Zb/2$$
satisfying $\psi(d_1,d_2)+\psi(d_2,d_1)=\tau(d_1,d_2)$. Then by \cite[Thm.5.13]{Da2} and  \cite[Prop. 5.5]{BD} in the $T$ equivariant setting, the diagram
\begin{equation}\label{eqn:DavisonMeinhardt}
\begin{tikzcd}[row sep = {50pt,between origins}, column sep = {20pt}]
\Al^T_{Q,W} \otimes_T \Al^T_{Q,W} \ar[rr,"m^{(\psi)}"]  \ar[d,"u\cdot \Delta_{\loc}\otimes_T \Delta_{\loc}"] &[-220pt] &[10pt]\Al^T_{Q,W} \ar[dd,"\Delta_{\loc}"]  \\ 
(\Al^T_{Q,W}\otimes_T \Al^T_{Q,W}) \otimes_T (\Al^T_{Q,W}\otimes_T \Al^T_{Q,W})_{(12,14,32,34)} \ar[rd,"\sim"' 
sloped, "\widetilde{\textup{sw}}^{(\tau)}_{23}"] & & \\[-15pt]
 &(\Al^T_{Q,W}\otimes_T \Al^T_{Q,W}) \otimes_T (\Al^T_{Q,W}\otimes_T \Al^T_{Q,W})_{(13,14,23,24)} \ar[r,"m^{(\psi)}\otimes_T m^{(\psi)}"]  & (\Al^T_{Q,W}\otimes_T \Al^T_{Q,W})_{(12)}
\end{tikzcd}
\end{equation}
commutes, where $u$ is the map into the further localisation by the $(14,32)$-factors, and where we are allowed to move the sign factor onto either the braiding or the multiplication, i.e. we have
$$(\widetilde{\textup{sw}}^{(\tau)}, m^{(\psi)}) \ = \ (\widetilde{\textup{sw}}^\tau, m) , \, \textup{or}\, (\widetilde{\textup{sw}},m^\psi = m\cdot (-1)^\psi).$$
We denote by $\Al^{T,\psi}_{Q,W}$ the vector space $\Al^T_{Q,W}$ but with the twisted version
\begin{equation}
  \label{eqn:TwistedCoHAProduct} 
  m^\psi\ =\ m\cdot (-1)^\psi
\end{equation}
of the CoHA product $m$ defined in subsection \ref{ssec:coha_prelims}.

\subsection{The CoHA is a vertex bialgebra} 
\label{ssec:Bialgebra}

\subsubsection{} We come to the first main result of the paper about arbitrary quivers $Q$ with potential $W$, graded over the character lattice of a torus $T$. See the sections \ref{ssec:MotivationQuantumGroups} and \ref{ssec:MotivationPhysics} for motivations from the theory of quantum groups and physics for why one should have expected this structure.  We begin with the symmetric case.

\begin{theorem}\label{thm:CoHABialgebra}
  Let $Q$ be symmetric. $\Al^{T,\psi}_{Q,W}$ is a vertex bialgebra inside $\Ht^{\sbt}(\Ml_Q^T)_{\taut}\Md(\Vect_\Lambda^T)$. Its Joyce--Liu vertex coproduct $\Delta(z)$ is braided colocal.
\end{theorem}

Unwinding the Definition \ref{defn:VertexBialgebra} of vertex bialgebra, this means the following:
\begin{enumerate}
    \item \label{item:BVect} Firstly $B= \Al^{T,\psi}_{Q,W}$ is an element of this category: it is a $\Lambda$-graded vector space with an action of $\Ht^{\sbt}(\BT)$, a cohomological $\Zb$-grading, and moreover an action of the tautological ring by cup product, which we denote by 
    $$\Ht^{\sbt}(\Ml^T_{Q})_{\taut} \otimes_T \Al^T_{Q,W} \ \stackrel{\cdot}{\to} \ \Al^T_{Q,W}.$$

\item \textbf{Tautological ring}. \label{item:Tautological} $H=\Ht^{\sbt}(\Ml_{Q}^T)_{\taut}$ is a commutative, cocommutative bialgebra with its cup product and $\oplus^*$ coproduct. It has a $\Cb[z]$-coaction $\act^*:H \to H[z]$ and hence forms a holomorphic vertex bialgebra with vertex coproduct $\oplus^*(z)= \act_{z,1}^*\oplus^*$.  There is an element 
$$R(z) \ \in \ H \otimes H((z^{-1}))$$
satisfying the spectral hexagon relations for $\oplus^*(w)$ as in Lemma \ref{lem:SpectralHexagon}.

We can thus define bialgebras in $H\Md$ as in Definition \ref{defn:VertexBialgebra}.

 \item \textbf{CoHA and cup product}.   $B =\Al^{T,\psi}_{Q,W}$ has an associative product $\star: B\otimes_T B  \stackrel{m^\psi}{\to}  B $   which is linear over the cup product action of $H$:
\begin{equation}
        m^\psi(\oplus^{*}(h) \cdot (b\otimes b')) \ = \  h \cdot m^\psi(b\otimes b').
    \end{equation}
  Thus $B$ is an associative algebra internal to this category.\label{item:CoHACup}

\item \textbf{$\Delta(z)$ and cup product}.   There is a coassociative vertex coproduct $\Delta(z) :  B \to  B \otimes_T B((z^{-1}))$  which is linear over the cup product action of $H$:
\begin{equation}
  \Delta(h\cdot b,z) \ = \ \oplus^*(h,z)\cdot \Delta(b,z).
\end{equation}
Thus $B$ is a coassociative vertex coalgebra internal to this category. \label{item:DeltaCup}

\item  \textbf{CoHA and $\Delta(z)$ form a vertex bialgebra} for the meromorphic braiding $\beta(z)= \sigma(R(z)\cdot (-))$ as in Definition 
\ref{defn:VertexBialgebra}:
\begin{equation}
\Delta(b\star b',z) \ = \ \Delta(b,z)\star_{R(z)}\Delta(b',z),
\end{equation}
\label{item:CoHACoProd} 

\item \textbf{Braided colocality}. \label{item:BraidedColocality} The Joyce--Liu coproduct is cocommutative up to the meromorphic braiding: 
    \begin{equation} \label{colocality_1}
    \Delta_{d_{1},d_{2}}(z) =   (-1)^{\rank \Ext_{d_{1},d_{2}}}\sigma_{d_{1},d_{2}} \cdot R_{d_{2},d_{1}}(z)\Delta_{d_{2},d_{1}}(-z)e^{zT} 
\end{equation}
\end{enumerate}

The same Theorem is true for the untwisted CoHA $\Al^T_{Q,W}$, except we use the $\tau$-twisted meromorphic braiding $\beta^\tau(z)$. Most proofs of properties such as \eqref{item:CoHACoProd} work by a torus localisation argument, but we sidestep this by using Davisons' version \cite{Da2} of this statement in the localised case.

\begin{cor}
  If $Q$ is graded symmetric, then $\Al^{T,\psi}_{Q,W}$ is a $\Lambda\times \Zb$-graded $\Ht^{\sbt}(\BT)$-linear vertex bialgebra, where colocality is satisfied up to a sign depending on the Euler form of $Q$. 
\end{cor}
\begin{proof}
  When the Euler form $\chi$ is symmetric, the vertex coproduct becomes
  $$\Delta_{\alpha+\beta}(z) \ : \ B_{\alpha+\beta} \ \to \ B_\alpha\otimes_{\Ht^{\sbt}(\BT)}B_\beta((z^{-1}))$$
  without a cohomological shift by the definition \eqref{eqn:OtimesTDefinition} of $\otimes_T$, and since $R(z)=1\otimes 1$ the braiding becomes $\beta(z)=\sigma$ the Koszul sign rule braiding, so colocality 
  $$\Delta(\alpha,z) \ = \ (-1)^{\rank \Ext}\sigma(\Delta(\alpha,-z) e^{zT})$$
  becomes ordinary (super)colocality up to sign.
\end{proof}
\begin{proof}[Proof of Theorem \ref{thm:CoHABialgebra}] All structures \eqref{item:Tautological} on the tautological ring were proven in section \ref{ssec:StructuresTautologicalRing}. On $B=\Al^{T,\psi}_{Q,W}$, we let $m^\psi$ be the sign-twisted associative CoHA product \eqref{eqn:TwistedCoHAProduct} and $\Delta(z)$ the Joyce--Liu coassociative vertex coproduct of Theorem \ref{thm:JoyceVCA}.

\textbf{Compatibility of CoHA with cup product }\eqref{item:CoHACup}. We compute
    \begin{align*}
        p_{*}q^{*}(\oplus^{*}(h) \cup (b \otimes b')) & = p_{*}(q^{*} \oplus^{*} h \cup q^{*}(b \otimes b'))  \\
        & = p_{*}(q^*(p \cdot s)^*h \cup q^{*}(b \otimes b')) \\
        & = p_{*}(q^* s^{*} p^{*}h \cup q^{*}(b \otimes b')) \\
        & = p_{*}(p^{*}h \cup q^{*}(b \otimes b')) \quad \text{projection formula \eqref{eqn:ProjectionFormulaVanishing}}\\
        & = h \cup p_{*}q^{*}(b \otimes b').
    \end{align*}
\textbf{Compatibility of $\Delta(z)$ with cup product }\eqref{item:DeltaCup}. We use Lemma \ref{lem:linearity_two_acts}, which also works for the tautological ring, to compute
\begin{align*}
    \Psi(\Ext,-z) \cdot \act^{*}_{1}(\oplus^{*}(\alpha \cdot \beta)) 
    & =  \Psi(\Ext,-z) \cdot \act^{*}_{1}(\oplus^{*}(\alpha) \cdot\oplus^{*} (\beta)) \\
    & =  \Psi(\Ext,-z) \cdot \act^{*}_{1}(\oplus^{*}(\alpha)) \cdot  \act^{*}_{1}(\oplus^{*} (\beta)) \\
    &= \act^{*}_{1}(\oplus^{*}(\alpha))  \Psi(\Ext,-z) \cdot \\
    & = \oplus^{*}(\alpha,z) \cdot \Delta(\beta,z).
\end{align*}

\textbf{Compatibility of $\Delta(z)$ and the CoHA} \eqref{item:CoHACoProd}. We use Davisons' result \eqref{eqn:DavisonMeinhardt} that $B = \Al^{T,\psi}_{Q,W}$ forms a localised bialgebra, i.e. that the inner face of 
\begin{center}
\begin{tikzcd}[row sep = {70pt,between origins}, column sep = {70pt, between origins}]
 B \otimes_T B \ar[rrrr,"m"]\ar[ddd,"\Delta(z) \otimes_T \Delta(z)"']   &[30pt]&[-30pt] &[80pt]&[-20pt] B \ar[dddd,"\Delta(z)"] \\[-20pt]
&B \otimes_T B \ar[rr,"m"] \ar[d,"u \cdot \Delta_{\loc} \otimes_T \Delta_{\loc}"'] \ar[lu,thin,"\sim" sloped,"\id"' sloped]   & &B  \ar[dd,"\Delta_{\loc}"] \ar[ru,thin,"\sim" sloped,"\id"' sloped] &\\ 
 &(\boldsymbol{B} \otimes_T B) \otimes_T (\boldsymbol{B}\otimes_T B)_{(12,14,32,34)}\ar[rd,"\beta_{\loc \textup{,} 23}"',"\sim" sloped]  \ar[ld, thin,"\act_{1,3}" sloped] && &\\[-10pt]
|[yshift=25pt]|(B \otimes_T B) \otimes_T (B\otimes_T B)((z^{-1})) \ar[rd,"\beta_{23}(z)"',"\sim" sloped] & &(\boldsymbol{B}\otimes_T \boldsymbol{B}) \otimes_T (B \otimes_T B)_{(13,14,23,24)}\ar[r,"m\otimes_T m"'] \ar[ld, thin,"\act_{1,2}" sloped]  & (\boldsymbol{B}\otimes_T B)_{(12)} \ar[rd,thin,"\act_1" sloped]&\\
& |[xshift=-40pt]|(B \otimes_T B) \otimes_T (B\otimes_T B) ((z^{-1})) \ar[rrr,"m\otimes_T m"'] & && B \otimes_T B ((z^{-1}))
\end{tikzcd} 
\end{center}
commutes. Here, the arrows are the $\loc$ maps from Definition \ref{defn:VertexFromLocalised} induced by $\act^*$ acting on the copies of $B$ drawn in bold, and the subscript $(ij)=[S_{ij}^{-1}]$ means a localisation along a copy of $S$ acting on the $ij$th tensor factors as in section \ref{ssec:LocalisedProduct}.  

The remaining five faces commute for the following reasons. The top face is the identity, and the bottom face commutes by linearity of the CoHA product over action by the cup product. Note we can commute the map $\act_{1,3}$ past the localisation map $u$ because the map $\act_{1,3}$, applied to the appropriate classes we need to localise, will be invertible. The remaining faces commute because 
$$\Delta(z)\cdot \act \ = \ \act\cdot \Delta_{\loc}, \hspace{15mm} \beta(z)\cdot \act \ = \ \act\cdot \beta_{\loc}$$
by Theorem \ref{thm:DavisonIsJoyce} and Corollary \ref{cor:BetaMeromorphicBetaLocalised}. Thus the outer face commutes, and $B$ forms a vertex bialgebra.

\textbf{ Braided colocality} \eqref{item:BraidedColocality}. We have that
\begin{align*}
  \Delta(\alpha,z) &\ = \ \Psi(\Ext,-z)\act_{z,1}^*\oplus^*(\alpha) \\
  & \ = \  \sigma\left(\Psi(\sigma^*\Ext,-z)\act_{z,2}^*\oplus^*(\alpha)\right) \\
  & \ = \   \sigma\left(\frac{\Psi(\sigma^*\Ext,-z)}{\Psi(\Ext,z)}\cdot \Psi(\Ext,z)\act_{-z,1}^*\oplus^*(\alpha)\act_{z,1}^*\right) \\
  &\ = \ (-1)^{\rank \sigma^*\Ext}\sigma\left(R(z)\cdot \Delta(\alpha,-z)e^{zT}\right) \\
  &\ = \ (-1)^{\rank \Ext}\beta(z)\cdot \Delta(\alpha,-z)e^{zT}
\end{align*}
where in the second equality we used that $\oplus^*$ is cocommutative and in the third equality we used the fact that 
$$\act_{z,2}^*\oplus^* \ = \ \act_{-z,1}^* \act_{z,1}^*\act_{z,2}^*\oplus^* \ = \ \act_{-z,1}^*\oplus^*\act_z^*.$$
This shows braided colocality.

\end{proof}

\subsubsection{Non-symmetric case} \label{ssec:non_symm}
 If $Q$ is not symmetric, we do not currently have a good generalisation $H'$ of the tautological ring with an element $R'(z) \in H\otimes H((z^{-1}))$ satisfying the spectral hexagon relations. Indeed, if we took $H'=\Ht^{\sbt}(\Ml_Q^T)_{\taut}$ itself then $R_{\taut}(z)$ does not in general satisfy the spectral hexagon relations, and $R(z)\not \in H\otimes H((z^{-1}))$.

We do not address this question in this paper.

 As a stopgap, we can consider the category 
$\Cl = \bigsqcap_{d \in \Nb^{Q_{0}}} \Ht^{\sbt}(\Ml_{Q,d}^T)\Md$ whose objects consist of a collection $V=(V_d)$ of $\Ht^{\sbt}(\Ml_{Q,d}^T)$-modules for every dimension vector $d$. There is a monoidal structure 
$$(V\otimes_T W)_d \ = \ \bigoplus_{d=d_1+d_2}V_{d_1}\otimes_T W_{d_2}$$
which is additive on dimension vectors, and $\Ht^{\sbt}(\Ml_{Q,d_{1}+d_{2}}^T)$ acts on the tensor via the direct sum pullback $\oplus^*$. The monoidal structure $\otimes_T$ has a natural meromorphic braiding given by 
\begin{align*}
    \beta(z) \ : \ V_d \otimes_T W_{d'}((z^{-1})) & \ \stackrel{\sim}{\to} \ W_{d'}\otimes_T V_d((z^{-1}))\\
    v \otimes w &\ \mapsto \ \sigma(R(z)_{d,d'}\cdot v\otimes w).
\end{align*}
We may view the CoHA as an element $B=(\Al_{Q,W,d}^{T,\psi})_{d \in \Nb^{Q_{0}}}$ of this category.

Thus--until we have a good analogue of the tautological ring for nonsymmetric quivers--the best analogue of Theorem \ref{thm:CoHABialgebra} we have is
\begin{theorem}\label{thm:CoHABialgebra_non_symm}
Let $(Q,W)$ be an arbitrary quiver with potential acted on by torus $T$ leaving the potential invariant, and satisfying the K{\"u}nneth assumption \eqref{eqn:KunnethAssumption}. The CoHA $\mathcal{A}^{T,\psi}_{Q,W}$ is a vertex bialgebra inside the category $\Cl$.
\end{theorem}
\begin{proof}
    The same as for \ref{thm:CoHABialgebra}.
    \end{proof}

Note however that because $\Cl$ is not a module category, we cannot apply bosonisation in the sense of next section to add a Cartan piece to $\Al^{T,\psi}_{Q,W}$. See subsection \ref{ssec:nonsymm_boson} for more discussion.

\newpage
\section{Extending CoHAs and bosonisation}

\noindent In this section, we explain how to extend CoHAs systematically. In the case of any symmetric $N$-graded quiver $Q$ with potential $W$ we will define the \textit{extended CoHA}
$$\Al^{T,\psi,\ext}_{Q,W} \ = \ \Al^{T,\psi}_{Q,W} \otimes_{T} \Ht^{\sbt}(\Ml_{Q}^T)_{\taut}$$
and show that it inherits algebraic structures from $\Al^{T,\psi}_{Q,W}$. We prove the following in section \ref{ssec:ProofCoHABosonisation}:

\begin{theorem}\label{thm:CoHABosonisation}
  Let $Q$ be a symmetric quiver. The extended CoHA has an associative product and nonlocal vertex coproduct 
   \begin{align}\label{eqn:CoHABosonisedProduct}
       b \otimes h\, \cdot \, b'\otimes h'& \ = \ \left[b \cdot ((\oplus^*h)_{(1)}\cup b')\right] \otimes \left[(\oplus^*h)_{(2)}\cup h'\right]\\[5pt]
       \label{eqn:CoHABosonisedCoproduct}
      \Delta^{\ext}(b\otimes h,z) &\ = \ R_{\taut}(z)_{32}\cup \left(  \Delta(b,z)_{13}\otimes (\act_{z,1}^*\oplus^* h)_{24} \right)
  \end{align}
  endowing it with the structure of a vertex bialgebra inside $\Vect_\Lambda^T$.
\end{theorem}
This statement and its proof is an application of a much more general phenomenon that we call vertex Majid--Radford bosonisation (Theorem \ref{thm:VertexBosonisation}). In pictures, the extended product and coproduct are
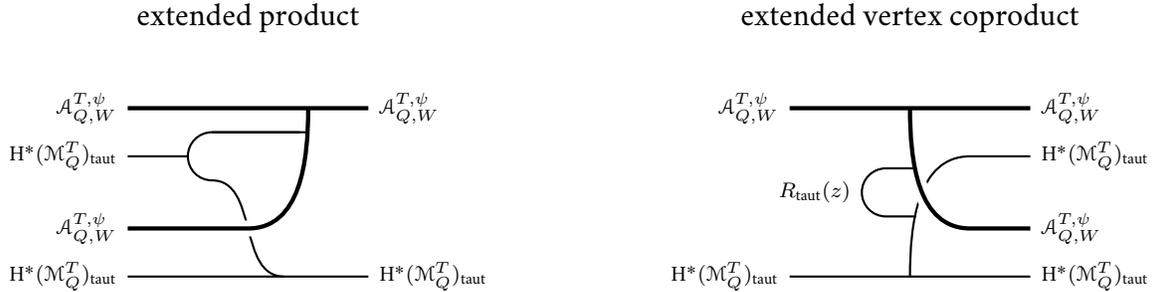
\begin{center}
\begin{tikzpicture}[scale=1.6]
 
    \begin{scope} 
  \draw[thick] (0,-0.4) -- (0.5,-0.4);
  \draw[thick] (1.5,-0.2) -- (0.7,-0.2) to[out=180,in=90] (0.5,-0.4) to[out=270,in=180] (0.7,-0.6) to[out=0,in=180]  (1.3,-1.4); 
  \draw[thick] (0,-1.4) -- (2,-1.4);

  \draw[white, line width = 6pt] (0,-1) -- (1,-1) to[out=0,in=270] (1.5,0);
   \draw[ultra thick] (0,-1) -- (1,-1) to[out=0,in=270] (1.5,0);
   \draw[ultra thick] (0,0) -- (2,0);
   \draw[thick] (1.5,-0.2) -- (0.7,-0.2);
    \node[] at (1,0.75) {extended product};

    \node[left] at (0,0) {$\scriptstyle\Al^{T,\psi}_{Q,W}$};
    \node[left] at (0,-0.4) {$\scriptstyle\Ht^{\sbt}(\Ml^T_{Q})_{\taut}$};
    \node[left] at (0,-1) {$\scriptstyle\Al^{T,\psi}_{Q,W}$};
    \node[left] at (0,-1.4) {$\scriptstyle\Ht^{\sbt}(\Ml^T_{Q})_{\taut}$};

    \node[right] at (2,0) {$\scriptstyle\Al^{T,\psi}_{Q,W}$};
    \node[right] at (2,-1.4) {$\scriptstyle\Ht^{\sbt}(\Ml^T_{Q})_{\taut}$};
     \end{scope}

   \begin{scope} 
    [xshift=4.5cm]
    
     \begin{scope} 
    \draw[thick] (1,-1.4) -- (3,-1.4);
    \draw[thick] (2,-1.4) to[out=90,in=180] (2.5,-0.4) -- (3,-0.4);
    
    \draw[line width = 6pt,white] (3,-1) -- (2.5,-1) to[out=180,in=270] (2,0);
     \draw[ultra thick] (3,-1) -- (2.5,-1) to[out=180,in=270] (2,0);
     \draw[ultra thick] (1,0) -- (3,0);
    \draw[thick] (2.05,-0.9) -- (1.8,-0.9) to[out=180,in=270] (1.6,-0.7) to[out=90,in=180] (1.8,-0.5) -- (2.05,-0.5); 

    \node[left] at (1.6,-0.7) { $\scriptstyle R_{\taut}(z)$};
    \node[] at (2,0.75) {extended vertex coproduct};

    \node[left] at (1,0) {$\scriptstyle\Al^{T,\psi}_{Q,W}$};
    \node[left] at (1,-1.4) {$\scriptstyle\Ht^{\sbt}(\Ml^T_{Q})_{\taut}$};

    \node[right] at (3,0) {$\scriptstyle\Al^{T,\psi}_{Q,W}$};
    \node[right] at (3,-0.4) {$\scriptstyle\Ht^{\sbt}(\Ml^T_{Q})_{\taut}$};
    \node[right] at (3,-1) {$\scriptstyle\Al^{T,\psi}_{Q,W}$};
    \node[right] at (3,-1.4) {$\scriptstyle\Ht^{\sbt}(\Ml^T_{Q})_{\taut}$};
      \end{scope}
    \end{scope}
 \end{tikzpicture}
\end{center}
where merging lines corresponds to the CoHA product or cup product, and splitting lines correspond to the Joyce--Liu vertex coproduct $\Delta(z)$ or $\oplus^*$. See \cite{ES} for an introduction to string diagrams for quantum groups.

\subsubsection{Remark} The above Theorem also holds for the untwisted CoHA $\Al^T_{Q,W}$ and $(-1)^{\tau}R_{\taut}(z)$.

\subsection{Majid--Radford bosonisation} 

\subsubsection{} Majid \cite{Ma2} and Radford \cite{Ra} noticed that if $B$ is a vector space with a (right) action of associative algebra $H$, then additional algebraic structures from $H,B$ induce analogous structures on the \textbf{bosonisation} $B\rtimes H = B\otimes H$:

\begin{theorem} \label{thm:Bosonisation} \cite[Thm. 4.1]{MaBos}
If $H$ is a bialgebra and $B$ is an $H$-linear algebra, then
  \begin{equation} \label{eqn:BosonisationProduct}
    (b \otimes h)\cdot (b'\otimes h') \ = \ b(h_{(1)}\cdot b') \otimes h_{(2)}h'
  \end{equation}
  makes $B \rtimes H$ into an associative algebra. If additionally $H$ has a quasitriangular element $R = R^{(1)}\otimes R^{(2)}\in H \otimes H$ and $B$ is a bialgebra in $H\Md$ with respect to the induced braiding, then
  \begin{equation} \label{eqn:BosonisationCoproduct}
    \Delta(b \otimes h) \ =  \  (b_{(1)} \otimes R^{(2)}h_{(1)})\otimes (R^{(1)}\cdot  b_{(2)}\otimes h_{(2)})
  \end{equation}
  makes $B \rtimes H$ into a bialgebra. 
\end{theorem}

This is obvious from a module category perspective. Indeed, applying Barr--Beck to the forgetful functor $\oblv:B\Md(H\Md) \to \Vect$ induces an equivalence 
\begin{equation}
  \label{eqn:TannakianBosonisation} 
  B\Md(H\Md) \ \simeq \ (B\rtimes H)\Md
\end{equation}
for \textit{some} algebra $B\rtimes H= \End(\oblv)$. An object in \eqref{eqn:TannakianBosonisation} is a vector space together with right $H$- and $B$-actions satisfying certain commutation relations determined by the action of $H$ on $B$. Therefore as a vector space we have $B\rtimes H = B\otimes H$, whose algebra structure \eqref{eqn:BosonisationProduct} is fixed by the commutation relations and the observation that $B,H$ are subalgebras. 

Noticing that a monoidal structure on $A\Md$ lifting that of $\Vect$ is precisely equivalent to a bialgebra structure on $A$, the coproduct \eqref{eqn:BosonisationCoproduct} on $B \rtimes H$ is then fixed by asking that the equivalence \eqref{eqn:TannakianBosonisation} is monoidal.

\subsubsection{} In summary, the following structures on $H$ and $B$ induce the following structures on the bosonisation:
\begin{center}
\begin{tabular}{ccc@{\hspace{10pt}}|@{\hspace{30pt}}cc}
 \textit{if} & $H$& $B$ & \textit{then} & $B\rtimes H$\\ 
 {\small\textit{has a}}& product & $H$-action & -- & --\\ 
 {\small\textit{and}}&coproduct & product & {\small\textit{has a}} & product\\ 
 {\small\textit{and}}&quasitriangular  &coproduct & {\small\textit{and}}& coproduct\\[-2pt]
 & element $R$ & & &  
 \end{tabular} 
\end{center}
and we may choose to take the vertex algebra analogue of any row(s). In what follows we will need the last:
\begin{center}
\begin{tabular}{ccc@{\hspace{10pt}}|@{\hspace{30pt}}cc}
 \textit{if} & $H$ & $B$ & \textit{then} & $B\rtimes H$\\ 
 {\small\textit{has a}}& product & $H$-action & -- & --\\ 
 {\small\textit{and}}&coproduct & product & {\small\textit{has a}} & product\\ 
 {\small\textit{and}}&spectral  &vertex & {\small\textit{and}}& vertex\\[-2pt]
 & quasitriangular &coproduct & &coproduct \\[-2pt]
 & element $R(z)$ & & &   
 \end{tabular} 
\end{center}
where now all vector spaces are also endowed with ``translation'' endomorphisms respecting the above structures. We should view Theorem \ref{thm:VertexBosonisation} as saying that we have an equivalence
  $$B\Md(H\Md) \ \simeq \ B \rtimes H \Md$$
which moreover respects the meromorphic tensor structures on both sides, although as mentioned we do not make this precise.

\subsubsection{Remark} 
In the presence of a Tannakian reconstruction statement for meromorphic monoidal categories we could view $\Al^{T,\psi,\ext}_{Q,W}\in \Vect^T_\Lambda$ as the endomorphism algebra of the forgetful functor
\begin{equation}\label{eqn:Forgetful}
   \Al^{T,\psi}_{Q,W}\Md\left(\Ht^{\sbt}(\Ml_{Q}^T)_{\taut}\Md(\Vect^T_\Lambda)\right) \ \to \ \Vect^T_\Lambda
\end{equation}
so that $\Al^{T,\psi,\ext}_{Q,W}\Md(\Vect^T_\Lambda)$ would be equivalent to the left side of \eqref{eqn:Forgetful} and the vertex coproduct is inherited from structure on that category.

\subsection{Vertex bosonisation}

\subsubsection{} Consider vector spaces $H,B$ and $R(z) \in H \otimes H((z^{-1}))$ satisfying the axioms \eqref{item:BVect}-\eqref{item:CoHACoProd} listed below Theorem \ref{thm:CoHABialgebra}. 
For simplicity, we will also assume that the product on $H$ is commutative.

In the following we use the notation for $R$-matrices
$$R(z) \ = \ R^{(1)}(z) \otimes R^{(2)}(z),$$
and the Sweedler notation for coproducts and vertex coproducts
$$\Delta(a) \ = \ a_{(1)}\otimes a_{(2)}, \hspace{15mm} \Delta(a,z) \ = \ a_{(1,z)}\otimes a_{(2,z)},$$
where in both cases the summation is implied. As with usual Sweedler notation, the term $a_{(1,z)}$ will never appear alone without $a_{(2,z)}$.

\begin{theorem} \label{thm:VertexBosonisation}  The vector space $B \rtimes H = B \otimes H$ with product
    \begin{equation} \label{eqn:VertexBosonisationProduct}
   \centering
   \begin{minipage}{.7\textwidth} \centering
   $$(b \otimes h)\cdot (b'\otimes h') \ = \ b(h_{(1)}\cdot b') \otimes h_{(2)}h'$$
   \end{minipage}%
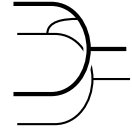
   \begin{minipage}{.3\textwidth} \centering
   
\begin{center}
\begin{tikzpicture}[rotate = -90]
   \draw[thick] (0,2) -- (0,2.5);
   \draw[thick] (-0.6,1) -- (-0.6,1.5) to[out=90,in=-180] (0,2) to[out=0,in=90] (0.6,1.5) -- (0.6,1);
 
\begin{scope} 
 [xshift = -0.4cm,yshift = -0.05cm] 
   \draw[white ultra thick = 4pt] (-0.6,1) -- (-0.6,1.5) to[out=90,in=-180] (0,2) to[out=0,in=90] (0.6,1.5) -- (0.6,1);
   \draw[ultra thick] (0,2) -- (0,2.5);
 \end{scope}

   \draw[thick] (-0.6,1.4) to[out=180,in=-90] (-0.8,1.85-0.05);

 \end{tikzpicture}
 \end{center}
   \end{minipage}
  \end{equation}
  is an associative algebra, and the (nonlocal) vertex coproduct
    \begin{equation} \label{eqn:VertexBosonisationCoproduct}
   \centering
   \begin{minipage}{.7\textwidth} \centering
    $$\Delta(b \otimes h,z) \ = \  (b_{(1,z)} \otimes R^{(2)}(z)h_{(1,z)})\otimes (R^{(1)}(z)\cdot  b_{(2,z)}\otimes h_{(2,z)})$$
   \end{minipage}%
   \begin{minipage}{.3\textwidth} \centering

  \begin{center}
  \begin{tikzpicture}[rotate = -90]
    \draw[thick] (0,-2) -- (0,-2.75);
    \draw[thick] (-0.6,-1.25) --  (-0.6,-1.5) to[out=-90,in=180] (0,-2) to[out=0,in=-90] (0.6,-1.5) -- (0.6,-1.25);
   
  \begin{scope} 
   [xshift = -0.6cm,yshift = -0.05cm] 
    \draw[white ultra thick = 4pt] (-0.6,-1.25) -- (-0.6,-1.5) to[out=-90,in=180] (0,-2) to[out=0,in=-90] (0.6,-1.5) -- (0.6,-1.25);
    \draw[ultra thick] (0,-2) -- (0,-2.75);
   \end{scope}

    \draw[thick] (-0.45,-2.15) --  (-0.45,-2.1) to[out=-90,in=180] (-0.3,-2.3) to[out=0,in=-90] (-0.15,-2.1) -- (-0.15,-2); 
   \end{tikzpicture}
   \end{center}
   \end{minipage}
  \end{equation}
  makes $B\rtimes H$ into a nonlocal vertex bialgebra.
\end{theorem}
\begin{rem}
   There are three proofs of this fact one could give: first by arduous explicit computations as below, and second by an equally arduous string diagram argument. The cleanest way is to deduce it formally from ordinary Majid--Radford bosonisation, by noticing that vertex bialgebras are equivalent to bialgebras in a certain category of sheaves on a Ran space, and likewise for the other structures appearing above, see e.g. \cite{La}.
\end{rem}
\begin{proof}[Proof of Theorem \ref{thm:VertexBosonisation}]$\ $

\textbf{Associativity.}
    \begin{align*}
      \left( (b \otimes h) \cdot (b'\otimes h') \right) \cdot (b'' \otimes h'')& \ = \ \left(  b (h_{(1)}\cdot b') \otimes h_{(2)}h'\right) \cdot (b'' \otimes h'') \\
      & \ = \  b(h_{(1)}\cdot b')( (h_{(2)}h')_{(1)}\cdot b'') \otimes (h_{2}h')_{(2)}h'' \\
      & \ = \  b(h_{(1)}\cdot b')( h_{(2)(1)}h'_{(1)}\cdot b'') \otimes h_{(2)(2)}h'_{(2)}h'' \\
      \intertext{\hfill\small  $H$ is a bialgebra}
      & \ = \  b(h_{(1)(1)}\cdot b')( h_{(1)(2)}h'_{(1)}\cdot b'') \otimes h_{(2)}h'_{(2)}h'' \\
      \intertext{\hfill\small coassociativity of $H$} 
      & \ = \ b \left(h_{(1)}\cdot (b'(h'_{(1)}\cdot b''))\right)\cdot h_{(2)}h'_{(2)} h'' \\
      \intertext{\hfill\small $H$-linearity of $B$'s product}
      & \ = \ (b\otimes h)\cdot \left( b'(h'_{(1)}\cdot b'') \otimes  h'_{(2)} h''\right) \\
      & \ = \ (b\otimes h)\cdot \left( (b' \otimes h')\cdot (b'' \otimes h'')\right)
    \end{align*}

\textbf{Coassociativity.}

 We prove vertex coassociativity 
 $$(\Delta^{\ext}(z)\otimes \id)\Delta^{\ext}(w) \ = \ (\id \otimes \Delta^{\ext}(w))\Delta^{\ext}(z+w)$$
 by computing both sides: 
\begin{align*}
    & \left( \Delta^{\ext}(z)\otimes \id\right)\Delta^{\ext}(b\otimes h, w)\\
   & \hspace{5mm}  \ = \ \left( \Delta^{\ext}(z)\otimes \id\right) \left(b_{(1,w)} \otimes R^{(2)}(w)h_{(1,w)})\otimes (R^{(1)}(w)\cdot  b_{(2,w)}\otimes h_{(2,w)} )\right)\\
   & \hspace{5mm} \ = \  \left[ b_{(1,w)(1,z)} \otimes \underline{R}^{(2)}(z) (R^{(2)}(w)h_{(1,w)})_{(1,z)}\right] \\
   & \hspace{15mm} \otimes \left[\underline{R}^{(1)}(z) \cdot b_{(1,w)(2,z)} \otimes (R^{(2)}(w)h_{(1,w)})_{(2,z)} \right] \\
   & \hspace{25mm}\otimes \left[ R^{(1)}(w)\cdot  b_{(2,w)}\otimes h_{(2,w)} \right]\\
   & \hspace{5mm} \ = \  \left[ b_{(1,w)(1,z)} \otimes \underline{R}^{(2)}(z) R^{(2)}(w)_{(1,z)}h_{(1,w)(1,z)}\right] \\
   &  \hspace{15mm} \otimes \left[\underline{R}^{(1)}(z) \cdot b_{(1,w)(2,z)} \otimes R^{(2)}(w)_{(2,z)}h_{(1,w)(2,z)} \right] \\
    & \hspace{25mm} \otimes \left[ R^{(1)}(w)\cdot  b_{(2,w)}\otimes h_{(2,w)} \right]\\
   \intertext{\hfill\small  $H$\textup{ is a vertex bialgebra}}
   & \hspace{5mm} \ = \  \left[ b_{(1,w)(1,z)} \otimes \underline{R}^{(2)}(z) \dot{R}^{(2)}(w + z)h_{(1,w)(1,z)}\right] \\
   & \hspace{15mm} \otimes \left[\underline{R}^{(1)}(z) \cdot b_{(1,w)(2,z)} \otimes \ddot{R}^{(2)}(w)h_{(1,w)(2,z)} \right] \\
   &\hspace{25mm} \otimes \left[ \dot{R}^{(1)}(z+w)\ddot{R}^{(1)}(w)\cdot  b_{(2,w)}\otimes h_{(2,w)} \right]\\
   \intertext{\small \hfill  hexagon identity for $R(w)$, i.e. $(\Delta_H(z)\otimes \id) R(w)_{21} \ = \ \dot{R}(z+w)_{31}\ddot{R}(w)_{32}$}
   & \hspace{5mm} \ = \  \left[ b_{(1,w)(1,z)} \otimes \underline{\dot{R}}^{(2)}(z+w) h_{(1,w)(1,z)}\right] \\
   & \hspace{15mm} \otimes \left[\underline{\dot{R}}^{(1)}(z+w)_{(1,w)} \cdot b_{(1,w)(2,z)} \otimes \ddot{R}^{(2)}(w)h_{(1,w)(2,z)} \right] \\
   &\hspace{25mm} \otimes \left[ \underline{\dot{R}}^{(1)}(z+w)_{(2,w)}\ddot{R}^{(1)}(w)\cdot  b_{(2,w)}\otimes h_{(2,w)} \right]\\
   \intertext{\small \hfill hexagon identity for $\underline{\dot{R}}(z+w)$, i.e. $(\id \otimes \Delta_H(w)) \underline{\dot{R}}(w + z)_{21} \ = \ \underline{R}(z)_{21}\dot{R}(z+w)_{31} $}
   & \hspace{5mm} \ = \  \left[ b_{(1,z+w)} \otimes \underline{\dot{R}}^{(2)}(z+w) h_{(1,z+w)}\right] \\
   & \hspace{15mm} \otimes \left[\underline{\dot{R}}^{(1)}(z+w)_{(1,w)} \cdot b_{(2,z+w)(1,w)} \otimes \ddot{R}^{(2)}(w)h_{(2,z+w)(1,w)} \right] \\
   &\hspace{25mm} \otimes \left[ \underline{\dot{R}}^{(1)}(z+w)_{(2,w)}\ddot{R}^{(1)}(w)\cdot  b_{(2,z+w)(2,w)}\otimes h_{(2,z+w)(2,w)} \right]\\
   \intertext{\small \hfill coassociativity for $H,B$, i.e. $ h_{(1,w)(1,z)}\otimes h_{(1,w)(2,z)} \otimes h_{(2,w)} \ = \   h_{(1,z+w)}\otimes h_{(2,z+w)(1,w)}\otimes h_{(2,z+w)(2,w)} $}  
   & \hspace{5mm} \ = \ \left[b_{(1,z+w)} \otimes \underline{\dot{R}}^{(2)}(z+w)h_{(1,z+w)} \right]  \\
   & \hspace{15mm}    \otimes \left[\underline{\dot{R}}^{(1)}(z+w)_{(1,w)}\cdot  b_{(2,z+w)(1,w)}\otimes \ddot{R}^{(2)}(w)h_{(2,z+w)(1,w)}\right]  \\
   & \hspace{25mm} \otimes \left[ \ddot{R}^{(1)}(w) \underline{\dot{R}}^{(1)}(z+w)_{(2,w)}\cdot  b_{(2,z+w)(2,w)} \otimes h_{(2,z+w)(2,w)} \right]\\
   \intertext{\small \hfill commutativity of $H$, i.e. $\underline{\dot{R}}^{(1)}(z+w)_{(2,w)}\ddot{R}^{(1)}(w) \ = \ \ddot{R}^{(1)}(w)\underline{\dot{R}}^{(1)}(z+w)_{(2,w)}$}
   & \hspace{5mm} \ = \ \left[b_{(1,z+w)} \otimes \underline{\dot{R}}^{(2)}(z+w)h_{(1,z+w)} \right]  \\
   & \hspace{15mm}    \otimes \left[(R^{(1)}(z+w)\cdot  b_{(2,z+w)})_{(1,w)} \otimes \ddot{R}^{(2)}(w)(h_{(2,z+w)})_{(1,w)}\right]  \\
   & \hspace{25mm} \otimes \left[ \ddot{R}^{(1)}(w)\cdot (\underline{\dot{R}}^{(1)}(z+w)\cdot  b_{(2,z+w)})_{(2,w)} \otimes (h_{(2,z+w)})_{(2,w)} \right]\\
   \intertext{\small \hfill $H$ is a vertex bialgebra}
   &  \hspace{5mm} \ = \ \left(\id \otimes \Delta^{\ext}(w)\right)\left((b_{(1,z+w)} \otimes \underline{\dot{R}}^{(2)}(z+w)h_{(1,z+w)})
   \otimes (\underline{\dot{R}}^{(1)}(z+w)\cdot  b_{(2,z+w)}\otimes h_{(2,z+w)} )\right)\\
   &  \hspace{5mm} \ = \ \left(\id \otimes \Delta^{\ext}(w)\right)\Delta^{\ext}(b \otimes h, z+w).
\end{align*}

\textbf{Bialgebra axiom}

Finally, it remains to show that $B \otimes H$ satisfies the bialgebra axiom, i.e.  
$$\Delta^{\ext}((b \otimes h)\cdot(b' \otimes h'),z) = \Delta^{\ext}(b \otimes h,z) \cdot \Delta^{\ext}(b' \otimes h',z).$$
Again, we compute both sides:

\begin{align*}
 & \Delta(b\otimes h \ \cdot \ b' \otimes h',z)\\ 
 & \hspace{5mm}  \ = \ \Delta(b(h_{(1)}\cdot b') \otimes h_{(2)}h',z)\\
 & \hspace{5mm}  \ = \ \left[ (b(h_{(1)}\cdot b'))_{(1,z)} \otimes R^{(2)}(z)(h_{(2)}h')_{(1,z)} \right] \\
 & \hspace{15mm} \otimes \left[ R^{(1)}(z) \cdot (b(h_{(1)}\cdot b'))_{(2,z)} \otimes (h_{(2)}h')_{(2,z)} \right]\\
 & \hspace{5mm}  \ = \ \left[ (b(h_{(1)}\cdot b'))_{(1,z)} \otimes R^{(2)}(z)h_{(2)(1,z)}h'_{(1,z)} \right] \\
 & \hspace{15mm}  \otimes \left[ R^{(1)}(z) \cdot (b(h_{(1)}\cdot b'))_{(2,z)} \otimes h_{(2)(2,z)}h'_{(2,z)} \right]\\
 \intertext{\hfill\small $H$ is a vertex bialgebra in $\Vect$}
 & \hspace{5mm}  \ = \ \left[ b_{(1,z)} (\underline{R}^{(2)}(z)  (h_{(1)}\cdot b')_{(1,z)}) \otimes R^{(2)}(z)h_{(2)(1,z)}h'_{(1,z)} \right] \\
 & \hspace{15mm} \otimes \left[ R^{(1)}(z) \cdot(\underline{R}^{(1)}(z) b_{(2,z)})(h_{(1)}\cdot b')_{(2,z)} \otimes h_{(2)(2,z)}h'_{(2,z)} \right]\\
 \intertext{\hfill\small $B$ is a vertex bialgebra in $(H\Md,\underline{R}(z))$}
 & \hspace{5mm}  \ = \ \left[ b_{(1,z)} (\underline{R}^{(2)}(z)  h_{(1)(1,z)}\cdot b'_{(1,z)}) \otimes R^{(2)}(z)h_{(2)(1,z)}h'_{(1,z)} \right] \\
 & \hspace{15mm} \otimes \left[ R^{(1)}(z) \cdot(\underline{R}^{(1)}(z) b_{(2,z)})(h_{(1)(2,z)}\cdot b'_{(2,z)}) \otimes h_{(2)(2,z)}h'_{(2,z)} \right]\\
 \intertext{\hfill\small $H$ is a vertex bialgebra in $\Vect$}
 & \hspace{5mm}  \ = \ \left[ b_{(1,z)} (\underline{R}^{(2)}(z)h_{(1)(1,z)}\cdot ( b'_{(1,z)})) \otimes R^{(2)}(z)h_{(2)(1,z)}h'_{(1,z)} \right] \\
 & \hspace{15mm} \otimes \left[ R^{(1)}(z) \cdot(\underline{R}^{(1)}(z)\cdot b_{(2,z)})(h_{(1)(2,z)}\cdot b'_{(2,z)}) \otimes h_{(2)(2,z)}h'_{(2,z)} \right]\\
 & \hspace{5mm}  \ = \ \left[ b_{(1,z)} (\underline{R}^{(2)}(z)h_{(1)(1,z)}\cdot ( b'_{(1,z)})) \otimes R^{(2)}(z)h_{(2)(1,z)}h'_{(1,z)} \right] \\
 & \hspace{15mm} \otimes \left[ ( R^{(1)}(z)_{(1)}\underline{R}^{(1)}(z)\cdot b_{(2,z)})( R^{(1)}(z)_{(2)}h_{(1)(2,z)}\cdot b'_{(2,z)}) \otimes h_{(2)(2,z)}h'_{(2,z)} \right]\\
 \intertext{\hfill\small $B$'s product is $H$-linear, i.e. $R(z)\cdot (\dot{b} \ddot{b}) \ = \ (R(z)_{(1)}\cdot \dot{b}) (R(z)_{(2)}\cdot \ddot{b})$}\\
 & \hspace{5mm}  \ = \ \left[ b_{(1,z)} (\underline{R}^{(2)}(z)  h_{(1)(1,z)}\cdot ( b'_{(1,z)})) \otimes \dot{R}^{(2)}(z) \ddot{R}^{(2)}(z)h_{(2)(1,z)}h'_{(1,z)} \right] \\
 & \hspace{15mm} \otimes \left[ (\dot{R}^{(1)}(z) \underline{R}^{(1)}(z)\cdot b_{(2,z)})(\ddot{R}^{(1)}(z)h_{(1)(2,z)}\cdot b'_{(2,z)}) \otimes h_{(2)(2,z)}h'_{(2,z)} \right]\\
 \intertext{\hfill\small hexagon relation for $R(z)$, i.e. $(\id\otimes\Delta_H(0))R(z) \ = \ \dot{R}_{12}(z)\ddot{R}_{13}(z)$}\\
 & \hspace{5mm}  \ = \ \left[ b_{(1,z)} (\dot{R}^{(2)}(z)_{(1)}  h_{(1)(1,z)}\cdot ( b'_{(1,z)})) \otimes \dot{R}^{(2)}(z)_{(2)} \ddot{R}^{(2)}(z)h_{(2)(1,z)}h'_{(1,z)} \right] \\
 & \hspace{15mm} \otimes \left[ (\dot{R}^{(1)}(z)\cdot  b_{(2,z)})(\ddot{R}^{(1)}(z)h_{(1)(2,z)}\cdot b'_{(2,z)}) \otimes h_{(2)(2,z)}h'_{(2,z)} \right]\\
 \intertext{\hfill\small hexagon relation, i.e. $(\Delta_H(0)\otimes \id)\dot{R}(z) \ = \ \dot{R}_{13}(z)\underline{R}_{23}(z)$}\\
& \hspace{5mm} \ = \ \left[ b_{(1,z)} ((\dot{R}^{(2)}(z)_{(1)}h_{(1)(1,z)})\cdot (b'_{(1,z)})) \otimes \dot{R}^{(2)}(z)_{(2)} h_{(2)(1,z)} \ddot{R}^{(2)}(z)h'_{(1,z)} \right] \\
& \hspace{15mm} \otimes \left[(\dot{R}^{(1)}(z)\cdot (b_{(2,z)})) (h_{(1)(2,z)}\ddot{R}^{(1)}(z)\cdot b'_{(2,z)})  \otimes h_{(2)(2,z)} h'_{(2,z)}) \right] \\
\intertext{\hfill\small $H$ commutative}
& \hspace{5mm} \ = \ \left[ b_{(1,z)} ((\dot{R}^{(2)}(z)_{(1)}h_{(1,z)(1)})\cdot (b'_{(1,z)})) \otimes \dot{R}^{(2)}(z)_{(2)} h_{(1,z)(2)} \ddot{R}^{(2)}(z)h'_{(1,z)} \right] \\
& \hspace{15mm} \otimes \left[(\dot{R}^{(1)}(z)\cdot (b_{(2,z)})) (h_{(2,z)(1)}\cdot\ddot{R}^{(1)}(z)\cdot b'_{(2,z)})  \otimes h_{(2,z)(2)} h'_{(2,z)}) \right] \\
\intertext{\hfill\small $H$ is cocommutative, so $h_{(1)(1,z)}\otimes h_{(2)(1,z)}\otimes h_{(1)(2,z)}\otimes h_{(2)(2,z)} \ = \ h_{(1,z)(1)}\otimes h_{(1,z)(1)}\otimes h_{(2,z)(1)}\otimes h_{(2,z)(2)}$ }
& \hspace{5mm} \ = \ \left[ b_{(1,z)} ((\dot{R}^{(2)}(z)h_{(1,z)})_{(1)} \cdot (b'_{(1,z)})) \otimes (\dot{R}^{(2)}(z)h_{(1,z)})_{(2)}  \ddot{R}^{(2)}(z)h'_{(1,z)} \right] \\
& \hspace{15mm} \otimes \left[(\dot{R}^{(1)}(z)\cdot (b_{(2,z)})) (h_{(2,z)(1)}\cdot\ddot{R}^{(1)}(z)\cdot b'_{(2,z)})  \otimes h_{(2,z)(2)} h'_{(2,z)}) \right] \\
\intertext{\hfill\small $H$ is a bialgebra}
& \hspace{5mm} \ = \ \left[ (b_{(1,z)} \otimes \dot{R}^{(2)}(z)h_{(1,z)})\cdot ( (b'_{(1,z)}) \otimes \ddot{R}^{(2)}(z)h'_{(1,z)}) \right] \\
& \hspace{15mm} \otimes \left[(\dot{R}^{(1)}(z)\cdot (b_{(2,z)}) \otimes h_{(2,z)})\cdot (\ddot{R}^{(1)}(z)\cdot b'_{(2,z)} \otimes h'_{(2,z)}) \right] \\
& \hspace{5mm} \ = \ \left[ (b_{(1,z)} \otimes \dot{R}^{(2)}(z)h_{(1,z)})\cdot (( b'_{(1,z)} \otimes \ddot{R}^{(2)}(z)h'_{(1,z)})) \right] \\
& \hspace{15mm} \otimes \left[((\dot{R}^{(1)}(z)\cdot b_{(2,z)} \otimes h_{(2,z)}))\cdot (\ddot{R}^{(1)}(z)\cdot b'_{(2,z)} \otimes h'_{(2,z)}) \right] \\
& \hspace{5mm} \ = \ \left( b_{(1,z)} \otimes \dot{R}^{(2)}(z)h_{(1,z)} \, \otimes \, \dot{R}^{(1)}(z)\cdot b_{(2,z)} \otimes h_{(2,z)}\right)\\
& \hspace{15mm} \cdot \left( b'_{(1,z)} \otimes \ddot{R}^{(2)}(z)h'_{(1,z)}\,\otimes \, \ddot{R}^{(1)}(z)\cdot b'_{(2,z)} \otimes h'_{(2,z)}\right)\\
& \hspace{5mm} \ = \ \Delta(b\otimes h,z) \cdot \Delta(b' \otimes h',z).
\end{align*}

\end{proof}

\subsection{Bosonising quiver CoHAs}

\subsubsection{Proof of Theorem \ref{thm:CoHABosonisation}} \label{ssec:ProofCoHABosonisation} We apply the vertex bosonisation Theorem \ref{thm:VertexBosonisation} to 
$$B \ = \ \Al_{Q,W}^{T,\psi}, \hspace{15mm} H \ = \ \Ht^{\sbt}_T(\Ml_Q)_{\taut}$$
inside the category $\Vect^T_\Lambda$. Note that Theorem \ref{thm:CoHABialgebra} is precisely the statement that $H=\Ht^{\sbt}_T(\Ml_Q)_{\taut}$ and $B=\Al_{Q,W}^{T,\psi}$ satisfy the conditions required to apply vertex bosonisation.  It thus remains to compute the bosonised product and vertex coproduct on $\Al_{Q,W}^{T,\psi}\otimes_T \Ht^{\sbt}_T(\Ml_Q)_{\taut}$. Firstly, by equation \eqref{eqn:VertexBosonisationProduct} we have that
\begin{align*}
   (b\otimes h)\cdot (b'\otimes h') &\ = \ b(h_{(1)}\cdot b') \otimes h_{(2)}h' \\
   &\ = \ (b\cdot ((\oplus^*h)_{(1)}\cup b'))\otimes (\oplus^*h)_{(2)}\cup h'
\end{align*} 
and from \eqref{eqn:BosonisationCoproduct} we have
\begin{align*}
    \Delta^{\ext}(b\otimes h,z)&\ =\ (b_{(1,z)}\otimes R^{(2)}(z)h_{(1,z)}) \otimes ( R^{(1)}(z)\cdot b_{(2,z)}\otimes h_{(2,z)})\\
  & \ = \ (\sigma R(z))_{23}\cup \left(  (e^{zT\otimes \id}\oplus^* h)_{13} \otimes \Delta(b,z)_{24}\right).
\end{align*}
Note that in the symmetric setting $R(z) = R_{\textup{taut}}(z)$.
Thus vertex bosonisation Theorem \ref{thm:VertexBosonisation} implies that this forms a nonlocal vertex bialgebra inside $\Vect^T_\Lambda$.

\subsubsection{Explicit formula} The bosonised coproduct is 
$$\Delta^{\ext}(b\otimes h,z) \ = \ \sigma(\frac{c(\sigma^{*}\Ext^\vee,z^{-1})}{c(\Ext,z^{-1})})_{23} \Psi(\Ext,-z)_{13}\cdot ((\act_{z,1}^*\oplus^*b)_{13}\otimes (\act_{z,1}^*\oplus^*h)_{24})$$
from the definition of $R_{\taut}(z)$ and the vertex coproducts on $\Al^{T,\psi}_{Q,W}$ and $\Ht^{\sbt}(\Ml_Q^T)_{\taut}$. We will continue exploring explicit forms for this coproduct in the following section.

\subsubsection{Non-symmetric case} \label{ssec:nonsymm_boson} When $Q$ is nonsymmetric, $R(z)$ no longer equals $R_{\taut}(z)$. Therefore, the latter does not satisfy the spectral hexagon relation (as $\tilde{\chi}\ne 0$ in Lemma \ref{lem:SpectralHexagon}) which is needed when proving coassociativity of the bosonised $\Delta^{\ext}(z)$, although the proof of the bialgebra property carries through. On the other hand, the former is no longer valued in the tautological ring $\Ht^{\sbt}(\Ml_Q^T)_{\taut}$.

We expect the fix to be redefining the tautological ring as generated by the coefficients of the full Joyce $R$-matrix.

\newpage 
 \section{Obtaining Drinfeld's coproduct on the Yangian}

 \noindent For a Dynkin quiver $Q$, we have an associative algebra called the Yangian $Y_\hbar(\gk_Q)$ with a compatible \textit{Drinfeld} vertex coproduct $\Delta_{\textup{Dr}}(z)$ as discussed in section \ref{ssec:DrinfeldYangians}.
In this case, Yang and Zhao \cite{YZ} identified the Yangian with the double of the cohomological Hall algebra of the tripled quiver $\widetilde{Q}$, an isomorphism which restricts to isomorphisms on the positive and Borel parts:
\begin{equation}\label{fig:YangianCoHAIntro}
  \begin{tikzcd}[row sep = {40pt,between origins}, column sep = {20pt}]
    Y_{\hbar}(\gk_Q)^+ \ar[r,"\sim"] \ar[d,hook] &  (\Al^{T,\psi}_{\widetilde{Q},\widetilde{W}}) \ar[d,hook]\\ 
    Y_{\hbar}(\gk_Q)^{\ge 0} \ar[r,"\sim"] &  (\Al^{T,\psi,\ext}_{\widetilde{Q},\widetilde{W}}) 
  \end{tikzcd}
\end{equation}
The main result (Theorem \ref{thm:DrinfeldJoyce}) of this section is that the identification \eqref{fig:YangianCoHAIntro} exchanges the Drinfeld and extended Joyce--Liu vertex coproducts:
  $$\Delta_{\textup{Dr}}(z)\ = \  \Delta^{\ext}(z),$$
  where the left and right sides are defined in section \ref{ssec:DrinfeldCoproduct} and in Theorem \ref{quiver_potential_jl}. An analogue of this theorem for a localised type coproduct was proven in \cite{YZ}.

In order to prove this we have to compute the bosonised Joyce--Liu vertex coproduct explicitly on spherical elements for tripled quivers, which we do in Proposition \ref{prop:JoyceTripledComputation}. To this general case \cite{BD,schiffmann2024s} extended the identification \eqref{fig:YangianCoHAIntro} for arbitrary tripled quivers $\widetilde{Q}$ using the MO Yangian. Therefore, these computations also give a computation on spherical elements of Drinfeld type coproducts on positive halves of MO Yangians.

\subsection{CoHA of tripled  quiver with canonical cubic potential}

\subsubsection{} We now compute the Joyce--Liu coproduct on spherical elements in the case that the quiver is a triple $\widetilde{Q}$ with its canonical cubic potential $\widetilde{W}$, as defined in section \ref{ssec:Quivers}, and the rank one weight function acting on edges by 
$$\wtt(e) \ = \ 1, \hspace{5mm} \wtt(e^*) \ = \ -1, \hspace{5mm} \wtt(\omega_i) \ = \ -2$$
and $T$ be the associated rank one torus with $\Ht^{\sbt}(\BT)\simeq \Cb[\hbar/2]$. Note that this quiver is of course symmetric. Furthermore, we have
\begin{equation}
  \label{eqn:AssumptionCrit}
    \Ht^{\sbt}(\Ml_{Q,\delta_i}^T, \varphi_W) \ \simeq \ \Ht^{\sbt}_T(\BGm)[-2]
 \end{equation}
since $\widetilde{W}\vert_{\delta_i}=0$. 

We will consider the twisted CoHA $\mathcal{A}^{T_{\hbar},\chi}_{\widetilde{Q},\widetilde{W}}$ with respect to the twist $\psi(d_{1},d_{2}) = \chi_{\widetilde{Q}}(d_{1},d_{2})$, which satisfies $\psi(d_1,d_2)+\psi(d_2,d_1) = \tau_{\widetilde{Q}}(d_1,d_2)$; see section \ref{ssec:SignTwist} for more on sign twists. For tripled quivers the K\"unneth Assumption \eqref{eqn:KunnethAssumption} is known. Note that in general for quiver without loops, outside the ADE case, the CoHA of a tripled quiver is \textit{not} spherically generated. See \cite[Prop. 5.7]{davison2022affine}

\subsubsection{Defining spherical elements} \label{ssec:spherical_elements} Let
 $$R_{i}(z) \ \in \ \Ht^{\sbt}(\Ml_Q^T)_{\taut}[x]((z^{-1}))$$
 denote the image of the swapped $R$-matrix $\sigma(R_{\taut}(z))$ under the map
 \begin{align*}
    \Ht^{\sbt}(\Ml_Q^T)_{\taut} \otimes_T \Ht^{\sbt}(\Ml_{Q}^T)_{\taut}((z^{-1})) &\ \to \ \Ht^{\sbt}(\Ml_Q^T)_{\taut}  \otimes_T \Ht^{\sbt}(\Ml_{Q,\delta_i}^T)((z^{-1}))\\
    &\ \simeq \ \Ht^{\sbt}(\Ml_Q^T)_{\taut}[x]((z^{-1}))
 \end{align*}
where we identify $\Ht^{\sbt}(\Ml_{Q,\delta_i})\simeq \Cb[x]$ for $x=1\otimes x_{i,1}$ the first chern class of the line bundle $\El_{i,\delta_{i}}$.

\begin{lem} \label{defnofphi}
For any $i \in Q_0$, $R_i(z)$ is an expansion of a power series 
$$\ \Phi_i(x-z) \ \in \ \Ht^{\sbt}(\Ml^{T}_Q)_{\taut}\left[ \left[ (x-z)^{-1}\right]\right],$$
where writing $u=x-z$, the component of $\Phi_i(u)$ in the $d$th dimension vector\footnote{i.e., its image under the map $\Ht^{\sbt}(\Ml^{T}_Q)_{\taut}[[u^{-1}]]\to \Ht^{\sbt}(\Ml^T_{Q,d})[[u^{-1}]]$.} is
\begin{align}
& \Phi_{i,d}(u) =  \nonumber \\ 
& \label{eqn:PhiiFormula} \bigsqcap_{1 \leq n \leq d_{i}} \left( \frac{u-x_{i,n}+\hbar}{u-x_{i,n}-\hbar} \right) \cdot  \bigsqcap_{e \in Q_1}\frac{\bigsqcap_{\substack{s(e)=i \\ 1 \leq n \leq d_{t(e)}}}(u-x_{t(e),n}-\hbar/2) \bigsqcap_{\substack{t(e)=i \\ 1 \leq n \leq d_{s(e)}}}(u-x_{s(e),n}-\hbar/2)}{\bigsqcap_{\substack{s(e)=i \\ 1 \leq n \leq d_{t(e)}}}(u-x_{t(e),n}+\hbar/2) \bigsqcap_{\substack{t(e)=i \\ 1 \leq n \leq d_{s(e)}}}(u-x_{s(e),n}+\hbar/2)}   
\end{align}
as an element of $\Ht^{\sbt}(\Ml^T_{Q,d})[[u^{-1}]]$.
\end{lem}
\begin{proof}
This follows from Proposition \ref{prop:explicit_fullrmatrix}. Spelling this out, we compute $\sigma (R_{\delta_{i},d}(z))$ using the formulas in Proposition \ref{prop:euler_ext_to_euler_ben} and grouping the edges of the tripled quiver as $\widetilde{Q}_{1} = Q_{1} \sqcup Q^{\textup{op}}_{1} \sqcup Q_{0}$, applying $\act^{*}_{1}$ and setting $u = x - z$ with $x = 1 \otimes x_{i,1} $.
\end{proof}

We call the coefficients of the power series
$$\Phi_i(u) \ = \ 1 + \hbar \sum_{r \ge 0 }\Phi_{i,r} u^{-r-1} \ \in \ \Ht^{\sbt}(\Ml_Q^T)_{\taut}\left[ \left[ u^{-1}\right]\right]$$ 
the \textbf{tautological spherical elements}. We define the \textbf{spherical CoHA elements} as the coefficients $x^{(r)}_{i,1}$ :
$$x_{i}(u) \ = \ \sum_{r \ge 0} x^{(r)}_{i,1} u^{-r-1} \ \in \ \Al^T_{\widetilde{Q},\widetilde{W}}[[u^{-1}]],$$
 defined as the elements corresponding to $x^{r}\in \Cb[x]\simeq \Ht^{\sbt}(\BGm)$ under the identification \eqref{eqn:AssumptionCrit}. We stress the notational difference between $x_{i,1}^r$, an element in the cohomology ring, and $x_{i,1}^{(r)}$ an element of critical cohomology.

\subsubsection{Remarks about relations to work of Yang-Zhao}
 We note that the formula in equation \eqref{eqn:PhiiFormula} matches up with the formula in \cite[Section 1.3]{YZ}. Therefore, we believe that our extended vertex coproduct is closely related to the coproducts defined in \cite[Section 2]{YZ} and \cite[Section 2.1]{RSYZ}.

\subsubsection{} 
In Lemma \ref{lem:TautBialgebra}, we produced a coproduct on the tautological ring. It would also be convenient to write this coproduct for the elements $\Phi_{i,r}$. 

\begin{lem} \label{lemma:coprodonphi} 
    The coproduct and coaction
    $$ \oplus^* \ : \ \Ht^{\sbt}(\Ml^{T}_Q)_{\taut} \ \to \ \Ht^{\sbt}(\Ml^{T}_Q)_{\taut} \otimes_T \Ht^{\sbt}(\Ml^{T}_Q)_{\taut}, \hspace{10mm} \act_z^* \ : \ \Ht^{\sbt}(\Ml^{T}_Q)_{\taut} \ \to \ \Ht^{\sbt}(\Ml^{T}_Q)_{\taut}[z] $$
    defined in Lemma \ref{lem:TautBialgebra} send respectively
    \begin{equation}
      \label{eqn:OPlusPhi} 
      \Phi_i(u) \ \mapsto \  \Phi_i(u) \otimes \Phi_i(u), \hspace{15mm} \Phi_i(u) \ \mapsto \ \Phi_i(u-z)
    \end{equation}
    More precisely, we have under the direct sum pullback
\[ \Phi_{i,0} \ \mapsto \  1 \otimes \Phi_{i,0} + \Phi_{i,0} \otimes 1\] and 
\[ \Phi_{i,r} \ \mapsto \   1 \otimes \Phi_{i,r} + \Phi_{i,r} \otimes 1+ \hbar\left( \sum_{r_1+r_2=r-1} \Phi_{i,r_1} \otimes \Phi_{i,r_2} \right)  \]
when $r \geq 1$, and under the action map pullback
$$\Phi_{i,r} \ \mapsto \ \sum_{r_1+r_2=r} \binom{r}{r_1} \Phi_{i,r_1} z^{r_2}.$$
\end{lem}

\begin{proof}
    It follows from Lemma \ref{lem:SpectralHexagon}, in particular from equation \ref{eqn:directsumcoprodtaut} and \ref{actcoprodtaut} that
    $$(\oplus^*\, \otimes\, \id)\sigma(R_{\taut}(z)) \ = \ \sigma(R_{\taut,13}(z))\sigma(R_{\taut,23}(z)), \hspace{5mm}  \act_{w,1}^*\sigma(R_{\taut}(z))\ = \ \sigma(R_{\taut}(z+w)).$$
    Taking the image of this equation under the map $\Ht^{\sbt}(\Ml^{T}_Q)_{\taut}\to \Ht^{\sbt}(\Ml^{T}_{Q,\delta_i})$ on the third tensor factor gives \eqref{eqn:OPlusPhi} by Lemma \ref{defnofphi}.
\end{proof}

These elements moreover give a new generating set for the tautological ring:

\begin{prop} \label{prop:phiareincartan}
There is an algebra isomorphism
\begin{align}
    \Ht^{\sbt}(\BT)[\xi_{i,r}\ : \ i\in Q_0,\, r\ge 0] & \ \stackrel{\sim}{\to} \  \Ht^{\sbt}(\Ml_{Q}^T)_{\textup{taut}} \label{eqn:PhiGenerate} \\
    \xi_{i,r} & \ \mapsto \  \Phi_{i,r}.\nonumber
\end{align}
\end{prop}
\begin{proof}Since both vector spaces have the same graded dimensions, it is enough to show that the map is injective, i.e. that the $\Phi_{i,r}$ are algebraically independent.

  Since both sides of \eqref{eqn:PhiGenerate} are free $\Ht^{\sbt}(\BT)$-modules, it is enough to show that the $\Phi_{i,r}$ are algebraically independent after setting $\hbar=0$. Note that this follows from 
  \begin{equation}
     \label{equation:phicalculation}
     \Phi_{i,d}(u) \ = \ 1 + \hbar \sum_{r \ge 0} r!\left(2 \textup{ch}_{r}(\El_{i,d})- \sum_{e \, : \, i \to j}\textup{ch}_{r}(\El_{j,d})- \sum_{e\, : \, j\to i}\textup{ch}_{r}(\El_{j,d}) \right)u^{-r-1}+ \Ol(\hbar^2)
  \end{equation}
  since the bracketed expressions for fixed $r$ form a basis for $\Cb\{\ch_r(\El_i) \ : \ i\in Q_0\}$ by nondegeneracy of the Cartan matrix of $Q$, hence varying over all $i,r$ they inherit the property of being algebraically independent from the $\ch_r(\El_i)$, which are algebraically independent by the proof of Proposition \ref{prop:TautFree}.

  To prove \eqref{equation:phicalculation}, we use $f'(u)= f(u) (\log f(u))'$ to compute that the first order $\hbar$ term is 
  \begin{align*}
    \frac{\partial \Phi_{i,d}(u)}{\partial \hbar}\vert_{\hbar =0}& \ = \  \Phi_i(u) \frac{\partial \log(\Phi_{i,d}(u))}{\partial \hbar}\vert_{\hbar=0} \\
    & \ = \ \Phi_{i,d}(u) \left(\sum_{n=1}^{d_i} \frac{\partial_\hbar V_n(u)}{V_n(u)} + \sum_{e \, :\, i\to j}\sum_{n=1}^{d_j}\frac{\partial_\hbar E_n^e(u)}{E_n^e(u)} + \sum_{e \, :\, j\to i}\sum_{n=1}^{d_j}\frac{\partial_\hbar E_n^e(u)}{E_n^e(u)}\right)\vert_{\hbar=0} \\
    & \ = \  \sum_{n=1}^{d_i} \frac{\partial V_n(u)}{\partial \hbar} + \sum_{e \, :\, i\to j}\sum_{n=1}^{d_j}\frac{\partial E_n^e(u)}{\partial \hbar} + \sum_{e \, :\, j\to i}\sum_{n=1}^{d_j}\frac{\partial E_n^e(u)}{\partial \hbar}\vert_{\hbar=0} 
  \end{align*}
  where $V_n(u),E^e_n(u)$ are the factors of $\Phi_{i,d}(u)$ in the expression \eqref{eqn:PhiiFormula} labelled by a specific edge of $\widetilde{Q}$ and $n$th chern root, and in the third equality we used that they all are equal to $1$ when $\hbar=0$.

  The result then follows upon noting that if $f(u)=\frac{u-x+\hbar}{u-x-\hbar}$, we have
  \begin{align*}
    \frac{\partial f(u)}{\partial \hbar}\vert_{\hbar=0} & \ = \ 2 \frac{1}{u-x} \ = \ 2 \sum_{r \ge 0} x^r u^{-r-1} ,
  \end{align*}
  and that $r!\ch_r(\El_j)$ is a sum over $r$th powers of all the chern roots of $\El_j$.

\end{proof}

  \subsubsection{Remark}
From the point of view of the above Lemma another definition of the Cartan part could be given by taking coefficients of the $R$-matrix. This approach could potentially generalise to study Joyce coproducts outside quiver settings.

 In the remainder of this subsection we compute the extended Joyce--Liu vertex coproduct on the extended CoHA of the tripled quiver with canonical cubic potential. 

\begin{prop} \label{prop:JoyceTripledComputation} 
Let $\widetilde{Q}$ be the triple of any quiver $Q$. Then the extended Joyce-Liu vertex coproduct on $\mathcal{A}^{T, \textup{ext}}_{\widetilde{Q},\widetilde{W}}$ acts as
  \begin{align}
    \label{eqn:JLCoproducth}
    \Delta^{\textup{ext}}(\Phi_i(u))& \ =\  \Phi_i(u-z) \otimes \Phi_i(u)\\[5pt]
     \label{eqn:JLCoproductx+}
    \Delta^{\textup{ext}}(x^{(n)}_{i,1} ,z)  &\ = \   (\tau_z(x^{(n)}_{i,1} ) \otimes 1 + 1 \otimes x^{(n)}_{i,1} ) \\ & \hspace{15mm} \ + \ \hbar \sum_{N\ge 0}\left(\sum_{p=0}^N (-1)^{p+1} \binom{N}{p} \Phi_{i,p} \otimes x^{(n+N-p)}_{i,1} \right)z^{-N-1} \nonumber 
   \end{align}
\end{prop}

\begin{proof} We unwind the definition of the extended coproduct \eqref{eqn:CoHABosonisedCoproduct} from Theorem \ref{thm:CoHABosonisation}. We will use again and again the fact that the components of the $R$-matrix in components $0$- and $\delta_i$-components of $\Ml^T_Q$ are the cohomology classes  
  \begin{equation}
    \label{eqn:RTautRestriction}
    R(z)_{(d,0)} \ = \ 1 \otimes 1, \hspace{10mm} R(z)_{(0,d)} \ = \ 1 \otimes 1 , \hspace{10mm} \sigma R(z)_{(\delta_{i},d)}\ = \ \Phi_{i,d}(1 \otimes x_{i,1}-z),
  \end{equation}
  since the restrictions of $\Ext$ to the the $(0,d)$ and $(d,0)$ connected components are zero and the third equality is the definition of $\Phi_i(z)$.

Writing as a vector space $\Al^{T,\ext}_{\widetilde{Q},\widetilde{W}} = \Al^{T}_{\widetilde{Q},\widetilde{W}} \otimes_T \Ht^{\sbt}(\Ml_Q^T)_{\taut}$, equation \eqref{eqn:CoHABosonisedCoproduct} first gives that
\begin{align*}
  \Delta^{\ext}(1\otimes \Phi_i(u))& \ = \ (\sigma R(z))_{23}\cdot (1\otimes \act_{z}^*\Phi_i(u)_{(1)})\otimes (1\otimes \Phi_i(u)_{(2)})\\
  &  \ = \  (1\otimes \act_{z}^*\Phi_i(u))\otimes (1\otimes \Phi_i(u))\\
  &  \ = \  (1\otimes \Phi_i(u-z))\otimes (1\otimes \Phi_i(u))
\end{align*}
 since the $R$-matrix acts trivially by \eqref{eqn:RTautRestriction} and by the formula in Lemma \ref{lemma:coprodonphi} for the coproduct and coaction on $\Phi_i(u)$.

Before showing \eqref{eqn:JLCoproductx+} and finishing the proof, we compute the unextended Joyce--Liu coproduct:
\begin{align*}
  \Delta(x^{(n)}_{i,1},z) &\ = \ \Psi(\Ext,z)\cdot \textup{act}^*_{z,1}(\oplus^*x^{(n)}_{i,1}) \\
  &\ = \  \Psi(\Ext,z)\cdot \left(\textup{act}^*_{z}x^{(n)}_{i,1}\otimes 1  + 1 \otimes x^{(n)}_{i,1}\right) \\
  &\ = \ \tau_z x^{(n)}_{i,1}\otimes 1  + 1 \otimes x^{(n)}_{i,1} 
\end{align*}
where the second equality is true because the direct sum map in this case is 
$$\oplus_{\delta_i,0}\sqcup \oplus_{0,\delta_i} \ : \ (\Ml^T_{Q,\delta_i} \times_{\BT} \Ml^T_{Q,0})\  \sqcup \ (\Ml^T_{Q,0} \times_{\BT} \Ml^T_{Q,\delta_i}) \ \to\ \Ml^T_{Q,\delta_i}$$
is after identifying $\Ml_{Q,0}\simeq \pt$ the identity map on each component, and the third equality is true because $\Ext_{(0,d)}=\Ext_{(d,0)}$ is zero.

As a second preparation, we compute the restriction of the $R$-matrix using power series expansions: 
\begin{align*}
   \sigma R(z)_{(\delta_{i},-)} &\ = \ \Phi_i(1 \otimes x_{i,1}-z)\\
   & \ = \ 1 + \hbar \sum_{r \geq 0} \Phi_{i,r} (1 \otimes x_{i,1}-z)^{-r-1}\\
   & \ = \ 1 + \hbar \sum_{r \geq 0} \Phi_{i,r} \sum_{k\ge 0} (-1)^{r+1} \binom{r+k}{r} (1 \otimes x_{i,1})^{k} z^{-(k+r+1)}\\
   & \ = \  1 + \hbar \sum_{N \geq 0 } \left(\sum_{p \geq 0 } (-1)^{-(p+1)} \binom{N}{p} \Phi_{i,p}  \otimes x^{N-p}_{i,1} \right)z^{-(N+1)}  
\end{align*}
as elements of $\Ht^{\sbt}(\Ml^{T}_Q)_{\taut}\otimes_T \Ht^{\sbt}(\Ml^T_{Q,\delta_i})[[z^{-1}]]$, i.e. 
\begin{equation}
  \label{eqn:ComputationOfRTaut} 
 R^{(2)}(z) \otimes R^{(1)}(z)_{\delta_{i}} \ = \  1\otimes 1 + \hbar \left(\sum_{N \geq 0 } \sum_{p \geq 0 } (-1)^{-(p+1)} \binom{N}{p} \Phi_{i,p}  \otimes x^{N-p}_{i,1} \right)z^{-(N+1)}.
\end{equation}

It follows that 
\begin{align*}
  \Delta^{\ext}(x^{(n)}_{i,1}\otimes 1,z) & 
   \ = \ \left( (1\otimes 1)\otimes (1 \otimes 1) + \hbar \left(\sum_{N \geq 0 } \sum_{p \geq 0 } (-1)^{-(p+1)} \binom{N}{p} (1 \otimes \Phi_{i,p}) \otimes (x_{i,1}^{(N-p)} \otimes 1) \right)z^{-(N+1)}\right)\\
  & \hspace{45mm} \cdot \left((\tau_z x^{(n)}_{i,1}\otimes 1)\otimes (1\otimes  1)  + (1 \otimes 1) \otimes (x^{(n)}_{i,1}\otimes 1)\right) \\
  & \ = \  \left((\tau_z x^{(n)}_{i,1}\otimes 1)\otimes (1\otimes  1)  + (1 \otimes 1) \otimes (x^{(n)}_{i,1}\otimes 1)\right) \\
  & \hspace{15mm} \ + \ \hbar \left(\sum_{N \geq 0 } \sum_{p \geq 0 } (-1)^{-(p+1)} \binom{N}{p} (1 \otimes \Phi_{i,p} ) \otimes (x^{(n+N-p)}_{i,1}\otimes 1) \otimes  \right)z^{-(N+1)}\cdot
\end{align*}
where in the last line we used that by definition of the spherical elements that
$$x^k_{i,1}\cdot x^{(n)}_{i,1} \ = \ x^{(n+k)}_{i,1}.$$
This proves \eqref{eqn:JLCoproductx+} and finishes the proof.
\end{proof}

\subsection{Comparison to Drinfeld's coproduct}

\subsubsection{} 
We now compare our constructions in the ADE quiver case to ADE Yangians. We start with the generators and relations definition due to Drinfeld \cite{Dr2}.

\subsubsection{ADE Yangians} \label{sssec:Yangians}Let $Q$ be any ADE quiver. Let $c_{ij} = 2 \delta_{ij} -a_{ij}-a_{ji}$ is symmetrised Cartan matrix of the quiver $Q$, where $a_{ij}$ are the number of arrows from vertex $i$ to vertex $j$.

\begin{defn}
 The Yangian $Y_\hbar(\gk)$ is the $\Cb[\hbar]$ linear algebra generated by elements $\{x^{\pm}_{i,r},\xi_{i,r}\}_{i\in Q_0,
r\in\N}$, satisfying the following relations for every $i,j\in Q_0$ and $r,s\in \Nb$:
\begin{enumerate}[font=\upshape]
\item[(R1)]  $[\xi_{i,r}, \xi_{j,s}] = 0$

\item[(R2)]  $
[\xi_{i,0}, x_{j,s}^{\pm}] = \pm c_{ij} x_{j,s}^{\pm}$

\item[(R3)] $[\xi_{i,r+1}, x^{\pm}_{j,s}] - [\xi_{i,r},x^{\pm}_{j,s+1}] =
\pm\hbar\frac{c_{ij}}{2}(\xi_{i,r}x^{\pm}_{j,s} + x^{\pm}_{j,s}\xi_{i,r})$

\item[(R4)] 
$[x^{\pm}_{i,r+1}, x^{\pm}_{j,s}] - [x^{\pm}_{i,r},x^{\pm}_{j,s+1}]=
\pm\hbar\frac{c_{ij}}{2}(x^{\pm}_{i,r}x^{\pm}_{j,s} 
+ x^{\pm}_{j,s}x^{\pm}_{i,r})$

\item[(R5)] 
$[x^+_{i,r}, x^-_{j,s}] = \delta_{ij} \xi_{i,r+s}$

\item[(R6)] Assume $i\ne j$ and set $m = 1-c_{ij}$. For any
$r_1,\cdots, r_m\in \N$ and $s\in \N$ we have
\[\sum_{\pi\in\Sym_m}
\left[x^{\pm}_{i,r_{\pi(1)}},\left[x^{\pm}_{i,r_{\pi(2)}},\left[\cdots,
\left[x^{\pm}_{i,r_{\pi(m)}},x^{\pm}_{j,s}\right]\cdots\right]\right]\right.=0\]
\end{enumerate} 
\end{defn}
Analogously to subsection \ref{ssec:spherical_elements} we define the generating series
$$\xi_i(u) \  = \ 1 + \hbar \sum_{r \ge 0 }\xi_{i,r} u^{-r-1} \ \in \ Y_\hbar(\gk)[[u^{-1}]], \hspace{15mm}  x^{\pm}_{i}(u) \  = \ \sum_{r \ge 0} x^{\pm}_{i,r} u^{-r-1} \ \in \ Y_\hbar(\gk)[[u^{-1}]].$$

\subsubsection{Comparison with the CoHA as an algebra} 
We will consider the Borel part of the Yangian, Let $Y_\hbar(\gk)^{\geq 0} \subseteq Y_\hbar(\gk)$ be the subalgebra generated by $\{ x^{+}_{i,1},\xi_{i,r}\}_{i\in Q_0,
r\in\N}$. Then we have the following theorem, which is essentially \cite[Thm. D]{yang2018cohomological} in our formulation.

\begin{prop} \label{prop: isomorphism of coha and yangian}
The morphism  
\begin{align*}
    f : Y_\hbar(\gk)^{\geq 0} &\ \to\ \mathcal{A}^{T,\chi, \textup{ext}}_{\widetilde{Q},\widetilde{W}} \\
    x^{+}_{i,r} &\ \mapsto\ x^{(r)}_{i,1} \\
    \xi_{i,r} &\ \mapsto\  \Phi_{i,r}
\end{align*}
is an isomorphism of algebras. 
\end{prop}
\begin{proof}
It is already proven in \cite{YZ} that this morphism when restricted to $Y_{\hbar}(\mathfrak{g}_Q)^{>0}$ is an isomorphism and we have proved in Proposition \ref{prop:phiareincartan} that the morphism restricted to Cartan $Y_\hbar(\gk_Q)^0$ is an isomorphism. Thus it suffices to show that this morphism $f$ is well defined: we have to check relations (R2) and (R3), since all other relations are implied by $f$ being well-defined on $>0$ and $0$ parts. We will write $1 \otimes h$ when view an element $h$ of the tautological ring as an element of the extended algebra, and similarly $b \otimes1$ for $b$ in the CoHA. 

To show (R2), we first note 
\[ f([\xi_{i,0},x^{+}_{j,s}]) = [1 \otimes \Phi_{i,0}, x^{(s)}_{j,1} \otimes 1] .\]
Thus by definition of the product $\star$ of \eqref{eqn:BosonisationProduct}, we have
\begin{align*}
    (1 \otimes \Phi_{i,0}) \star (x^{(s)}_{j,1} \otimes 1) &\ =\ (\Phi_{i,0})_{(1)} \cdot x^{(s)}_{j,1} \otimes (\Phi_{i,0})_{(2)}\\
    & \ =\ x^{(s)}_{i,1} \otimes \Phi_{i,0}+ \Phi_{i,0}\cdot x^{(s)}_{j,1} \otimes 1 \\
    & \ =  \ x^{(s)}_{i,1} \otimes \Phi_{i,0}+ c_{ij} (x^{(s)}_{j,1} \otimes 1)
\end{align*}
where we computed the coproduct of $\Phi_{i,0}$ in Lemma \ref{lemma:coprodonphi}. Furthermore, by definition $\Phi_{i,0}\cdot x^{(s)}_{j,1} = \Phi_{i,0,\delta_j} \cup x^{(s)}_{j,1}$ and so by equation \eqref{equation:phicalculation}, it follows that in fact $\Phi_{i,0,\delta_j} = c_{ij}$. Similarly, by definition of $\star$ we have
\[ (x^{(s)}_{j,1} \otimes 1)\star (1 \otimes \Phi_{i,0})\ =\ x^{(s)}_{j,1} \otimes \Phi_{i,0}\] 
and therefore
 \[  [1 \otimes \Phi_{i,0}, x^{(s)}_{j,1} \otimes 1]\ =\ c_{ij} x^{(s)}_{j,1} \otimes 1\ = \ f(c_{ij}x^{+}_{i,s})\] 
 and thus the relation (R2) is satisfied.

We now check (R3). We have 
\[ [1 \otimes \Phi_{i,r+1}, x^{(s)}_{j,1} \otimes1]\ =\  \Phi_{i,r+1} \cdot x^{(s)}_{j,1} \otimes 1 + \hbar \left(\sum_{k_1+k_2=r} \Phi_{i,k_1} \cdot x^{(s)}_{j,1} \otimes \Phi_{i,k_2}\right) \] using same strategy as above, since
\[ (1 \otimes \Phi_{i,r+1}) \star ( x^{(s)}_{j,1} \otimes1) = x^{(s)}_{j,1} \otimes \Phi_{i,r+1} + \Phi_{i,r+1} \cdot x^{(s)}_{j,1} \otimes 1 + \hbar \left(\sum_{k_1+k_2=r} \Phi_{i,k_1} \cdot x^{(s)}_{j,1} \otimes \Phi_{i,k_2}\right) ,\]
therefore combining this with $
 ( x^{(s)}_{j,1} \otimes1) * (1 \otimes \Phi_{i,r+1}) = x^{(s)}_{j,1} \otimes \Phi_{i,r+1}$ gives us that 
\[ [1 \otimes \Phi_{i,r}, x^{(s+1)}_{j,1} \otimes1] \ =\  \Phi_{i,r} \cdot x^{(s+1)}_{j,1} \otimes 1 + \hbar \left(\sum_{k_1+k_2=r-1} \Phi_{i,k_1} \cdot x^{(s+1)}_{j,1} \otimes \Phi_{i,k_2}\right) \] 
and so it remains to compute the coefficients of these series. Note that by definition $\Phi_{i,r} \cdot x^{(s)}_{j,1} = \Phi_{i,r,\delta_j} \cdot x^{(s)}_{j,1}$. If there is no arrow from $i$ to $j$ in $\widetilde{Q}$, $\Phi_{i,r,\delta_j}=0$ and hence the commutator satisfies:
\[ [1 \otimes \Phi_{i,r+1}, x^{(s)}_{j,1} \otimes1] - [1 \otimes \Phi_{i,r}, x^{(s+1)}_{j,1} \otimes1]\ =\ 0 \] 
which is exactly the relation (R3) is satisfied. If not, then
\[ \Phi_{i,\delta_j}(u) = \delta_{i,j} \frac{u-x_{j,1}+\hbar}{u-x_{j,1}-\hbar} + \delta_{i \to j} \frac{u-x_{j,1}-\hbar/2}{u-x_{j,1}+\hbar/2} + \delta_{j \to i } \frac{u-x_{j,1}-\hbar/2}{u-x_{j,1}+\hbar/2} \]
where $\delta_{i \to j}$ is 1 if there is a arrow from $i$ to $j$ in the quiver $Q$ and $0$ otherwise. This gives that \[ \Phi_{i,r,\delta_j} =  2 \delta_{ij} (x_{j,1}+\hbar)^r - \delta_{i \to j} (x_{j,1}-\hbar/2)^r - \delta_{j \to i}(x_{j,1}-\hbar/2)^r \]
and thus 
\begin{align*} \Phi_{i,r+1} \cdot x^{(s)}_{j,1} &\ =\  2 \delta_{ij} (x_{j,1}+\hbar)^{r+1} \cdot x^{(s)}_{j,1}-  \delta_{i \to j} (x_{j,1}-\hbar/2)^{r+1}\cdot x^{(s)}_{j,1} - \delta_{j \to i}(x_{j,1}-\hbar/2)^{r+1} \cdot x^{(s)}_{j,1} \\ &\ =\ \Phi_{i,r}\cdot x^{(s+1)}_{j,1}-\hbar/2(2 \delta_{ij} - \delta_{i \to j} - \delta_{j \to i})\Phi_{i,r} \cdot x^{(s)}_{j,1} = \Phi_{i,r}\cdot x^{(s+1)}_{j,1}-\hbar/2(c_{ij})\Phi_{i,r} \cdot x^{(s)}_{j,1}.
\end{align*}
Thus it follows that 
\begin{align*}& 
[1 \otimes \Phi_{i,r+1}, x^{(s)}_{j,1} \otimes1] - [1 \otimes \Phi_{i,r}, x^{(s+1)}_{j,1} \otimes1]  \\ 
& \hspace{15mm} \ = \ c_{ij} \hbar/2  \left(1 \otimes \Phi_{i,r}) * ( x^{(s)}_{j,1} \otimes1) + ( x^{(s)}_{j,1} \otimes1) * (1 \otimes \Phi_{i,r}) \right)
\end{align*}
and so (R3) follows.
\end{proof}
The cohomological Hall algebra on the right has a coproduct, as we defined in previous sections. The Yangian also a has coproduct, due to Drinfeld which we now recall. 
\subsubsection{Drinfeld's coproduct} \label{ssec:DrinfeldCoproduct} 
We now define the vertex coproduct due to Drinfeld \cite{Dr2} and Gautam--Toledano-Laredo \cite{GTW}, who understood how to express it as a Laurent power series in $z^{-1}$. First let $\tau_z: Y_\hbar(\gk) \to Y_\hbar(\gk)[z]$ be the map defined by its action on generating series:
$$x^{\pm}_{i}(u) \mapsto x_i^{\pm}(u-z) \quad \text{and} \quad \xi_i(u) \mapsto \xi_i(u-z),$$
then by \cite[Prop. 3.3, Thm. 3.4]{GTW}: the following so-called \textit{deformed Drinfeld coproduct} defines an algebra morphism
 \begin{equation} \label{eqn:drinfeld_coproduct_expl}
   \Delta_{\textup{Dr}}(z) \ : \  Y_{\hbar}^{\geq 0}(\mathfrak{g}_Q) \to Y_{\hbar}^{\geq 0}(\mathfrak{g}_{Q}) \otimes_\hbar Y_{\hbar}^{\geq 0}(\mathfrak{g}_Q) ((z^{-1})).
 \end{equation}
if we set
\begin{align*}
   \Delta_{\textup{Dr}}(\xi_i(u),z)&\ =\  \xi_i(u-z) \otimes \xi_i(z) \\
   \Delta_{\textup{Dr}}(x_{i,n}^+,z) &\ = \ \tau_z(x_{i,n}^+) \otimes 1 \ +\ 1 \otimes x_{i,n}^+\ +\ \hbar \sum_{N\ge 0}\left(\sum_{p=0}^N (-1)^{p+1} \binom{N}{p} \xi_{i,p} \otimes x_{i,n+N-p}\right)z^{-N-1}.
\end{align*}

 Our main Theorem in this section is then:

\begin{theorem} \label{thm:DrinfeldJoyce} [Drinfeld and Joyce--Liu coproducts agree]
Let Q be any ADE quiver. Then the map $f$ of Proposition \ref{prop: isomorphism of coha and yangian} \begin{align*}
       f : Y_\hbar(\gk)^{\geq 0} & \ \stackrel{\sim}{\to} \  \mathcal{A}^{T,\chi, \textup{ext}}_{\widetilde{Q},\widetilde{W}}
   \end{align*}  
   is an isomorphism of vertex bialgebras, intertwining the meromorphic Drinfeld coproduct on $Y_\hbar(\gk)^{\geq 0}$ with the extended Joyce-Liu coproduct on $\mathcal{A}^{T,\chi, \textup{ext}}_{\widetilde{Q},\widetilde{W}}$. More precisely, $f$ is a map of algebras and the following diagram commutes
     \begin{equation} \label{fig:MainThm}
\begin{tikzcd}
	{Y_{\hbar}(\mathfrak{g}_{Q})^{\geq 0}} & {Y_{\hbar}(\mathfrak{g}_{Q})^{\geq 0} \otimes_\hbar Y_{\hbar}(\mathfrak{g}_{Q})^{\geq 0} ((z^{-1}))} \\
	{\Al^{T,\chi,\textup{ext}}_{\widetilde{Q},\widetilde{W}}} & {\Al^{T,\chi,\textup{ext}}_{\widetilde{Q},\widetilde{W}} \otimes_T \Al^{T,\chi,\textup{ext}}_{\widetilde{Q},\widetilde{W}} ((z^{-1}))}
	\arrow["{\Delta_{\textup{Dr}}(z)}", from=1-1, to=1-2]
	\arrow["f"',"\wr", from=1-1, to=2-1]
	\arrow["{f \otimes f}","\wr"', from=1-2, to=2-2]
	\arrow["{\Delta^{\textup{ext}}(z)}", from=2-1, to=2-2]
\end{tikzcd}
  \end{equation}
\end{theorem}
\begin{proof}
  The map $f$ is an algebra isomorphism by Proposition \ref{prop: isomorphism of coha and yangian}. By Proposition \ref{prop:JoyceTripledComputation}, the formulas for the Drinfeld and extended Joyce--Liu coproducts match up via the morphism $f$.
\end{proof}

\newpage
\section{Vertex bialgebras for symmetric quivers with no potential} \label{quiver_no_potential_section}

\noindent
In this section we will consider our constructions in the special case when $W= 0$ and the quiver is symmetric. In this case vertex algebra structures on the Cohomological Hall algebra were already considered in \cite{latyntsev2021cohomological} and \cite{dotsmozgva}. Our goal is to give another characterization of the vertex algebra structure in this case and compare our constructions to \cite{dotsmozgva}. Along the way, we give another proof of cohomological integrality for symmetric quivers with no potential. \par

\subsubsection{Warning about sign twists}
We make a break from the conventions in the rest of the paper here by working in the $\tau$-twisted symmetric monoidal category $\textup{Vect}_{\mathbf{Z} \times \mathbf{N}^{Q_{0}}}$ and working with the untwisted CoHA $\mathcal{A}_{Q}$ of a symmetric quiver. Then by Theorem \ref{thm:CoHABialgebra} the Joyce coproduct is colocal up to the swap morphism \[\beta^{\tau}: v \otimes w \mapsto (-1)^{|v||w|+\tau} w \otimes v. \] for elements $v,w \in \mathcal{A}_{Q}$. Suppose $v \in\mathcal{A}_{Q,d}$ and $w \in \mathcal{A}_{Q,e}$ for some dimension vectors $d,e$. Then since $\mathcal{A}_{Q,d}$ is defined to be  $\Ht^*(\Ml_{Q,d},\mathbf{Q}[-\chi(d,d)])$, if $v \in \mathcal{A}_{Q,d}$ then $(-1)^{|v|}= (-1)^{\chi(d,d)}$ as the cohomology $\Ht^*(\Ml_{Q,d})$ is concentrated in even cohomological degrees. Since $\tau(d,e) = \chi(d,d)\chi(e,e)+\chi(d,e)$, $(-1)^{\tau+ |v||w|} = (-1)^{\chi(d,e)}$. Thus in this case, $\beta^{\tau}$ is same as \[ \beta^{\chi}_0: v \otimes w \mapsto (-1)^{\chi} w \otimes v.\] Therefore, we obtain an honest colocal vertex coalgebra in the symmetric monoidal structure arising from $\beta^{\chi}_0$. This is the same as the convention used in \cite{dotsmozgva}(See Remark 4.2).

 In this section we will consider the $\mathbf{N}^{{Q}_{0}} \times \mathbf{Z}$-graded dual $\mathcal{A}^{*}_{Q}$ of $\mathcal{A}_{Q}$ and use the following theorem
 \begin{theorem}\cite[Thm. 6.1]{Hu} \label{hubbard_dualizing}
There is a functor, which given a colocal coassociative vertex coalgebra produces a vertex algebra.
\begin{align}
    F : \textup{CoVertex}_{\mathbf{N}^{Q_{0}} \times \mathbf{Z}} & \to \textup{Vertex}_{\mathbf{N}^{Q_{0}} \times \mathbf{Z}} \\
    V = \bigoplus_{(d,k) \in \mathbf{N}^{Q_{0}} \times \mathbf{Z}} V_{d,k} & \mapsto F(V) = \bigoplus_{(d,k) \in \mathbf{N}^{Q_{0}} \times \mathbf{Z}} V^{\vee}_{d,-k}.
\end{align}
\end{theorem}
 Therefore, $\mathcal{A}^{*}_{Q,0}$ now has a vertex algebra and a cocommutative coalgebra structure. By the $0$-potential version of Theorem \ref{thm:CoHABialgebra} $\mathcal{A}^{*}_{Q}$ is a cocommutative coalgebra coming from the dual of the commutative CoHA product and is a vertex algebra coming from the dual of the vertex coproduct. 
\subsection{Structural results for cocommutative vertex bialgebras}
Cocommutative vertex bialgebras have a similar structure theory to cocommutative bialgebras. We start by defining the analogue of a universal enveloping algebra.
We define the \emph{Universal chiral envelope} in the following way.
Let $L $ be a vertex Lie algebra as in \cite[Definition 3.6]{han2021cocommva}. Then define $\mathfrak{g}_{L}$ to be the Lie algebra
\begin{equation}
    L \otimes \mathbf{C}[t,t^{-1}]/ \im \partial
\end{equation}
where $\partial$ is the translation operator 
\begin{equation}
    D \otimes \id + \id \otimes \frac{d}{dt}
\end{equation}
on $L \otimes \mathbf{C}[t,t^{-1}]$ induced by the translation operator $D$ on $L$.
We have a decomposition
\begin{equation}
    \mathfrak{g}_{L} = \mathfrak{g}^{+}_{L} \oplus \mathfrak{g}^{-}_{L}
\end{equation}
where $\mathfrak{g}^{-}_{L} = \spn \{ a \otimes t^{i} \mid i < 0 , \, a \in L \}$ and  $\mathfrak{g}^{+}_{L} = \spn \{ a \otimes t^{i} \mid i \geq 0, \, a \in L \}$. \par
Take the trivial $\mathfrak{g}^{+}_{L}$ module structure on $\mathbf{C}$ and define
\begin{equation}
    \UU^{\ch}(L) = \UU \mathfrak{g}_{L} \otimes_{\UU \mathfrak{g}^{+}_{L}} \otimes \mathbf{C}
\end{equation}
then $\UU^{\ch}(L)$ is a vertex algebra and we have an isomorphism $\UU^{\ch}(L) \simeq \UU \mathfrak{g}^{-}_{L}$ of $\UU \mathfrak{g}^{-}_{L}$- modules. 
We now summarise some structural results about cocommutative vertex bialgebras from \cite{han2021cocommva}. 
\begin{theorem}[ Structure of cocommutative vertex bialgebras ]  \label{han_hopfva}
\,
  \begin{enumerate}
    \item Let $V$ be a vertex algebra with a compatible cocommutative coproduct $\Delta$. Denote by $P(V)$ the set of primitive elements with respect to $\Delta$, then $P(V)$ has a vertex Lie algebra structure induced from $V$ and $P(V)$ is a vertex Lie subalgebra of $V$. \cite[Prop. 4.8]{han2021cocommva}
    \item (Milnor-Moore) Let $V$ be a connected cocommutative vertex bialgebra, then we have an isomorphism of vertex algebras \cite[Thm. 4.13]{han2021cocommva}
    \begin{equation}
        V \simeq \UU^{\ch}(P(V)).
    \end{equation}
    and an isomorphism of coalgebras $\UU^{\ch}(P(V)) \simeq \UU \mathfrak{g}^{-}_{P(V)}$.
    \item (PBW theorem) We have the isomorphism of vector spaces
    \begin{equation}
        \UU^{\ch}(P(V)) \simeq \Sym(\mathfrak{g}^{-}_{P(V)}).
    \end{equation}
\end{enumerate}
\end{theorem}
\subsection{Cohomological integrality via vertex algebras}
Note that since $\mathcal{A}_{Q}$ is connected, we can now use the Theorem recalled above to obtain cohomological integrality and also compute the algebra structure of $\mathcal{A}_{Q}$. \par We start by  recalling two different constructions of symmetric algebras.
Let $V$ be a graded vector space with finite dimensional pieces. There is a natural action of $S_{n}$ on each $V^{ \otimes n}$, this then induces an action of all the symmetric groups on the tensor algebra $T(V)$. Then define 
\begin{enumerate}
    \item the \emph{symmetric algebra} $\Sym(V)$ as the bialgebra with product induced by the quotient map $T(V) \to \Sym(V)$ and coproduct $\Delta$ induced by  $\Delta(x) = x \otimes 1 + 1 \otimes x$.
    \item the \emph{invariant algebra} $\Gamma(V) \xhookrightarrow{} T(V)$ as the subspace of invariant tensors. $\Gamma(V)$ is an algebra with the shuffle algebra product
    \begin{equation}
        (v_{1} \otimes \cdots \otimes v_{n}) * (w_{1} \otimes \cdots \otimes w_{m}) \mapsto \frac{1}{n!m!}\sum_{\sigma \in S_{n+m}} \sigma( v_{1} \otimes \cdots \otimes v_{n} \otimes  w_{1} \otimes \cdots \otimes w_{m}).
    \end{equation}
\end{enumerate}
\begin{lem} \label{symcoalg}
Let $V$ be a graded vector space with finite dimensional graded pieces. Then we have the following isomorphisms
\begin{enumerate}
    \item We have an isomorphism of algebras $(\Sym V)^{\vee} \simeq \Gamma(V^{*})$, where the algebra structure on $(\Sym V)^{\vee}$ is induced by the coproduct $\Delta$ on $\Sym (V)$
    \item We have an isomorphism of algebras $\Sym(V) \simeq \Gamma(V)$.
    \end{enumerate}
\end{lem}
\begin{proof}
    Part $(1)$ is \cite[Exercise 1.8.7]{loday2012algebraic}.
    For part $(2)$ we can consider the map
    \begin{align}
        \textup{Sym}(V) & \to \Gamma(V)  \nonumber \\
        v_{1} \otimes \cdots \otimes v_{n} & \mapsto \sum_{\sigma \in S_{n}} \sigma(v_{1} \otimes \cdots \otimes v_{n})
    \end{align}
    which is an isomorphism. This is clearly a morphism of algebras since 
    \begin{align*}
    &\frac{1}{n!m!}\sum_{\sigma \in S_{n+m}} \sigma \left( \sum_{\sigma^{\prime} \in S_{n}, \sigma^{\prime \prime} \in S_{m}} \sigma^{\prime}(v_1\otimes \cdots \otimes v_n) \sigma^{\prime \prime}(w_1 \otimes \cdots w_m)  \right) \\ &= \sum_{\sigma \in S_{n+m}} \sigma(v_1\otimes \cdots \otimes v_n \otimes w_1 \cdots \otimes w_m)
    \end{align*} 
    because we are repeating $n!m!$ permutations in the sum on LHS. 
    
    Let $f \in \Gamma^n(V)$, where $\Gamma^n(V) \subset \Gamma(V)$ is subspace of invariant tensors in $V^{\otimes n}$. We define its inverse as $1/n! \overline{f}$ where $\overline{f}$ is the image of the composition $\Gamma(V) \subset T(V) \to \Sym(V)$. We linearly extend this to morphism to define the inverse $ \Gamma(V) \to \Sym(V)$. Clearly the composition is identity since $\sigma(v_1 \otimes \cdots \otimes v_n) = v_1 \otimes \cdots v_n$ in the quotient $\Sym(V)$. 
\end{proof}
The dual of the Cohomological Hall algebra $\mathcal{A}^{*}_{Q}$ is a connected vertex bialgebra. Then we have the Theorem
\begin{theorem} \label{chiral_envelope}
    We have the following
    \begin{enumerate}
        \item  We have an isomorphism of vertex  bialgebras
    \begin{equation}
        \mathcal{A}^{*}_{Q} \simeq \UU^{\ch}(\cprim).
    \end{equation}
    Here, the CoHA coproduct is identified with the coproduct on $\mathcal{A}^{*}_{Q} \simeq \UU^{\ch}(\cprim)$ and the Joyce vertex algebra with the Universal chiral envelope.
    \item $\mathcal{A}_{Q}$ is isomorphic to a symmetric algebra of a graded vector space.
    \end{enumerate}
\end{theorem}
\begin{proof}
    Part $(1)$ follows immediately from the compatibility of the CoHA structure and the Joyce vertex algebra. By Theorem \ref{han_hopfva} for connected vertex bialgebras we then immediately get the result. \par
    Using part $(2)$ of Theorem \ref{han_hopfva} we can conclude that we have an isomorphism of coalgebras 
    
    $$\UU^{\ch}(\cprim) \simeq \UU \mathfrak{g}^{-}_{\cprim} \simeq \Sym(\mathfrak{g}^{-}_{\cprim}),$$ using \cite[Thm. 1.3.6]{loday2012algebraic} that the PBW isomorphism $\UU \mathfrak{g} \simeq \Sym(\mathfrak{g})$ is an isomorphism of coalgebras for any $g$. Furthermore, in \cite[Thm. 5.7, Remark 5.9]{dotsmozgva} it is shown that 
    \begin{equation}
        \mathfrak{g}^{-}_{\cprim} \simeq W[u]
    \end{equation}
    for some graded vector space $W$. Then using Lemma \ref{symcoalg} we get an isomorphism of algebras
    \begin{align*}
         \mathcal{A}_{Q} \simeq (\UU^{\ch}(\cprim))^{\vee} & \simeq (\Sym(W[u]))^{\vee} \\
         & \simeq \Gamma(W^{*}[u]) \qquad \text{(by part 1 of Lemma } \ref{symcoalg} )\\
         & \simeq \Sym(W^{*}[u]) \qquad \text{(by part 2 of Lemma } \ref{symcoalg})
    \end{align*}
This proves that $\mathcal{A}_{Q}$ is a symmetric algebra.
\end{proof}
We note that the above Theorem gives a proof of the integrality conjecture for DT invariants of quivers without potential as well as a proof of Efimov's theorem that cohomological Hall algebras of symmetric quivers without potential are symmetric algebras. 
\subsection{Comparison to Dotsenko-Mozgovoy}
In \cite{dotsmozgva}, the authors also consider vertex algebra structures on $\mathcal{A}^{*}_{Q}$. In particular, they define a vertex Lie algebra $P_{DM}$ \cite[Section 4.4.4]{dotsmozgva}, in this case the universal chiral envelope $\UU^{\ch}(P_{DM})$ is a certain free vertex algebra. In \cite[Thm. 5.6]{dotsmozgva} the authors give an isomorphism of coalgebras 
\begin{equation}
    \UU^{\ch}(P_{DM}) \simeq \mathcal{A}^{*}_{Q}.
\end{equation}
This isomorphism then \textit{endows} $\mathcal{A}^{*}_{Q}$ with a structure of vertex algebra via $\UU^{\ch}(P_{DM})$. We now wish to compare the two vertex algebra structures. The key point is that both vertex algebras embed into a lattice vertex algebra given by the Joyce vertex algebra of the homology of the moduli stack of objects of the derived category of quiver representations. By showing that the two subalgebras are the same we get the following theorem
\begin{theorem} \label{DM_comp}
We have an isomorphism of vertex bialgebras
\begin{equation}
    \mathcal{A}^{*}_{Q} \simeq \UU^{\ch}(P_{DM})
\end{equation}
where $\mathcal{A}^{*}_{Q}$ is equiped with the dual CoHA cocommutative coproduct and the Joyce vertex algebra structure. This isomorphism then induces an isomorphism of primitive vertex Lie algebras
    \begin{equation}
        \cprim \xrightarrow{\simeq} P_{DM}
    \end{equation}
\end{theorem}
Using the results of Dotsenko-Mozgovoy we can view the above theorem as a computation in terms of generators and relations of the Joyce vertex bialgebra  for any symmetric quiver $Q$.
\begin{lem}
    Let $\mathcal{C}$ be an abelian category. Consider the inclusion $\Ml_{\mathcal{C}} \to \Ml_{\Dt^{b}\mathcal{C}}$ from the moduli of objects of the abelian category to the derived category. Then we have an injection of vertex algebras
    \begin{equation}
        \Ht_{\sbt}(\Ml_{\mathcal{C}}) \to \Ht_{\sbt}(\Ml_{\Dt^{b}\mathcal{C}}).
    \end{equation}
\end{lem}
\begin{proof}
    Firstly, the map is injective since $\Ml_{\mathcal{C}} \subseteq \Ml_{\Dt^{b}\mathcal{C}}$ is an inclusion of a component. The compatibility of vertex algebra structures follows since we have commutative diagrams
    \begin{equation}
\begin{tikzcd}
	{\Ml_{\mathcal{C}} \times \Ml_{\mathcal{C}}} & {\Ml_{\Dt^{b}\mathcal{C}} \times \Ml_{\Dt^{b}\mathcal{C}}} & {\B \mathbf{G}_{m} \times \Ml_{\mathcal{C}}} & {\B \mathbf{G}_{m} \times \Ml_{\Dt^{b}\mathcal{C}}} \\
	{\Ml_{\mathcal{C}}} & {\Ml_{\Dt^{b}\mathcal{C}}} & {\Ml_{\mathcal{C}}} & {M_{\Dt^{b}\mathcal{C}}}
	\arrow[hook, from=1-1, to=1-2]
	\arrow["\oplus"', from=1-1, to=2-1]
	\arrow["\oplus"', from=1-2, to=2-2]
	\arrow[hook, from=1-3, to=1-4]
	\arrow["\act"', from=1-3, to=2-3]
	\arrow["\act", from=1-4, to=2-4]
	\arrow[hook, from=2-1, to=2-2]
	\arrow[hook, from=2-3, to=2-4]
\end{tikzcd}
    \end{equation}
    and the fact that the $\textup{Ext}$ complex restricts to the open component.
\end{proof}
We now have the following theorem, which proves that the relation between Joyce vertex algebra for the derived category and Lattice vertex algebras. We remark that originally the theorem works for arbitary quivers but with symmetrised Euler form, however in our case the quiver is symmetric and so there is no need to symmetrise the Euler form and the proof of the original statement goes through word for word. 

\begin{theorem}\cite[Thm. 5.19]{Jo} \label{joyce_va_sub_gen}
There is an isomorphism of graded vertex algebras
\begin{equation}
    \Ht_{\sbt}(\Ml_{\Dt^{b} \Rep Q}) \simeq L_{\mathbf{Q}^{0}}
\end{equation}
for the lattice $\mathbf{Z}^{Q_{0}}$ and Euler form $\chi_{Q}$. \par
\end{theorem}
\begin{proof}[Proof of Theorem \ref{DM_comp}]
By \cite[Section 4.5.3]{dotsmozgva} we can view $\UU^{\ch} (P_{DM})$ as a vertex subalgebra of $L_{Q^{0}} \simeq \Ht_{\sbt}(\Ml_{\Dt^{b} \Rep_{Q}})$ generated by the elements $e^{i}$. Note that by Theorem \ref{joyce_va_sub_gen} the subvertex algebra
\begin{equation}
    \Ht_{\sbt}(\Ml_{Q}) \xhookrightarrow{} L_{Q^{0}}
\end{equation}
 contains the elements $e^{i}$. This therefore defines an injective map of graded vertex algebras
\begin{equation}
    \UU^{\ch}( P_{DM}) \xrightarrow{} \Ht_{\sbt}(\Ml_{Q})
\end{equation}
In \cite[Prop. 5.3 ]{dotsmozgva} it is proven that we have an isomorphism as graded vector spaces
\begin{equation}\label{mozg_53}
    \UU^{\ch}( P_{DM}) \simeq \Ht_{\sbt}(\Ml_{Q})
\end{equation}

Note that both of our spaces have finite dimensional graded pieces. The statement of \cite[Prop. 5.3]{dotsmozgva} is the dual of equation \eqref{mozg_53} but because the graded pieces are finite dimensional the two statements are equivalent. Equation \eqref{mozg_53} implies that both $\UU^{\ch}(P_{DM})$ and $\Ht_{\sbt}(\Ml_{Q})$ have the same graded dimensions. Since the map  $\UU^{\ch}( P_{DM}) \xrightarrow{} \Ht_{\sbt}(\Ml_{Q})$ is injective we can therefore conclude it is an isomorphism.
\end{proof}

\newpage 

\appendix

\section{Cohomology of stacks and characteristic classes of perfect complexes} \label{sec:appendix}

\noindent Throughout this paper, we will work with the derived category of constructible sheaves $\Dt(X)$ on a stack $X$ with the usual $6$-functors and all functors are implicitly derived. We recall that we can define the cohomology of a stack with coefficients in a sheaf $\mathcal{F}$
\begin{equation}
    \Ht^{\sbt}(X, \mathcal{F})\ = \ (X \to \pt)_{*} \mathcal{F}
\end{equation}
We recover the usual cohomology of $X$ by taking $\mathcal{F} = \Qb_{X}$, the constant sheaf.

\subsection{Cohomology} \label{sec:app_cohomology}

\subsubsection{K\"unneth formula} \label{ssec:Kunneth} The K\"unneth map 
\begin{equation}
  \label{eqn:Kunneth}
   \Ht^{\sbt}(X_1,\Fl_1) \otimes \Ht^{\sbt}(X_2,\Fl_2) \ \to \ \Ht^{\sbt}(X_1\times X_2,\Fl_1 \boxtimes \Fl_2)
\end{equation}
is an isomorphism whenever $X_i$ are Artin stacks locally of finite type and $\Ht^{\sbt}(X_i,\Fl_i)$ are graded finite dimensional. Indeed, for such $X_i$ we have by \cite[Thm. 0.1.1]{LZ} a K\"unneth isomorphism for homology groups 
$$ \Ht_{\sbt}(X_1,\Fl_1) \otimes \Ht_{\sbt}(X_2,\Fl_2)\ \simeq \ \Ht_{\sbt}(X_1\times X_2,\Fl_1 \boxtimes \Fl_2)$$ 
and taking Verdier duals gives \eqref{eqn:Kunneth} if the cohomologies of $\Fl_i$ are graded finite dimensional. 

If the $X_{i}$ have infinitely many connected components, \eqref{eqn:Kunneth} is usually not an isomorphism, so one must work component-by-component. If we work over a general base $B$, then the relative K\"unneth map 
$$ \Ht^{\sbt}(X_1,\Fl_1) \otimes_{\Ht^{\sbt}(B)} \Ht^{\sbt}(X_2,\Fl_2) \ \to \ \Ht^{\sbt}(X_1\times_B X_2,\Fl_1 \boxtimes \Fl_2) $$
is often not an isomorphism; \cite{LZ} only supplies an isomorphism of sheaves on $B$.

\subsubsection{Cup products} The \textit{cup product} map on two sheaves is the map 
$$\cup \ : \ \Ht^{\sbt}(X,\Fl) \otimes \Ht^{\sbt}(X,\Gl)\ \to \ \Ht^{\sbt}(X\times X,\Fl\boxtimes \Gl) \ \to \ \Ht^{\sbt}(X,\Fl\otimes \Gl)$$
composing the K\"unneth map with pullback along the diagonal $\Fl \boxtimes \Gl \to \Delta_*\Delta^*(\Fl \boxtimes \Gl) = \Fl\otimes \Gl$. We often denote it by $\cdot$ . If $f:Y \to X$ is a map of Artin stacks, it is easy to show that 
\begin{equation}
  \label{eqn:CupPullback} 
  f^*(\alpha \cup \beta) \ = \ f^*(\alpha)\cup f^*(\beta)
\end{equation}
for any cohomology classes $\alpha,\beta$ of $\Fl,\Gl$.

Since $\Qb_X$ is the unit for $\otimes$, it follows that $\Ht^{\sbt}(X)$ is a supercommutative algebra and acts on the cohomology of every sheaf. If $\Fl\to \Gl$ is any map of sheaves, the induced map
  \begin{equation} \label{linearity_lemma} 
      \Ht^{\sbt}(X,\Fl) \ \to\ \Ht^{\sbt}(X,\Gl)
  \end{equation}
  is $\Ht^{\sbt}(X)$-linear, and

\begin{lem} \label{sublemma_cuplinearity}
    Let $f : X \to Y$ be a map of stacks. Then we have an isomorphism of $\Ht^{\sbt}(Y)$-modules
    \begin{equation}
        \Ht^{\sbt}(X, \Fl) \simeq \Ht^{\sbt}(Y, f_{*} \Fl)
    \end{equation}
    where the action on the left is via $\Ht^{\sbt}(Y)\to \Ht^{\sbt}(X)$.
\end{lem}
\begin{proof}
Note that the sheaf theoretic projection formula gives a natural isomorphism 
  \begin{equation}
      - \otimes f_{*} \Gl \ \stackrel{\sim}{\to} \  f_{*}(f^{*}- \otimes \Gl)
  \end{equation}
which we apply to a morphism $\alpha  : \Qb_{Y} \to \Qb_{Y}$ to get a commutative diagram
\begin{equation*}
\begin{tikzcd}[row sep = {30pt,between origins}, column sep = {20pt}]
	{\Qb_{Y} \otimes f_{*} \Gl} & {f_{*}(f^{*} \Qb_{Y} \otimes \Gl)} & {f_{*}(\Qb_{X} \otimes \Gl)} \\
	{\Qb_{Y} \otimes f_{*} \Gl} & {f_{*}(f^{*} \Qb_{Y} \otimes \Gl)} & {f_{*}( \Qb_{X} \otimes \Gl)}
	\arrow[from=1-1, to=1-2]
	\arrow["{\alpha \otimes \id}"', from=1-1, to=2-1]
	\arrow["\sim", from=1-2, to=1-3]
	\arrow["{f^{*} \alpha \otimes \id}", from=1-2, to=2-2]
	\arrow["{f^{*}\alpha \otimes \id}", from=1-3, to=2-3]
	\arrow[from=2-1, to=2-2]
	\arrow["\sim", from=2-2, to=2-3]
\end{tikzcd}
\end{equation*}
from which it follows that there is an isomorphism $\Ht^{\sbt}(Y,f_{*}\Gl)  =  \Ht^{\sbt}(X,\Gl)$ of $\Ht^{\sbt}(Y)$-modules. 
\end{proof}


   

\subsubsection{Vanishing cycles} If we have a function $f:X\to \Ab^1$, we may define the \textit{vanishing cycles} functor 
$$\varphi_f \ : \ \Dt(X) \ \to \ \Dt(X).$$
It satisfies the following properties:
\begin{enumerate}
       \item (Thom--Sebastiani). If $X_i$ are smooth with functions $f_i$, then there is an isomorphism
       \begin{equation} \label{crit_TS_eq}
           \textup{TS} \ :\ \Ht^{\sbt}(X_1, \varphi_{f_1} \Fl_1) \otimes \Ht^{\sbt}(X_2, \varphi_{f_2} \Fl_2) \ \simeq \ \Ht^{\sbt}(X_1 \times X_2, \varphi_{f_1 \boxplus f_2} \Fl_1 \boxtimes \Fl_2)
       \end{equation}
       \item (Push and pull). If $g:X \to Y$ is a map of smooth stacks and $f$ a function on $Y$, then there are natural transformations 
       \begin{align*}
          \varphi_f & \ \to \ g_* \varphi_{g^*f}g^*\\
          g_!\varphi_{fg}g^!\Db & \ \to \ \varphi_f \Db
       \end{align*}
       which induce maps on cohomology (for the second, assuming that $g$ is proper):
       \begin{align}\label{eqn:CritPullback}
        g^{*} \ : \ \Ht^{\sbt}(Y,\varphi_{f} \Fl) & \ \to \  \Ht^{\sbt}(X, \varphi_{g^*f} \Fl).\\
         g_* \ : \ \Ht^{\sbt}(X, \varphi_{g^*f} \Qb_X) & \ \to \  \Ht^{\sbt}(Y, \varphi_{f} \Qb_Y)[\delta] 
       \end{align}
       where $\delta=2\dim Y - 2 \dim X$ and the pushforward map is the composition
       \begin{equation}\label{eqn:PushforwardDefinition}
          \Ht^{\sbt}(X, \varphi_{gf} \Qb_X) \ = \ \Ht^{\sbt}(Y, g_*\varphi_{gf}g^* \Qb_Y) \ = \ \Ht^{\sbt}(Y,g_!\varphi_{gf}g^! \Qb_Y)[\delta] \ \to \ \Ht^{\sbt}(Y, \varphi_{f} \Qb_X)[\delta].
       \end{equation}
       Here the second equivalence uses properness of $g$ to deduce $g_!=g_*$ and smoothness of the spaces to get that $g^!\Qb_Y[\delta] = g^*\Qb_Y$.
       
    \item (Projection formula) If $g$ as above is proper, then we have 
    \begin{equation}
      \label{eqn:ProjectionFormulaVanishing}  
      g_*(g^*(\alpha)\cup \beta) \ = \ \alpha \cup g_*(\beta)
    \end{equation}
    for any cohomology class $\alpha \in \Ht^{\sbt}(Y)$ and critical cohomology class $\beta \in \Ht^{\sbt}(X,\varphi_{fg}\Qb_X)$.
   \end{enumerate}

To show the projection formula, we note that the first equality in \eqref{eqn:PushforwardDefinition} intertwines the actions of $f^*(\alpha)$ and $\alpha$ by Lemma \ref{sublemma_cuplinearity}. The second and third map are induced by maps of sheaves 
$$g_* \varphi_{gf}g^*\Qb_Y \ \simeq \ g_! \varphi_{gf}g^!\Qb_Y[\delta ] \ \to \ \varphi_f \Qb_X[\delta]$$
which intertwines the action of $\alpha$ on both sides by the linearity of the map \eqref{linearity_lemma} over action of cohomology by cup product.

\subsection{Chern classes for perfect complexes}

\label{sec:ChernClasses}
\subsubsection{Tautological bundles}

We have the stacks 
$$\BGL \ = \ \bigsqcup_{n \ge 0}\BGL_n , \hspace{15mm} \Perf \ = \ \bigsqcup_{n\in \Zb} \Perf_n$$
whose $S$-points are rank $n$ vector bundles and rank $n$ perfect complexes on $S$ respectively. In particular, both carry a tautological vector bundle $\El$ and perfect complex $\Fl$ respectively, which have rank $n$ on the $n$th component.
\subsubsection{} 
Define a locally constant function $\rank$ on $\BGL$ or $\Perf$, which on a given component $\BGL_n$ or $\Perf_n$ is the constant function $n$. The \textit{rank} $\rank(V)$ of a vector bundle or perfect complex $V$ is the locally constant function on $X$ given by pulling back $\rank$ via the classifying map $X \to \BGL$ or $X \to \Perf$.

\subsubsection{} 
We denote the tautological line bundle on $\BGm = \BGL_1$ by $\gamma$. We define the tensor product and direct sum maps 
\begin{align*}
   \otimes \ : \ \BGL \times \BGL  \ \to \ \BGL , \hspace{15mm} \otimes \ : \ \Perf \times \Perf  \ \to \ \Perf \\
  \oplus \ : \ \BGL \times \BGL  \ \to \ \BGL , \hspace{15mm} \oplus \ : \ \Perf \times \Perf  \ \to \ \Perf
\end{align*}
by defining $\otimes^*\El = \El \boxtimes \El$ and $\oplus^* \El = \El \boxplus \El$ and similarly for $\Fl$. In particular, we have maps 
\begin{align*}
   \act \ : \ \BGm \times \BGL  \ \to \ \BGL , \hspace{15mm} \act \ : \ \BGm \times \Perf  \ \to \ \Perf
\end{align*}
defined by $\act^* \El = \gamma \boxtimes \El$ and similarly for $\Fl$. We also have the maps 
$$ \vee \ : \ \BGL \ \to \ \BGL, \hspace{15mm} \vee \ : \ \Perf \ \to \ \Perf $$
defined by $\vee^* \El = \El^\vee$ and similarly for $\Fl$.

\subsubsection{Chern classes} We have isomorphisms 
$$\Ht^{\sbt}(\BGL_n) \ \simeq \ \Qb[c_1, \ldots ,c_n], \hspace{15mm} \Ht^{\sbt}(\Perf_n) \ \simeq \ \Qb[c_1,c_2, \ldots  \ ]$$
where $c_i$ have cohomological degree $2i$ and satisfy
\begin{equation}
  \label{eqn:WhitneySum} 
  \oplus^*c_n \ = \ c_n \otimes 1 + c_{n-1}\otimes c_1 + \cdots + 1 \otimes c_n.
\end{equation}
It was shown by Grothendieck in the vector bundle case that this property characterises the $c_i$ uniquely, up to multiplying $c_1$ by a nonzero scalar. 

For any vector bundle or perfect complex $V$, we define the \textit{chern classes} $c_i(V)$ by pullback of the generators $c_i$ along the classifying map $X \to \BGL$ or $X \to \Perf$, and the \textit{chern series} by 
$$c(V,z) \ = \ \sum_{r \ge 0} c_r(V,z)z^r$$
where $c_0=1$. Moreover,

\begin{prop}\label{chern_series_prop}
  The chern series satisfies 
  \begin{enumerate}
  \item $f^*c(V,z) = c(f^*V,z)$ for any map $f:X \to Y$ and vector bundle or perfect complex $V$ on $Y$.
  \item $c(V_1\oplus V_2,z)\ =\ c(V_1,z)\cdot c(V_2,z)$ for vector bundles or perfect complexes $V_1,V_2$.
   \item $c(V\otimes \Ll,z^{-1}) = \left(\frac{z+c_1(\Ll)}{z}\right)^{\rank \Vl}c(V,(z+c_1(\Ll))^{-1})$ for line bundle $\Ll$, i.e. the left hand side is the expansion as a power series in $z^{-1}$ of the right side.
   \item $c(V^\vee,z) = c(V,-z)$.
  \end{enumerate}
\end{prop}

Part (3) can be proven by expanding the chern series as an exponential of chern characters, see \cite[Lem. 2.6.19]{al_thesis}.

\subsubsection{Euler classes} \label{ssec:two_term_Euler}
If $V$ is a vector bundle of rank $n$, its \textit{Euler class} is its top chern class $e(V) = c_n(V)$. Likewise, given a perfect complex given by a complex of vector bundles
$$V \ = \ \left( V_{-n}\to V_{-n+1}\to \cdots \to V_{m-1}\to V_m\right)$$
we define its \textit{Euler class} as
$$e(V) \ = \ e(V_n)^\pm \cdots \frac{1}{e(V_{-1})}\cdot e(V_0)\cdot \frac{1}{e(V_1)}\cdots e(V_m)^\pm \ \in \ \Ht^{\sbt}(X)[e(V_i)^{-1} \ : \ i \textup{ odd}].$$

\subsubsection{The generating series $\Psi$} We define for any perfect complex or vector bundle $V$ the power series 
\begin{equation}
  \Psi(V,z) =  z^{\rank V}c(V,z^{-1})
\end{equation}
and from the properties of the chern series in Proposition \ref{chern_series_prop} it follows that
\begin{lem} \label{lem:psi_properties}
   We have that
\begin{enumerate}
  \item $f^*\Psi(V,z) = \Psi(f^*V,z)$ for any map $f:X \to Y$ and vector bundle or perfect complex $V$ on $Y$.
\item $\Psi(V_1 \oplus V_2,z) = \Psi(V_1,z) \cdot \Psi(V_2,z)$.
\item $\Psi(V\otimes \Ll,z) = \Psi(V,z+c_1(\Ll))$ for line bundle $\Ll$, i.e. the left hand side is the expansion as a Laurent series in $z$ of the right side. 
\item $\Psi(V^\vee,z) = (-1)^{\rank V} \Psi(V,-z) $.
\end{enumerate}
\end{lem}

\subsubsection{Chern roots} Given a vector bundle $V$ on $X$ induced by $X\to \BGL_n$, its \textit{chern roots} are the images $x_{V,1}, \ldots, x_{V,n} \in \Ht^{\sbt}(X\times_{\BGL_n}\BT)$ of the generators of $\Ht^2(\BT)$ 
\begin{center}
\begin{tikzcd}[row sep = {30pt,between origins}, column sep = {20pt}]
 X\times_{\BGL_n}\BT_n \ar[d]\ar[r]& \BT_n \ar[d] \\ 
 X \ar[r,"V"]& \BGL_n
\end{tikzcd}
\end{center}
where $T_n$ is the maximal torus of diagonal elements inside $\GL_n$. In the setting in the body of the paper, $X=\Ml_{Q,d}^T$ and the pullback $\Ht^{\sbt}(X) \hookrightarrow \Ht^{\sbt}(X\times_{\BGL_n}\BT_n)$ is an injection, so for computations in $X$ it suffices to compute in terms of chern roots, for instance we have
$$c(V,z) \ = \ (1+zx_{V,1})\cdots (1+zx_{V,n})$$
and therefore
$$\Psi(V,z)\ = \ (z+x_{V,1})\cdots (z +x_{V,n}), \hspace{15mm} e(V) \ = \ x_{V,1}\cdots x_{V,n}.$$

\begin{prop}\label{prop:PsiEuler}
For $V$ a vector bundle or bounded complex of vector bundles on $Y$ and $n\ne 0$, we have that 
$$\Psi(V,nz) \ = \ e(\gamma^n\boxtimes V )$$
where $\gamma$ is the tautological line bundle on $\BGm$, and we expand the right hand side--a priori a localised cohomology class in $\Ht^{\sbt}(\BGm\times Y)$--as a Laurent series in $z^{-1}=c_1(\gamma)^{-1}$.
\end{prop}
\begin{proof}
 By multiplicativity of Euler classes and $\Psi(-,z)$ over direct sums of perfect complexes by the Lemma \ref{lem:psi_properties}, it suffices to prove this for each factor of $V$, i.e. assume that $V$ is a vector bundle. The chern roots of $\gamma^{n}\boxtimes V$  are $x_i + c_1(\gamma^{ n})=x_i+nz$ where $x_i$ are the chern roots of $V$, and so we have 
 $$e(\gamma^{n}\boxtimes V) \ = \ \bigsqcap_i (x_i+nz)  \ = \ \Psi(V,nz),$$
 proving the Proposition.
\end{proof}

\subsubsection{Example} 
For instance, the chern series, Euler class and $\Psi$ series of a two term perfect complex $V=(V_0\to V_1)$ is 
\begin{equation}
    c(V,z) = \frac{c(V_{0},z)}{c(V_{1},z)}, \hspace{15mm} e(V) \ =\ \frac{e(V_0)}{e(V_1)}, \hspace{15mm} \Psi(V,z) \ = \ \frac{\Psi(V_0,z)}{\Psi(V_1,z)}.
\end{equation}
If the chern roots of $V_0$ and $V_1$ are denoted by $x_i$ and $y_j$, these are 
\begin{equation}
    c(V,z) = \frac{\bigsqcap(1+zx_i)}{\bigsqcap(1+zy_j)}, \hspace{15mm} e(V) \ =\ \frac{\bigsqcap x_i}{\bigsqcap y_j}, \hspace{15mm} \Psi(V,z) \ = \ \frac{\bigsqcap(z+x_i)}{\bigsqcap(z+y_j)}.
\end{equation}

\printbibliography    

@misc{diaconescu2026cohomologicalhallalgebrasonedimensional,
      title={Cohomological Hall algebras of one-dimensional sheaves on surfaces and Yangians}, 
      author={Duiliu-Emanuel Diaconescu and Mauro Porta and Francesco Sala and Olivier Schiffmann and Eric Vasserot},
      year={2026},
      eprint={2603.03386},
      archivePrefix={arXiv},
      primaryClass={math.AG},
      url={https://arxiv.org/abs/2603.03386}, 
}

@misc{jindal2026cohacyclicquiversintegral,
      title={CoHA of Cyclic Quivers and an Integral Form of Affine Yangians}, 
      author={Shivang Jindal},
      year={2026},
      eprint={2408.02618},
      archivePrefix={arXiv},
      primaryClass={math.RT},
      url={https://arxiv.org/abs/2408.02618}, 
}

@incollection {dimensionalreductionalgebras,
    AUTHOR = {Ren, Jie and Soibelman, Yan},
     TITLE = {Cohomological {H}all algebras, semicanonical bases and
              {D}onaldson-{T}homas invariants for 2-dimensional
              {C}alabi-{Y}au categories (with an appendix by {B}en
              {D}avison)},
 BOOKTITLE = {Algebra, geometry, and physics in the 21st century},
    SERIES = {Progr. Math.},
    VOLUME = {324},
     PAGES = {261--293},
 PUBLISHER = {Birkh\"auser/Springer, Cham},
      YEAR = {2017},
      ISBN = {978-3-319-59938-0; 978-3-319-59939-7},
   MRCLASS = {16G20 (14F42 14N35 18F99 81R99)},
  MRNUMBER = {3727563},
MRREVIEWER = {Xueqing\ Chen},
       DOI = {10.1007/978-3-319-59939-7\_7},
       URL = {https://doi.org/10.1007/978-3-319-59939-7_7},
}

@article {yangzhaodimension,
    AUTHOR = {Yang, Yaping and Zhao, Gufang},
     TITLE = {On two cohomological {H}all algebras},
   JOURNAL = {Proc. Roy. Soc. Edinburgh Sect. A},
  FJOURNAL = {Proceedings of the Royal Society of Edinburgh. Section A.
              Mathematics},
    VOLUME = {150},
      YEAR = {2020},
    NUMBER = {3},
     PAGES = {1581--1607},
      ISSN = {0308-2105,1473-7124},
   MRCLASS = {14N35 (14F43 16G20 17B37)},
  MRNUMBER = {4091073},
MRREVIEWER = {Xiaobin\ Li},
       DOI = {10.1017/prm.2018.162},
       URL = {https://doi.org/10.1017/prm.2018.162},
}

@article {KS,
    AUTHOR = {Kontsevich, Maxim and Soibelman, Yan},
     TITLE = {Cohomological {H}all algebra, exponential {H}odge structures
              and motivic {D}onaldson-{T}homas invariants},
   JOURNAL = {Commun. Number Theory Phys.},
  FJOURNAL = {Communications in Number Theory and Physics},
    VOLUME = {5},
      YEAR = {2011},
    NUMBER = {2},
     PAGES = {231--352},
      ISSN = {1931-4523,1931-4531},
   MRCLASS = {14N35 (14F43 16G20)},
  MRNUMBER = {2851153},
MRREVIEWER = {Mark\ Gross},
       DOI = {10.4310/CNTP.2011.v5.n2.a1},
       URL = {https://doi.org/10.4310/CNTP.2011.v5.n2.a1},
}

@article{COZZ,
  title={Shifted quantum groups via critical stable envelopes},
  author={Cao, Yalong and Okounkov, Andrei and Zhou, Yehao and Zhou, Zijun},
  journal={arXiv preprint arXiv:2601.01518},
  year={2026},
  url={https://arxiv.org/abs/2601.01518}
}

@book{beilinson_drinfeld,
  title={Chiral algebras},
  author={Beilinson, Alexander and Drinfeld, Vladimir},
  volume={51},
  year={2025},
  publisher={American Mathematical Society}
}

@article{Bo,
  title={Vertex algebras, Kac-Moody algebras, and the Monster},
  author={Borcherds, Richard E},
  journal={Proceedings of the National Academy of Sciences},
  volume={83},
  number={10},
  pages={3068--3071},
  year={1986},
  url={https://math.berkeley.edu/~reb/papers/va/va.pdf}
}

@article{Ri,
  title={Hall algebras and quantum groups},
  author={Ringel, Claus Michael},
  journal={Inventiones mathematicae},
  volume={101},
  number={1},
  year={1990}
}

@article{Gr,
  title={Hall algebras, hereditary algebras and quantum groups},
  author={Green, James A},
  journal={Inventiones mathematicae},
  volume={120},
  number={1},
  pages={361--377},
  year={1995},
  publisher={Springer}
}

@article{DK,
  title={Weyl group extension of quantized current algebras},
  author={Ding, Jintai and Khoroshkin, Sergei},
  journal={Transformation groups},
  volume={5},
  pages={35--59},
  year={2000},
  publisher={Springer},
  url={https://arxiv.org/abs/math/9804139}
}

@inproceedings{TV,
  title={Moduli of objects in dg-categories},
  author={To{\"e}n, Bertrand and Vaqui{\'e}, Michel},
  booktitle={Annales scientifiques de l'Ecole normale sup{\'e}rieure},
  volume={40},
  number={3},
  pages={387--444},
  year={2007},
  url={https://arxiv.org/abs/math/0503269}
}

@article{gautam2017meromorphic,
  title={Meromorphic tensor equivalence for Yangians and quantum loop algebras},
  author={Gautam, Sachin and Toledano Laredo, Valerio},
  journal={Publications mathematiques de l'IH{\'E}S},
  volume={125},
  number={1},
  pages={267--337},
  year={2017},
  publisher={Springer},
  url={https://arxiv.org/abs/1403.5251}
}

@article{Ra,
  title={The structure of Hopf algebras with a projection},
  author={Radford, David E},
  journal={Journal of Algebra},
  volume={92},
  number={2},
  pages={322--347},
  year={1985},
  publisher={Elsevier}
}

@article{MaBos,
  title={Cross products by braided groups and bosonization},
  author={Majid, Shahn},
  journal={Journal of algebra},
  volume={163},
  number={1},
  pages={165--190},
  year={1994},
  publisher={Elsevier}
}

@inproceedings{Ma2,
  title={Double-bosonization of braided groups and the construction of $U_q(\mathfrak{g})$},
  author={Majid, Shahn},
  booktitle={Mathematical Proceedings of the Cambridge Philosophical Society},
  volume={125},
  number={1},
  pages={151--192},
  year={1999},
  organization={Cambridge University Press}
}

@article {MO,
    AUTHOR = {Maulik, Davesh and Okounkov, Andrei},
     TITLE = {Quantum groups and quantum cohomology},
   JOURNAL = {Ast\'erisque},
  FJOURNAL = {Ast\'erisque},
    NUMBER = {408},
      YEAR = {2019},
     PAGES = {ix+209},
      ISSN = {0303-1179,2492-5926},
      ISBN = {978-2-85629-900-5},
   MRCLASS = {14D20 (14D21 16G20 17B37)},
  MRNUMBER = {3951025},
MRREVIEWER = {Leonardo\ Constantin\ Mihalcea},
       DOI = {10.24033/ast},
       URL = {https://doi.org/10.24033/ast},
}

@book{ES,
  title={Lectures on Quantum Groups},
  author={Etingof, P.I. and Schiffmann, O.},
  isbn={9781571460943},
  lccn={95025393},
  series={Lectures in mathematical physics},
  url={https://books.google.com.cu/books?id=Ys9UAAAAYAAJ},
  year={2002},
  publisher={International Press}
}

@article{BD,
  title={Okounkov's conjecture via BPS Lie algebras},
  author={Botta, Tommaso Maria and Davison, Ben},
  journal={arXiv preprint arXiv:2312.14008},
  year={2023},
  url={https://arxiv.org/abs/2312.14008}
}

@phdthesis{Bu,
  title={Equivariant localization in factorization homology and vertex algebras from supersymmetric gauge theory},
  author={Butson, Dylan William},
  year={2021},
  school={University of Toronto (Canada)},
  url={https://arxiv.org/abs/2011.14988}
}

@article{Da1,
  title={The integrality conjecture and the cohomology of preprojective stacks},
  author={Davison, Ben},
  journal={Journal f{\"u}r die reine und angewandte Mathematik (Crelles Journal)},
  volume={2023},
  number={804},
  pages={105--154},
  year={2023},
  publisher={De Gruyter},
  url={https://arxiv.org/abs/1602.02110}
}

@article{Da2,
  title={The critical CoHA of a quiver with potential},
  author={Davison, Ben},
  journal={Quarterly Journal of Mathematics},
  volume={68},
  number={2},
  pages={635--703},
  year={2017},
  publisher={OUP},
  url={https://arxiv.org/abs/1311.7172}
}

@misc{schiffmann2022cohomologicalhallalgebrasquivers,
      title={On cohomological Hall algebras of quivers : generators}, 
      author={Olivier Schiffmann and Eric Vasserot},
      year={2022},
      eprint={1705.07488},
      archivePrefix={arXiv},
      primaryClass={math.RT},
      url={https://arxiv.org/abs/1705.07488}, 
}

@Article{sch_vas_cheredenik,
 Author = {Schiffmann, O. and Vasserot, E.},
 Title = {Cherednik algebras, {{\(W\)}}-algebras and the equivariant cohomology of the moduli space of instantons on {{\(\mathbb A^2\)}}},
 FJournal = {Publications Math{\'e}matiques},
 Journal = {Publ. Math., Inst. Hautes {\'E}tud. Sci.},
 ISSN = {0073-8301},
 Volume = {118},
 Pages = {213--342},
 Year = {2013},
 Language = {English},
 DOI = {10.1007/s10240-013-0052-3},
 zbMATH = {6233894},
 Zbl = {1284.14008}
}

@article{GL,
  title={Metaplectic Whittaker category and quantum groups: the" small" FLE},
  author={Gaitsgory, Dennis and Lysenko, Sergey},
  journal={arXiv preprint arXiv:1903.02279},
  year={2019},
  url={https://arxiv.org/abs/1903.02279}
}

@book{Lu,
  title={Introduction to quantum groups},
  author={Lusztig, George},
  year={2010},
  publisher={Springer Science $\&$ Business Media},
  url={https://math.mit.edu/~gyuri/papers/quant.gr.pdf}
}

@article{Ga,
  title={On factorization algebras arising in the quantum geometric Langlands theory},
  author={Gaitsgory, Dennis},
  journal={Advances in Mathematics},
  volume={391},
  pages={107962},
  year={2021},
  publisher={Elsevier},
  url={https://arxiv.org/abs/1909.09775}
}

@incollection{GTW,
  title={The meromorphic R-matrix of the Yangian},
  author={Gautam, Sachin and Laredo, Valerio Toledano and Wendlandt, Curtis},
  booktitle={Representation Theory, Mathematical Physics, and Integrable Systems: In Honor of Nicolai Reshetikhin},
  pages={201--269},
  year={2021},
  publisher={Springer},
  url={https://arxiv.org/abs/1403.5251}
}

@article{Hu,
  title={Vertex coalgebras, comodules, cocommutativity and coassociativity},
  author={Hubbard, Keith},
  journal={Journal of Pure and Applied Algebra},
  volume={213},
  number={1},
  pages={109--126},
  year={2009},
  publisher={Elsevier},
  url={https://arxiv.org/abs/0801.3260}
}

@article{Jo,
  title={Ringel--Hall style vertex algebra and Lie algebra structures on the homology of moduli spaces},
  author={Joyce, Dominic},
  journal={Preliminary version},
  year={2018},
  url={https://people.maths.ox.ac.uk/joyce/hall.pdf}
}

@misc{Sch,
  author = {Schiffmann, Olivier},
  title = {Lectures on {H}all algebras},
  year = {2006},
  eprint = {math/0611617},
  archiveprefix = {arXiv},
  url = {https://arxiv.org/abs/math/0611617}
}

@article{La,
  title={Factorisation quantum groups},
  author={Latyntsev, Alexei},
  journal={arXiv preprint arXiv:2312.07274},
  year={2023},
  url={https://arxiv.org/abs/2312.07274}
}

@article{YZ,
  title={Cohomological Hall algebras and affine quantum groups},
  author={Yang, Yaping and Zhao, Gufang},
  journal={Selecta Mathematica},
  volume={24},
  number={2},
  pages={1093--1119},
  year={2018},
  publisher={Springer},
  url={https://arxiv.org/abs/1604.01865}
}

@article {RSYZ,
    AUTHOR = {Rap\v c\'ak, Miroslav and Soibelman, Yan and Yang, Yaping and
              Zhao, Gufang},
     TITLE = {Cohomological {H}all algebras and perverse coherent sheaves on
              toric {C}alabi-{Y}au 3-folds},
   JOURNAL = {Commun. Number Theory Phys.},
  FJOURNAL = {Communications in Number Theory and Physics},
    VOLUME = {17},
      YEAR = {2023},
    NUMBER = {4},
     PAGES = {847--939},
      ISSN = {1931-4523,1931-4531},
   MRCLASS = {14N35 (14J32 14M25 17B67 81T30)},
  MRNUMBER = {4704941},
MRREVIEWER = {Dylan\ Spence},
       DOI = {10.4310/cntp.2023.v17.n4.a2},
       URL = {https://doi.org/10.4310/cntp.2023.v17.n4.a2},
}

@article{EK,
 Author = {Etingof, Pavel and Kazhdan, David},
 Title = {Quantization of {Lie} bialgebras. {V}: {Quantum} vertex operator algebras},
 FJournal = {Selecta Mathematica. New Series},
 Journal = {Sel. Math., New Ser.},
 ISSN = {1022-1824},
 Volume = {6},
 Number = {1},
 Pages = {105--130},
 Year = {2000},
 Language = {English},
 DOI = {10.1007/s000290050004},
 zbMATH = {1486242},
 Zbl = {0948.17008},
 url={https://arxiv.org/abs/q-alg/9706023}
}

@article{FR,
  title={Towards deformed chiral algebras},
  author={Frenkel, Edward and Reshetikhin, Nikolai},
  journal={arXiv preprint q-alg/9706023},
  year={1997},
  url={https://arxiv.org/abs/q-alg/9706023}
}

@book{FBZ,
  title={Vertex algebras and algebraic curves},
  author={Frenkel, Edward and Ben-Zvi, David},
  number={88},
  year={2004},
  publisher={American Mathematical Soc.}
}

@article {dotsmozgva,
    AUTHOR = {Dotsenko, Vladimir and Mozgovoy, Sergey},
     TITLE = {D{T} invariants from vertex algebras},
   JOURNAL = {J. Inst. Math. Jussieu},
  FJOURNAL = {Journal of the Institute of Mathematics of Jussieu. JIMJ.
              Journal de l'Institut de Math\'ematiques de Jussieu},
    VOLUME = {24},
      YEAR = {2025},
    NUMBER = {1},
     PAGES = {291--339},
      ISSN = {1474-7480,1475-3030},
   MRCLASS = {17B69 (14N35 16G20)},
  MRNUMBER = {4847121},
       DOI = {10.1017/S1474748024000288},
       URL = {https://doi.org/10.1017/S1474748024000288},
}

@article {han2021cocommva,
    AUTHOR = {Han, Jianzhi and Li, Haisheng and Xiao, Yukun},
     TITLE = {Cocommutative vertex bialgebras},
   JOURNAL = {J. Algebra},
  FJOURNAL = {Journal of Algebra},
    VOLUME = {598},
      YEAR = {2022},
     PAGES = {536--569},
      ISSN = {0021-8693,1090-266X},
   MRCLASS = {17B69 (16T10)},
  MRNUMBER = {4386635},
MRREVIEWER = {Ching\ Hung\ Lam},
       DOI = {10.1016/j.jalgebra.2022.02.003},
       URL = {https://doi.org/10.1016/j.jalgebra.2022.02.003},
}

@book{loday2012algebraic,
  title={Algebraic operads},
  author={Loday, Jean-Louis and Vallette, Bruno},
  volume={346},
  year={2012},
  publisher={Springer Science and Business Media}
}

@article{yang2018cohomological,
  title={The cohomological Hall algebra of a preprojective algebra},
  author={Yang, Yaping and Zhao, Gufang},
  journal={Proceedings of the London Mathematical Society},
  volume={116},
  number={5},
  pages={1029--1074},
  year={2018},
  publisher={Wiley Online Library},
  url={https://arxiv.org/abs/1407.7994}
}

@article{DN,
  title={Tannakian QFT: from spark algebras to quantum groups},
  author={Dimofte, Tudor and Niu, Wenjun},
  journal={arXiv preprint arXiv:2411.04194},
  year={2024},
  url={https://arxiv.org/abs/2411.04194}
}

@article{DNP,
  title={Line operators in 3d holomorphic qft: meromorphic tensor categories and dg-shifted Yangians},
  author={Dimofte, Tudor and Niu, Wenjun and Py, Victor},
  journal={arXiv preprint arXiv:2508.11749},
  year={2025},
  url={https://arxiv.org/abs/2508.11749}
}

@article{CWY,
  title={Gauge theory and integrability, I},
  author={Costello, Kevin and Witten, Edward and Yamazaki, Masahito},
  journal={arXiv preprint arXiv:1709.09993},
  year={2017},
  url={https://arxiv.org/abs/1709.09993}
}

@article{elliot_safronov,
  title={Topological twists of supersymmetric algebras of observables},
  author={Elliott, Chris and Safronov, Pavel},
  journal={Communications in Mathematical Physics},
  volume={371},
  number={2},
  pages={727--786},
  year={2019},
  publisher={Springer},
  url={https://arxiv.org/abs/1805.10806}
}

@article{LZ,
  title={Enhanced six operations and base change theorem for higher Artin stacks},
  author={Liu, Yifeng and Zheng, Weizhe},
  journal={arXiv preprint arXiv:1211.5948},
  year={2012},
  url={https://arxiv.org/abs/1211.5948}
}

@article {liu2022multiplicative,
    AUTHOR = {Liu, Henry},
     TITLE = {Multiplicative vertex algebras and quantum loop algebras},
   JOURNAL = {J. Lond. Math. Soc. (2)},
  FJOURNAL = {Journal of the London Mathematical Society. Second Series},
    VOLUME = {112},
      YEAR = {2025},
    NUMBER = {2},
     PAGES = {Paper No. e70270, 48},
      ISSN = {0024-6107,1469-7750},
   MRCLASS = {17B69 (14C35 14D23 17B37)},
  MRNUMBER = {4947743},
       DOI = {10.1112/jlms.70270},
       URL = {https://doi.org/10.1112/jlms.70270},
}

@article{kinjo2024coha,
  title={Cohomological Hall algebras for 3-Calabi-Yau categories},
  author={Kinjo, Tasuki and Park, Hyeonjun and Safronov, Pavel},
  journal={arXiv preprint arXiv:2406.12838},
  year={2024}
}

@article{efimov2012cohomological,
  title={Cohomological Hall algebra of a symmetric quiver},
  author={Efimov, Alexander I},
  journal={Compositio Mathematica},
  volume={148},
  number={4},
  pages={1133--1146},
  year={2012},
  publisher={London Mathematical Society},
  url={https://arxiv.org/abs/1103.2736}
}

@misc{al_thesis,
      title={Vertex algebras, moduli stacks, cohomological Hall algebras and quantum groups}, 
      author={Alexei Latyntsev},
journal={PhD Thesis},
      year={2022},
      url={https://people.maths.ox.ac.uk/joyce/theses/LatyntsevDPhil.pdf}, 
}

@misc{davison2022affine,
      title={Affine BPS algebras, W algebras, and the cohomological Hall algebra of $\mathbb{A}^2$}, 
      author={Ben Davison},
      year={2025},
      eprint={2209.05971},
      archivePrefix={arXiv},
      primaryClass={math.RT},
      url={https://arxiv.org/abs/2209.05971}, 
}

@misc{schiffmann2024s,
      title={Cohomological Hall algebras of quivers and Yangians}, 
      author={Olivier Schiffmann and Eric Vasserot},
      year={2024},
      eprint={2312.15803},
      archivePrefix={arXiv},
      primaryClass={math.RT},
      url={https://arxiv.org/abs/2312.15803}, 
}

@article{latyntsev2021cohomological,
  title={Cohomological Hall algebras and vertex algebras},
  author={Latyntsev, Alexei},
  journal={arXiv preprint arXiv:2110.14356},
  year={2021},
  url={https://arxiv.org/abs/2110.14356}
}

@inproceedings{Dr2,
  title={A new realization of Yangians and of quantum affine algebras},
  author={Drinfeld, Vladimir G},
  booktitle={Dokl. Akad. Nauk SSSR},
  volume={296},
  pages={13--17},
  year={1987}
}

@article{Soibelman,
  title={Meromorphic tensor categories},
  author={Soibelman, Yan},
  journal={arXiv preprint q-alg/9709030},
  year={1997},
  url={https://arxiv.org/abs/q-alg/9709030}
}

@article {lim_bojko_moreira,
    AUTHOR = {Bojko, Arkadij and Lim, Woonam and Moreira, Miguel},
     TITLE = {Virasoro constraints for moduli of sheaves and vertex
              algebras},
   JOURNAL = {Invent. Math.},
  FJOURNAL = {Inventiones Mathematicae},
    VOLUME = {236},
      YEAR = {2024},
    NUMBER = {1},
     PAGES = {387--476},
      ISSN = {0020-9910,1432-1297},
   MRCLASS = {14F08 (14H60 14J60 17B68 17B69)},
  MRNUMBER = {4712868},
MRREVIEWER = {Hsian-Hua\ Tseng},
       DOI = {10.1007/s00222-024-01245-5},
       URL = {https://doi.org/10.1007/s00222-024-01245-5},
}

@article {bu_counting_sheaves,
    AUTHOR = {Bu, Chenjing},
     TITLE = {Counting sheaves on curves},
   JOURNAL = {Adv. Math.},
  FJOURNAL = {Advances in Mathematics},
    VOLUME = {434},
      YEAR = {2023},
     PAGES = {Paper No. 109334, 87},
      ISSN = {0001-8708,1090-2082},
   MRCLASS = {14N35 (14D23 14H60 14N10)},
  MRNUMBER = {4650626},
MRREVIEWER = {Hsian-Hua\ Tseng},
       DOI = {10.1016/j.aim.2023.109334},
       URL = {https://doi.org/10.1016/j.aim.2023.109334},
}

@article {gross_joyce_tanaka,
    AUTHOR = {Gross, Jacob and Joyce, Dominic and Tanaka, Yuuji},
     TITLE = {Universal structures in {$\Bbb C$}-linear enumerative
              invariant theories},
   JOURNAL = {SIGMA Symmetry Integrability Geom. Methods Appl.},
  FJOURNAL = {SIGMA. Symmetry, Integrability and Geometry. Methods and
              Applications},
    VOLUME = {18},
      YEAR = {2022},
     PAGES = {Paper No. 068, 61},
      ISSN = {1815-0659},
   MRCLASS = {14D20 (16G20 17B69)},
  MRNUMBER = {4489089},
MRREVIEWER = {Jon\ Eivind\ Vatne},
       DOI = {10.3842/SIGMA.2022.068},
       URL = {https://doi.org/10.3842/SIGMA.2022.068},
}

@misc{joyce2021enumerativeinvariantswallcrossingformulae,
      title={Enumerative invariants and wall-crossing formulae in abelian categories}, 
      author={Dominic Joyce},
      year={2021},
      eprint={2111.04694},
      archivePrefix={arXiv},
      primaryClass={math.AG},
      url={https://arxiv.org/abs/2111.04694}, 
}

@misc{liu2025equivariantktheoreticenumerativeinvariants,
      title={Equivariant K-theoretic enumerative invariants and wall-crossing formulae in abelian categories}, 
      author={Henry Liu},
      year={2025},
      eprint={2207.13546},
      archivePrefix={arXiv},
      primaryClass={math.AG},
      url={https://arxiv.org/abs/2207.13546}, 
}

@article{andruskiewitsch2002pointed,
  title={Pointed hopf algebras},
  author={Andruskiewitsch, Nicol{\'a}s and Schneider, Hans-J{\"u}rgen},
  journal={New directions in Hopf algebras},
  volume={43},
  pages={1--68},
  year={2002}
}

\end{document}